\documentclass[reqno,12pt,letterpaper]{amsart}
\usepackage[proof]{sdmacros}
\usepackage{mathtools}

\title{Semiclassical measures for higher dimensional quantum cat maps}
\author{Semyon Dyatlov}
\email{dyatlov@math.mit.edu}
\address{Department of Mathematics, Massachusetts Institute of Technology, Cambridge, MA 02139}
\author{Malo J\'ez\'equel}
\email{mpjez@mit.edu}
\address{Department of Mathematics, Massachusetts Institute of Technology, Cambridge, MA 02139}

\newcommand{\p}[1]{\left(#1\right)}
\newcommand{\n}[1]{\left\|#1\right\|}
\newcommand{\set}[1]{\left\{#1\right\}}
\newcommand{\va}[1]{\left|#1\right|}
\newcommand{\T}{\mathbb{T}}
\newcommand{\Q}{\mathbb{Q}}
\newcommand{\R}{\mathbb{R}}
\newcommand{\N}{\mathbb{N}}
\newcommand{\Z}{\mathbb{Z}}
\newcommand{\C}{\mathbb{C}}
\newcommand{\w}{\mathtt{w}}

\begin{document}

\begin{abstract}
Consider a quantum cat map $M$ associated to a matrix~$A\in\Sp(2n,\mathbb Z)$, which
is a common toy model in quantum chaos. We show that the mass of eigenfunctions
of~$M$ on any nonempty open set in the position-frequency space satisfies a lower bound
which is uniform
in the semiclassical limit, under two assumptions: (1) there is a unique
simple eigenvalue of $A$ of largest absolute value and (2) the characteristic
polynomial of $A$ is irreducible over the rationals.
This is similar to previous work~\cite{meassupp,varfup} on negatively curved
surfaces and~\cite{Schwartz-cat} on quantum cat maps with $n=1$, but this paper gives
the first results of this type which apply in any dimension. When condition (2) fails we provide a
weaker version of the result and discuss relations to existing counterexamples.
We also obtain corresponding statements regarding semiclassical measures and damped
quantum cat maps. 
\end{abstract}

\maketitle

\addtocounter{section}{1}
\addcontentsline{toc}{section}{1. Introduction}

In~\cite{varfup}, Dyatlov--Jin--Nonnenmacher proved the following lower
bound on $L^2$ mass of eigenfunctions: if $(\mathcal M,g)$ is a compact connected Riemannian surface
with Anosov geodesic flow (e.g. a negatively curved surface) and $u$ is an eigenfunction of the Laplacian
on~$\mathcal M$ with eigenvalue $-\lambda^2$, then
\begin{equation}
  \label{e:varfup-estimate}
\|u\|_{L^2(\mathcal M)}\leq C_\Omega\|u\|_{L^2(\Omega)}\quad\text{for any nonempty open}\quad
\Omega\subset \mathcal M
\end{equation}
where the constant $C_\Omega>0$ depends on $\mathcal M$ and $\Omega$, but does not depend
on~$\lambda$. This result has applications to control for the Schr\"odinger equation,
exponential energy decay for the damped wave equation, and semiclassical
measures, and belongs to the domain of \emph{quantum chaos}~-- see~\cite{varfup} and~\S\ref{s:history} for a historical overview.

The paper~\cite{varfup} only deals with the case of surfaces because the key ingredient,
the \emph{fractal uncertainty principle} of Bourgain--Dyatlov~\cite{fullgap},
is only known for subsets of $\mathbb R$. To prove an analogous result
for manifolds of dimension $n+1>2$ would require a fractal uncertainty principle
for subsets of $\mathbb R^n$. A naive extension of the fractal uncertainty
principle to higher dimensions is false and
no generalization suitable for applications to~\eqref{e:varfup-estimate} is currently known~--
see the review of Dyatlov~\cite[\S6]{FUP-ICMP} and the paper of Han--Schlag~\cite{Han-Schlag}.

In this paper we give a class of higher dimensional examples where a bound of type~\eqref{e:varfup-estimate} can still be shown using the one-dimensional fractal uncertainty principle of~\cite{fullgap}.
We work in the setting of \emph{quantum cat maps}, which are toy models
commonly used in quantum chaos.
In this setting, the geodesic flow on a Riemannian manifold
is replaced by a \emph{classical cat map}, which is the automorphism of the
torus $\mathbb T^{2n}:=\mathbb R^{2n}/\mathbb Z^{2n}$ induced
by an integer symplectic matrix
$A\in\Sp(2n,\mathbb Z)$.
The eigenfunctions of the Laplacian are replaced
by those of a \emph{quantum cat map},
an operator~$M_{\mathbf N,\theta}$ on an $\mathbf N^n$-dimensional space
$\mathcal H_{\mathbf N}(\theta)$ which
quantizes~$A$ in the sense of~\eqref{e:egorov-intro} below.
The high eigenvalue limit $\lambda\to\infty$ is replaced by the limit $\mathbf N\to\infty$.
More general quantum maps have also been used in the study of continuous systems
(such as Laplacians on Riemannian manifolds),
where the quantum map corresponds to Poincar\'e map(s) of the original dynamical system,
see in particular Bogomolny~\cite{Bogomolny-Poincare} for a physics perspective and Sj\"ostrand--Zworski~\cite{Sjostrand-Zworski-monodromy} for an approach relevant to trace formulas.

For two-dimensional quantum cat maps (analogous
to the case of Laplacian eigenfunctions on surfaces) an estimate similar to~\eqref{e:varfup-estimate} was recently proved by Schwartz~\cite{Schwartz-cat}. The novelty of the present paper is that it also applies in higher dimensions.

\subsection{Setting and lower bounds on eigenfunctions}

To explain quantum cat maps in more detail, we use a \emph{semiclassical quantization
procedure}, mapping each classical observable (a function on a symplectic manifold called
the phase space)
to a quantum observable (an operator on some Hilbert space). In our setting
the phase space is the torus
$\mathbb T^{2n}$
and each classical observable is quantized to a family of operators
(see~\S\ref{s:quantization-on-torus}):
$$
a\in C^\infty(\mathbb T^{2n})\quad \mapsto\quad \Op_{\mathbf N,\theta}(a):\mathcal H_{\mathbf N}(\theta)
\to\mathcal H_{\mathbf N}(\theta).
$$
Here $\theta\in\mathbb T^{2n}$ is a parameter, $\mathbf N\geq 1$ is an integer,
and the Hilbert spaces of quantum states $\mathcal H_{\mathbf N}(\theta)$,
defined in~\S\ref{s:quantum-states}, have dimension $\mathbf N^n$.
We denote the inner product on these spaces by $\langle\bullet,\bullet\rangle_{\mathcal H}$.
The semiclassical parameter is $h:=(2\pi \mathbf N)^{-1}$.

Every matrix $A\in\Sp(2n,\mathbb Z)$ defines a symplectic automorphism of the torus
$\mathbb T^{2n}$. This automorphism is quantized by a family of unitary maps
$$
M_{\mathbf N,\theta}:\mathcal H_{\mathbf N}(\theta)\to\mathcal H_{\mathbf N}(\theta)
$$
which satisfy the following \emph{exact Egorov's theorem}~\eqref{e:our-egorov} intertwining
the action of $A$ on $\mathbb T^{2n}$ with conjugation by $M_{\mathbf N,\theta}$:
\begin{equation}
  \label{e:egorov-intro}
M_{\mathbf N,\theta}^{-1}\Op_{\mathbf N,\theta}(a)M_{\mathbf N,\theta}=\Op_{\mathbf N,\theta}(a\circ A)\quad\text{for all}\quad a\in C^\infty(\mathbb T^{2n}).
\end{equation}
To construct such $M_{\mathbf N,\theta}$ we need to impose a \emph{quantization condition}
\eqref{e:theta-quantize-condition} on~$\theta$. The constructed operators are unique
up to multiplication by a unit complex constant. See~\S\ref{s:quantization-toric-auto} for more details and~\S\ref{s:explicit-formulas} for explicit formulas
for $\Op_{\mathbf N,\theta}(a)$ and~$M_{\mathbf N,\theta}$.

Throughout the paper we fix $A\in\Sp(2n,\mathbb Z)$ and
assume the following \emph{spectral gap} condition
on the spectrum $\Spec(A)$:
\begin{equation}
  \label{e:spectral-gap}
\text{$A$ has a simple eigenvalue $\lambda_+$ 
such that}\quad
\max_{\lambda\in \Spec(A)\setminus \{\lambda_+\}}|\lambda|<|\lambda_+|.
\end{equation}
This condition is crucial in the proof because it means there is
a one-dimensional `fast' direction in which the powers of $A$
grow faster than in other directions, which makes
it possible to apply the one-dimensional fractal uncertainty principle~--
see~\S\ref{s:outline}.

Our first result is the following analog of~\eqref{e:varfup-estimate}:
\begin{theo}
  \label{t:basic}
Assume that~$A$ satisfies~\eqref{e:spectral-gap} and the characteristic
polynomial of~$A$ is irreducible over the rationals. Then
for each $a\in C^\infty(\mathbb T^{2n})$, $a\not\equiv 0$,
there exists $C_a>0$ such that for all large enough $\mathbf N$
and every eigenfunction~$u\in\mathcal H_{\mathbf N}(\theta)$
of $M_{\mathbf N,\theta}$ we have
\begin{equation}
  \label{e:basic-estimate}
\|u\|_{\mathcal H}\leq C_a \|\Op_{\mathbf N,\theta}(a)u\|_{\mathcal H}.
\end{equation}
\end{theo}
Here the norm $\|\Op_{\mathbf N,\theta}(a)u\|_{\mathcal H}$
on the right-hand side of~\eqref{e:basic-estimate}
plays a similar role to the norm $\|u\|_{L^2(\Omega)}$
on the right-hand side of~\eqref{e:varfup-estimate}:
if $a$ is supported in some $\mathbf N$-independent subset
of $\mathbb T^{2n}$, then $\|\Op_{\mathbf N,\theta}(a)u\|_{\mathcal H}$
can be thought of as the norm of $u$ localized in the position-frequency
space to this set.

We remark that the conditions of Theorem~\ref{t:basic} are always satisfied
if $n=1$ (i.e. $A$ is a $2\times 2$ matrix) and $A$ is hyperbolic,
that is it has no eigenvalues on the unit circle. Thus Theorem~\ref{t:basic}
(or more precisely, Theorem~\ref{t:general} below) implies the result of~\cite{Schwartz-cat}.
See Figure~\ref{f:qcat1d} for a numerical illustration in the case $n=1$.
For $n\geq 2$, our assumption~\eqref{e:spectral-gap} does not require $A$ to be hyperbolic.
We also remark that having characteristic polynomial irreducible over~$\mathbb Q$ is a generic
property for integer symplectic matrices, see Rivin~\cite{rivin-irreducible}
and the book of Kowalski~\cite[Theorem~7.12]{Kowalski-sieve}.

\begin{figure}
\includegraphics[height=7.3cm]{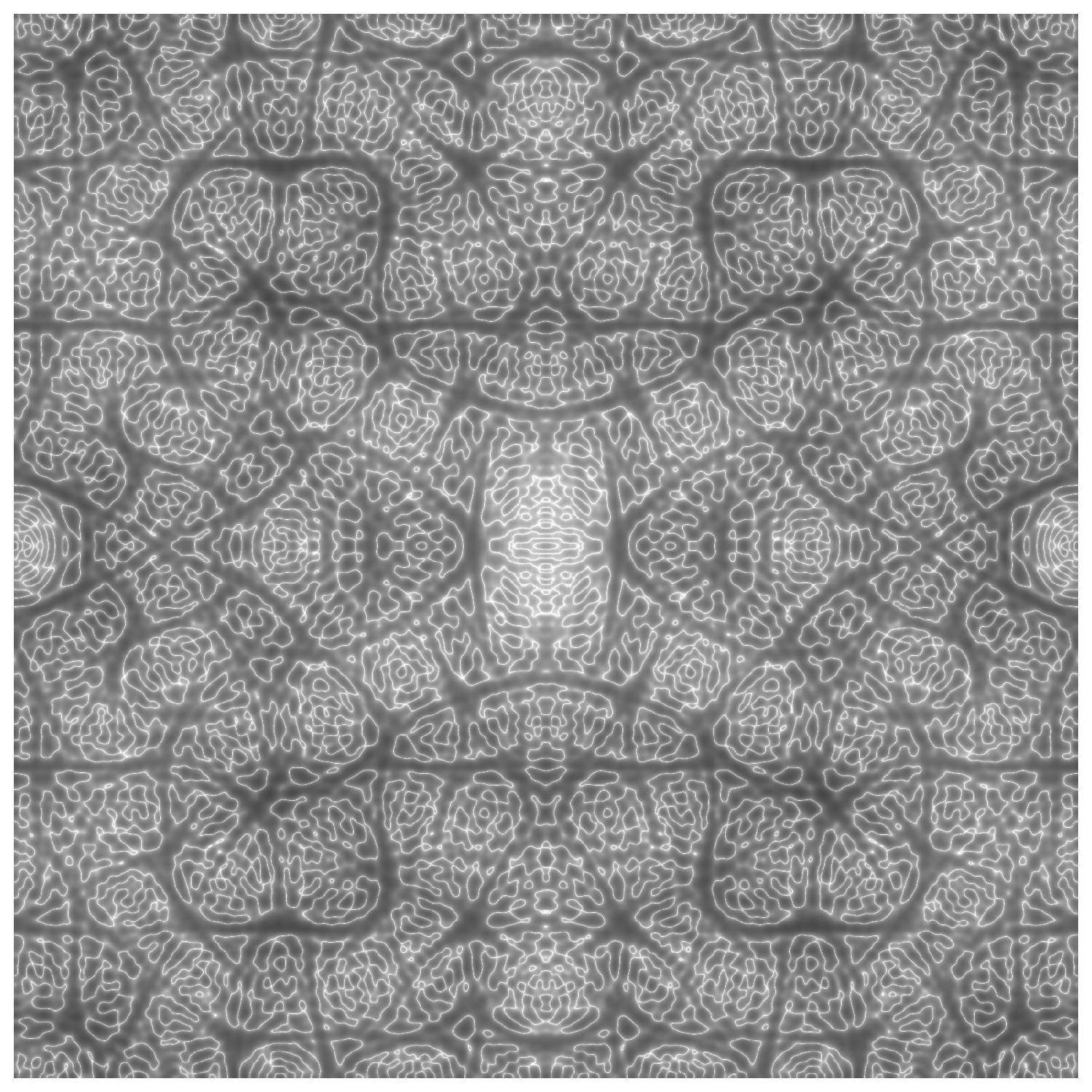}\qquad
\includegraphics[height=7.3cm]{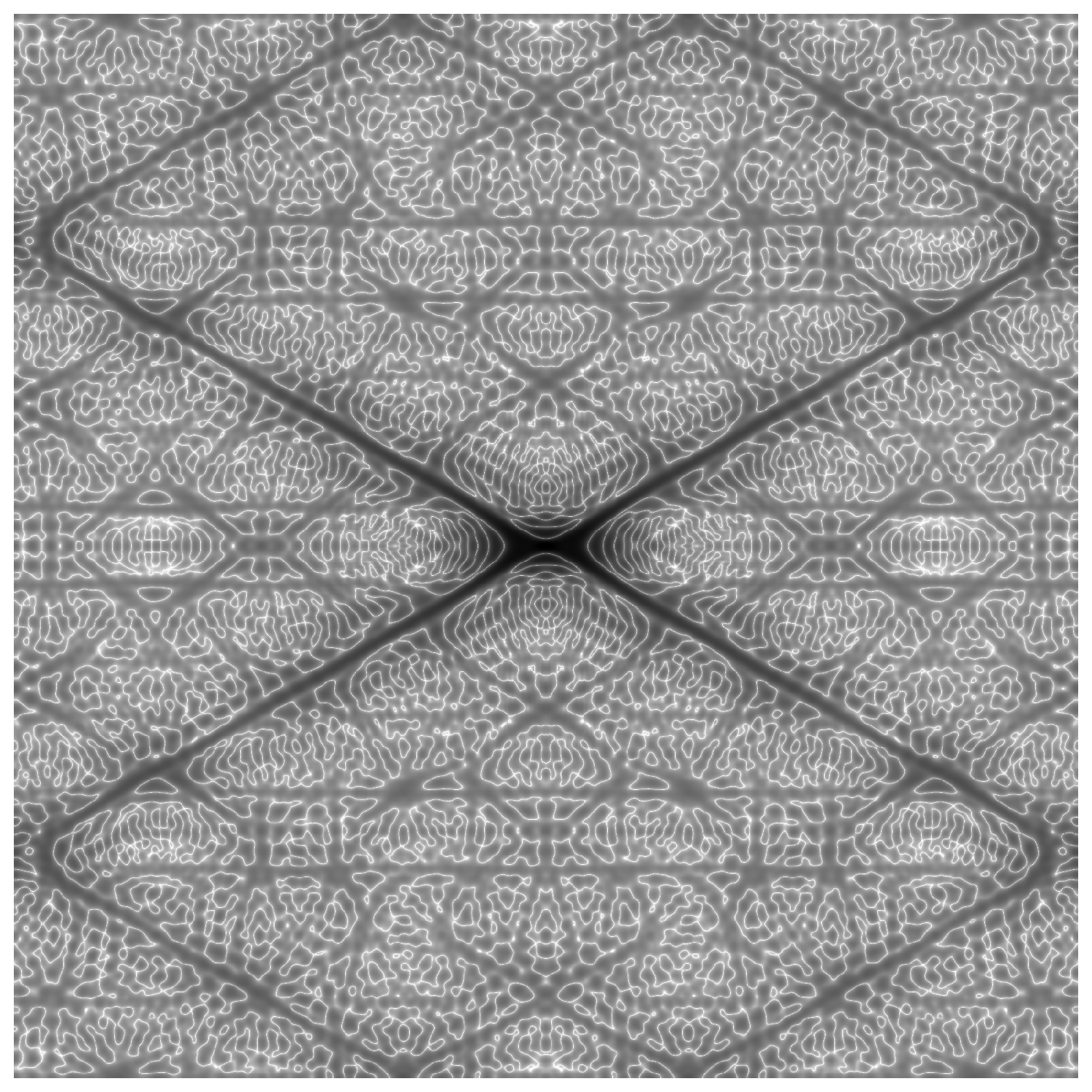}
\caption[]{Plots of concentration in the position-frequency space $\mathbb T^2$
(more precisely, plots of the corresponding Wigner matrices
convolved with appropriate Gaussians, on the logarithmic scale,
see~\eqref{e:Wigner}) of two eigenfunctions of a quantum cat map corresponding
to $A=\begin{pmatrix} 2 &3 \\ 1 & 2 \end{pmatrix}$,
$\mathbf N=780$.
On the left is a typical eigenfunction, showing equidistribution
consistent with the Quantum Ergodicity result of~\cite{Bouzouina-deBievre}.
On the right is a particular eigenfunction which exhibits scarring
discovered in~\cite{Faure-Nonnenmacher-dB}. This eigenfunction
does not violate Theorem~\ref{t:basic} because
the mass on the scar in the center
is approximately equal to the mass elsewhere.}
\label{f:qcat1d}
\end{figure}

\subsection{Further results}
\label{s:more-results}

Theorem \ref{t:basic} is a consequence of a more general result, Theorem~\ref{t:general} below, that applies to quasimodes of $M_{\mathbf{N},\theta}$ and does not require the irreducibility over~$\mathbb{Q}$ of the characteristic polynomial of $A$.
Before stating it, we need to introduce more notation.
In order to measure the strength of a quasimode, we introduce for $u \in \mathcal{H}_{\mathbf{N}}(\theta)$ the quantity
\begin{equation}\label{e:quasimode}
\mathbf r_M(u) := \min_{z \in \mathbb{S}^1} \n{ M_{\mathbf{N},\theta} u - z u}_{\mathcal{H}},
\end{equation}
where $\mathbb{S}^1$ denotes the unit circle in $\mathbb{C}$. 
Note that $\mathbf r_M(u)=0$ if and only if $u$ is an eigenfunction
of $M_{\mathbf N,\theta}$.

To relax the assumption on the characteristic polynomial of~$A$,
let us notice that by~\eqref{e:spectral-gap} the leading eigenvalue
$\lambda_+$ is real. Since the matrix~$A$ is symplectic, its transpose
is conjugate to its inverse,
thus $\lambda_- := \lambda_+^{-1}$ is also a simple eigenvalue for $A$
and $|\lambda_-|<1<|\lambda_+|$.
Moreover, all other eigenvalues $\lambda$ of $A$ satisfy 
$|\lambda_-|<|\lambda|<|\lambda_+|$.

Denote by $E_\pm\subset\mathbb R^{2n}$ the (real) eigenspaces of $A$ associated to $\lambda_\pm$. Let $V_{\pm}$ be the smallest subspace of~$\Q^{2n}$ such that $E_{\pm}$ is contained in $V_{\pm} \otimes \R$. Note that $V_\pm$ are invariant under~$A$.
Denote by
\begin{equation}
  \label{e:T-pm-def}
\mathbb T_\pm\ \subset\ \mathbb T^{2n}
\end{equation}
the subtori given by the projections of $V_\pm\otimes\mathbb R$ to $\mathbb T^{2n}$.

The tori $\mathbb{T}_\pm$ are relevant here because of their alternative dynamical definition given in Lemma \ref{l:minimality}: if $e_\pm$ is any eigenvector of $A$ associated to the eigenvalue $\lambda_\pm$, then the closure of the orbit of a point $x \in \mathbb{T}^{2n}$ by the translation flow generated by $e_\pm$ is $x + \mathbb{T}_\pm$. In our setting, these translation flows will play the role that was played by the horocyclic flows on the unit tangent bundle of hyperbolic surfaces in \cite{meassupp}. Let us also give
\begin{defi}\label{d:geometric_control_condition}
Let $U$ and $\T'$ be respectively an open subset and a subtorus of $\T^{2n}$
and assume that $A(\T')=\T'$. We say that $U$ satisfies the \emph{geometric control condition transversally to} $\T'$ if, for every $x \in \T^{2n}$, there exists $m \in \Z$ such that $A^m x + \T'$ intersects~$U$.
\end{defi}
We now state a more general version of Theorem~\ref{t:basic}:
\begin{theo}\label{t:general}
Assume that $A$ satisfies \eqref{e:spectral-gap}. Let $a \in C^\infty(\mathbb{T}^{2n})$ be such that $\set{a \neq 0}$ satisfies the geometric control condition transversally to $\T_+$ and $\T_-$.
Then there exists $C_a > 0$ such that for all large enough $\mathbf{N}$ and every $u \in \mathcal{H}_{\mathbf{N}}(\theta)$, we have
\begin{equation}
  \label{e:general-estimate}
\n{u}_{\mathcal{H}} \leq C_a \n{\Op_{\mathbf{N},\theta}(a)u}_{\mathcal{H}} + C_a \mathbf{r}_M(u) \log \mathbf{N} .
\end{equation}
\end{theo}
Theorem~\ref{t:general} implies Theorem~\ref{t:basic}. Indeed,
if the characteristic polynomial of~$A$ is irreducible over the rationals,
then $\T_\pm=\mathbb T^{2n}$ (see Lemma~\ref{l:irreducible}),
thus every nonempty open set satisfies the geometric control condition
transversally to $\T_+$ and $\T_-$.
However, one can find a matrix $A$ that satisfies \eqref{e:spectral-gap} but for which \eqref{e:basic-estimate} fails for certain choices of $a$ (in particular the characteristic polynomial of $A$ is reducible over $\mathbb{Q}$) -- see the examples in \S \ref{s:most_favorable}.

We will also prove the following theorem about damped quantum cat maps.
It is analogous to exponential energy decay for negatively curved surfaces
proved in~\cite{varfup} (following earlier work of Jin~\cite{Jin-DWE} in the constant
curvature case):
\begin{theo}\label{t:damped}
Assume that $A$ satisfies \eqref{e:spectral-gap}. Let $b \in C^\infty(\T^{2n})$ be such that $\va{b} \leq 1$. Assume that the set $\{|b| < 1\}$ satisfies the geometric control condition transversally to~$\T_+$ and~$\T_-$.
Then there exists $0 < \eta < 1$ such that for all $\mathbf{N}$ large enough the spectral radius of the operator $\Op_{\mathbf N,\theta}(b) M_{\mathbf{N},\theta}$ is less than $\eta$.
\end{theo}
Note that by Lemma~\ref{l:irreducible}, if the characteristic polynomial of $A$ is irreducible over~$\Q$ then the condition on~$b$ simplifies to $\{|b|<1\}\neq\emptyset$.
See Figure~\ref{f:damped} for a numerical illustration.

\begin{figure}
\includegraphics[width=7.45cm]{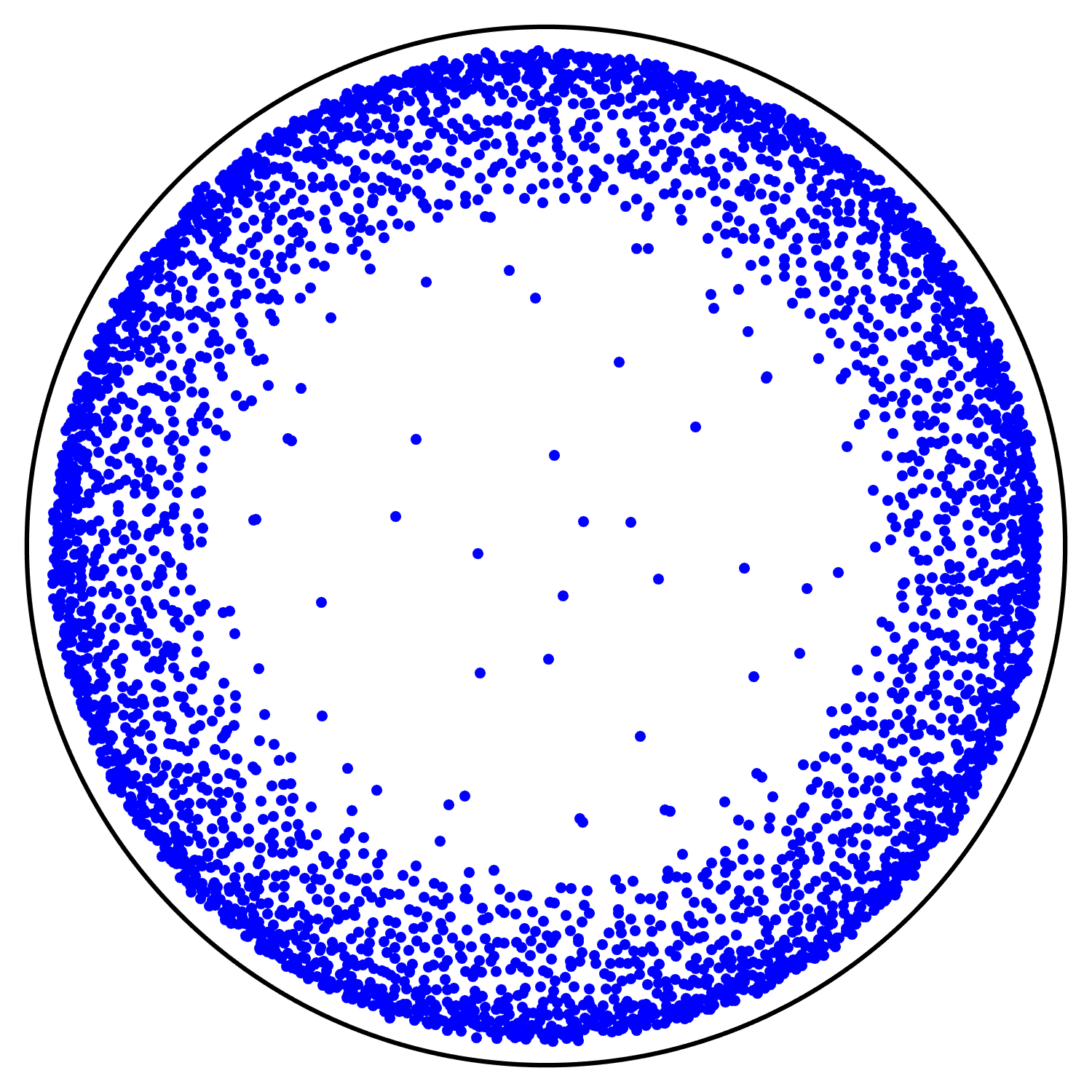}
\quad
\includegraphics[width=7.45cm]{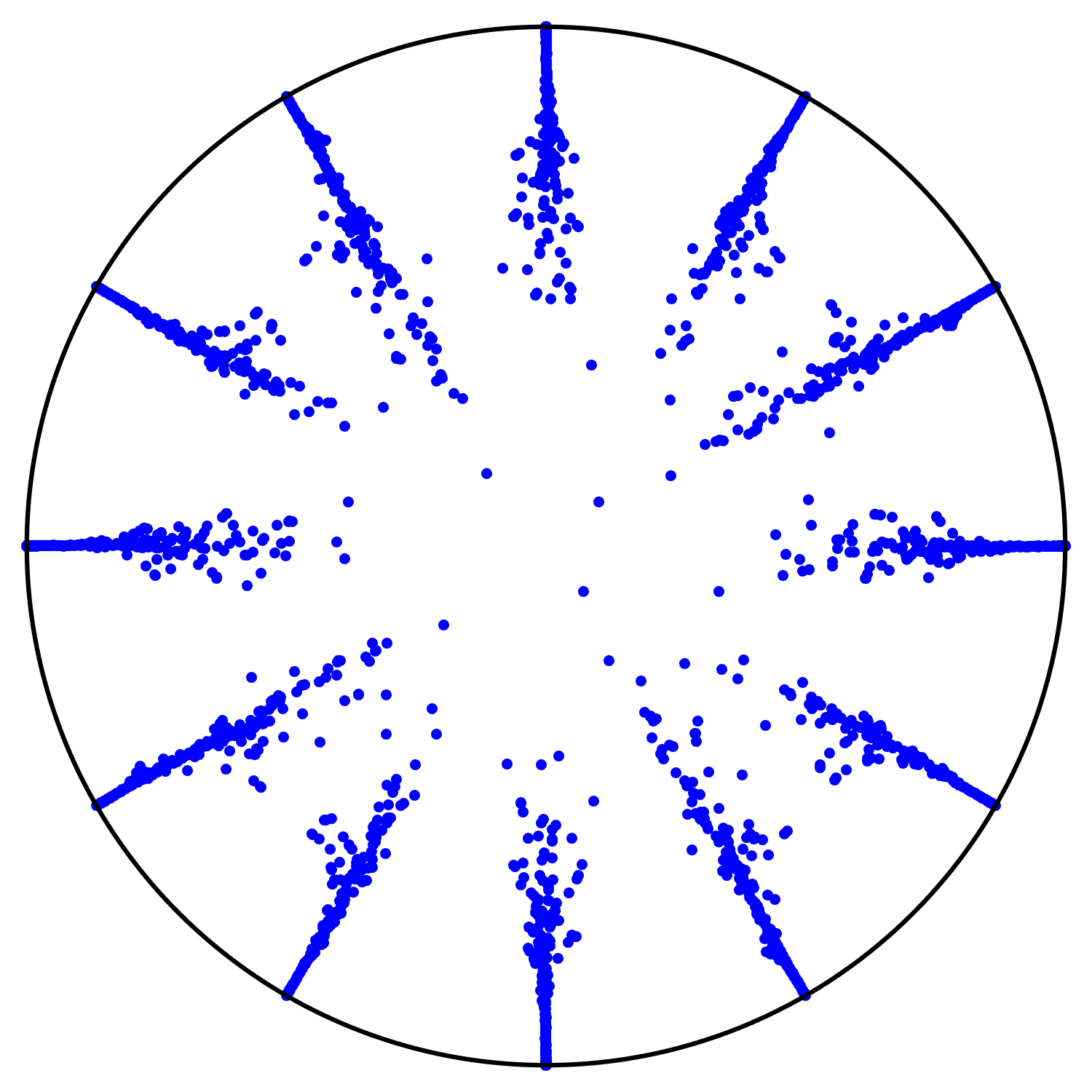}
\caption{Numerically computed eigenvalues for two damped quantum cat maps
$\Op_{\mathbf N,\theta}(b) M_{\mathbf{N},\theta}$, with the unit circle in the background.
In both cases $n=2$ (i.e. $A$ is a $4\times 4$ matrix), $\mathbf N=60$, and the set $\{|b|<1\}$ is inside the ${1\over 4}$-neighborhood of the origin $(0,0,0,0)\in\mathbb T^4$
in the $\ell^\infty$ norm. On the left
we use the matrix~\eqref{e:A-example-1} with irreducible characteristic polynomial
and the spectral gap of Theorem~\ref{t:damped} is visible.
On the right we use the matrix~\eqref{e:A-example-2}
and the condition of Theorem~\ref{t:damped} is not satisfied;
there does not appear to be a spectral gap.}
\label{f:damped}
\end{figure}

\subsection{Semiclassical measures and overview of history}
\label{s:history}

We now give an application of Theorem~\ref{t:general} to \emph{semiclassical measures}.
These measures describe the possible ways in which the mass of eigenfunctions
of $M_{\mathbf N,\theta}$ can distribute in the position-frequency space
in the limit $\mathbf N\to\infty$, and they are defined as follows:
\begin{defi}
Let $\mathbf N_j\to\infty$, $\theta_j\in\mathbb T^{2n}$ be sequences
such that the quantization condition~\eqref{e:theta-quantize-condition}
holds for all $\mathbf N_j,\theta_j$. Let $u_j\in\mathcal H_{\mathbf N_j}(\theta_j)$
be eigenfunctions of $M_{\mathbf N_j,\theta_j}$ of norm~1.
We say that the sequence $u_j$ \emph{converges weakly} to a Borel
measure $\mu$ on $\mathbb T^{2n}$ if
\begin{equation}
\label{e:measure-def}
  \langle \Op_{\mathbf N_j,\theta_j}(a)u_j,u_j\rangle_{\mathcal H}
  \to \int_{\mathbb T^{2n}}a\,d\mu\quad\text{for all}\quad
  a\in C^\infty(\mathbb T^{2n}).
\end{equation}
A measure $\mu$ on $\mathbb T^{2n}$ is called a \emph{semiclassical measure}
(associated to the toric automorphism $A$) if there exist sequences
$\mathbf N_j,\theta_j,u_j$ such that~\eqref{e:measure-def} holds.
\end{defi}
One can show (similarly to~\cite[Theorem~5.2]{Zworski-Book}
using a diagonal argument, the norm bound~\eqref{e:basic-norm},
Riesz representation theorem for the dual to $C^0(\mathbb T^{2n})$, and the sharp G\r arding inequality~\eqref{e:our-garding})
that each norm~1 sequence $u_j\in \mathcal H_{\mathbf N_j}(\theta_j)$
has a subsequence which has a weak limit in the sense of~\eqref{e:measure-def}.
Every semiclassical measure is a probability measure on $\mathbb T^{2n}$ (as follows from taking
$a=1$ in~\eqref{e:measure-def} and using that~$\Op_{\mathbf N,\theta}(1)=I$)
which is invariant under the map~$A$ (as follows from~\eqref{e:egorov-intro}).
Semiclassical measures for quantum cat maps are analogous to those for Laplacian
eigenfunctions on a Riemannian manifold $(\mathcal M,g)$, which are probability
measures on the cosphere bundle $S^*\mathcal M$ invariant under the geodesic flow~--
see~\cite[Chapter~5]{Zworski-Book} for more information.

As explained in~\S\ref{s:proof-measures},
Theorem~\ref{t:general} implies the following property of the support
of semiclassical measures:
\begin{theo}
  \label{t:measures}
Assume that $A$ satisfies~\eqref{e:spectral-gap}
and $\mu$ is a semiclassical measure associated to~$A$.
Then $\supp\mu$ contains a set of the form
$x+\mathbb T_+$ or $x+\mathbb T_-$ for some $x\in\mathbb T^{2n}$.
In particular, if the characteristic polynomial of~$A$ is irreducible over
the rationals, then $\supp\mu=\mathbb T^{2n}$.  
\end{theo}
In Appendix~\ref{appendix:symplectic_matrices}, we discuss the algebraic properties of the spaces $V_+$ and $V_-$, and their implications in terms of the supports of semiclassical measures for $A$. In particular, we give examples of the following situations:
\begin{itemize}
\item the characteristic polynomial of $A$ is irreducible, and hence the semiclassical measures are fully supported;
\item $\mathbb{Q}^{2n} = V_+ \oplus V_-$, in which case there are semiclassical measures supported on translates of $\mathbb{T}_+$ and $\T_-$ due to \cite{Kelmer-cat};
\item $V_+ + V_- \neq \mathbb{Q}^{2n}$ and there are semiclassical measures supported in $\T_+$ and $\T_-$;
\item $V_+ + V_- \neq \mathbb{Q}^{2n}$ but the supports of semiclassical measures (associated to a certain basis of eigenfunctions) are strictly larger than $\T_+$ or $\T_-$.
\end{itemize}
The last two cases emphasize the fact that when $V_+ + V_- \neq \mathbb{Q}^{2n}$, one has to study the action of $A$ on $\mathbb{Q}^{2n}/(V_+ + V_-)$ in order to refine the information on the support of semiclassical measures given by Theorem~\ref{t:measures}.

We now give a brief overview of history of semiclassical measures in quantum chaos,
referring the reader to review articles by Marklof~\cite{Marklof-QE}, Sarnak~\cite{Sarnak-QE-review}, Nonnenmacher~\cite{Nonnenmacher-Anatomy}, and Dyatlov~\cite{around-shnirelman}
for more information.
In the setting of eigenfunctions of the Laplacian on compact negatively curved Riemannian manifolds, the
Quantum Ergodicity theorem of Shnirelman~\cite{Shnirelman1}, Zelditch~\cite{Zelditch-QE},
and Colin de Verdi\`ere~\cite{CdV-QE} states that a density~1 sequence
of eigenfunctions \emph{equidistributes}, i.e. converges to the Liouville measure
in a way analogous to~\eqref{e:measure-def}.
The Quantum Unique Ergodicity conjecture of Rudnick--Sarnak~\cite{Rudnick-Sarnak-QUE}
claims that the whole sequence of eigenfunctions equidistributes,
i.e. the Liouville measure is the only semiclassical measure.
It is open in general but known in the particular cases of \emph{arithmetic}
hyperbolic surfaces for joint eigenfunctions of the Laplacian and
the Hecke operators, which are additional symmetries present in the arithmetic case~--
see Lindenstrauss~\cite{Lindenstrauss-QUE} and Brooks--Lindenstrauss~\cite{Brooks-Lindenstrauss-QUE}.

The works of Anantharaman~\cite{Anantharaman-Entropy},
Anantharaman--Nonnenmacher~\cite{Anantharaman-Nonnenmacher-Entropy}, Anan\-tharaman--Koch--Nonnenmacher~\cite{Anantharaman-Koch-Nonnenmacher}, Rivi\`ere~\cite{Riviere-Entropy-1,Riviere-Entropy-2}, and Anantharaman--Silberman~\cite{Anantharaman-Silberman}
establish \emph{entropy bounds}, which are positive lower bounds
on the Kolmogorov--Sina\u\i{} entropy of semiclassical measures.
As mentioned in the beginning of the introduction, in the setting of negatively curved surfaces the paper~\cite{varfup} gives another restriction: all semiclassical
measures have full support. In the case of hyperbolic surfaces this was previously proved by Dyatlov--Jin~\cite{meassupp}.

The study of quantum cat maps, which is the setting of the present paper,
goes back to the work of Hannay--Berry~\cite{Hannay-Berry}. Bouzouina--De Bi\`evre~\cite{Bouzouina-deBievre} proved
the analogue of quantum ergodicity in this setting.
Faure--Nonnenmacher--De Bi\`evre~\cite{Faure-Nonnenmacher-dB} constructed 
examples of semiclassical measures for 2-dimensional quantum cat maps which are not the Liouville measure; this contradicts the Quantum Unique Ergodicity conjecture for quantum cat maps
but it does not contradict Theorem~\ref{t:basic} since
these measures were supported on the entire $\mathbb T^2$.
Faure--Nonnenmacher~\cite{Faure-Nonnenmacher-cat}, Brooks~\cite{Brooks-cat},
and Rivi\`ere~\cite{Riviere-Entropy-3}
established `entropy-like' bounds on semiclassical measures;
see also Anantharaman--Nonnenmacher~\cite{Anantharaman-Nonnenmacher-Baker} and
Gutkin~\cite{Gutkin} for entropy bounds for other models of quantum maps.
As mentioned above, Schwartz~\cite{Schwartz-cat} obtained
Theorem~\ref{t:basic} for 2-dimensional quantum cat maps.

Kurlberg--Rudnick~\cite{Kurlberg-Rudnick} introduced the analogue of Hecke
operators in the setting of 2-dimensional quantum cat maps. For the joint eigenfunctions
of a quantum cat map and the Hecke operators (known as \emph{arithmetic}
eigenfunctions) they showed that Quantum Unique Ergodicity holds;
see also Gurevich--Hadani~\cite{Gurevich-Hadani-cat}.
This does not contradict the counterexample in~\cite{Faure-Nonnenmacher-dB}
since the quantum cat maps used there have eigenvalues of high multiplicity.

In the setting of higher dimensional quantum cat maps which have an
\emph{isotropic invariant rational subspace},
Kelmer~\cite{Kelmer-cat} constructed examples of semiclassical measures
(for arithmetic eigenfunctions) which are supported on proper submanifolds
of $\mathbb T^{2n}$. On the other hand, if there are no isotropic invariant rational
subspaces, then~\cite{Kelmer-cat} gives Quantum Unique Ergodicity for
arithmetic eigenfunctions. Compared to Theorem~\ref{t:measures}, the conclusion of~\cite{Kelmer-cat}
is stronger than $\supp \mu=\mathbb T^{2n}$ and the assumption
is weaker than the characteristic polynomial of~$A$ being irreducible
over rationals (since~\cite{Kelmer-cat} only assumes
that there are no \emph{isotropic} invariant rational subspaces),
however the result of~\cite{Kelmer-cat} only applies to arithmetic eigenfunctions.
For a further discussion of the relation of our results
with those of~\cite{Kelmer-cat}, see Appendix~\ref{appendix:symplectic_matrices}.

\subsection{Outline of the proof}
  \label{s:outline}

We now give an outline of the proof of Theorem~\ref{t:general}.
For simplicity we assume that the symbol $a$ satisfies $0\leq a\leq 1$, as well as the following stronger version
of the geometric control condition: there exists an open set $\mathcal U\subset\mathbb T^{2n}$
such that $a=1$ on $\mathcal U$ and
each shifted torus $x+\T_+$, $x+\T_-$ intersects~$\mathcal U$.
We also assume that $\mathbf N$ is large and $u\in\mathcal H_{\mathbf N}(\theta)$ is an eigenfunction of~$M_{\mathbf N,\theta}$.

\subsubsection{Reduction to the key estimate}

Take the partition of unity on $\mathbb T^{2n}$
\begin{equation}
  \label{e:partitor-intro}
1=b_1+b_2,\quad
b_1:=a,\quad
b_2:=1-a,\quad
\supp b_2\cap\mathcal U=\emptyset.
\end{equation}
Quantizing $b_1,b_2$ to pseudodifferential operators, we get the quantum partition of unity
$$
I=B_1+B_2,\quad
B_1:=\Op_{\mathbf N,\theta}(b_1),\quad
B_2:=\Op_{\mathbf N,\theta}(b_2).
$$
For an operator $L$ on $\mathcal H_{\mathbf N,\theta}$ and an integer~$T$,
define the conjugated operator
$$
L(T):=M_{\mathbf N,\theta}^{-T} L M_{\mathbf N,\theta}^T:\mathcal H_{\mathbf N}(\theta)\to\mathcal H_{\mathbf N}(\theta).
$$
Since $u$ is an eigenfunction of the unitary operator $M_{\mathbf N,\theta}$,
we have for all $T$ and $L$
\begin{equation}
  \label{e:prop-norm}
\|L(T)u\|_{\mathcal H}=\|L M_{\mathbf N,\theta}^Tu\|_{\mathcal H}=\|Lu\|_{\mathcal H}.
\end{equation}
We will propagate up to time $2T_1$, where $T_1\in\mathbb N$ is chosen so that
\begin{equation}
  \label{e:T-1-def-intro}
T_1\approx {\rho\over\log|\lambda_+|}\log\mathbf N
\end{equation}
and the constant $\rho>0$ is chosen later.
We write a refined quantum partition of unity
\begin{equation}
  \label{e:partition-intro}
\begin{gathered}
I=B_{\mathcal X}+B_{\mathcal Y},\quad
B_{\mathcal X}:=B_2(2T_1-1)\cdots B_2(1)B_2(0),\\
B_{\mathcal Y}:=\sum_{j=0}^{2T_1-1} B_2(2T_1-1)\cdots B_2(j+1)B_1(j).
\end{gathered}
\end{equation}
Since $|b_2|\leq 1$, we have $\|B_2\|_{\mathcal H\to\mathcal H}\leq 1+\mathcal O(\mathbf N^{-{1\over 2}})$ (see~\eqref{e:basic-norm}), thus we can bound the norm
of the product $B_2(2T_1-1)\cdots B_2(j+1)$ by~2. Applying the partition~\eqref{e:partition-intro} to~$u$ and using~\eqref{e:prop-norm} with $L:=B_1=\Op_{\mathbf N,\theta}(a)$, we then get
\begin{equation}
  \label{e:puttor-intro}
\|u\|_{\mathcal H}\leq \|B_{\mathcal X}u\|_{\mathcal H}+4T_1\|\Op_{\mathbf N,\theta}(a)u\|_{\mathcal H}.
\end{equation}
The key component of the proof is the following estimate valid
for the right choice of~$T_1$:
\begin{equation}
  \label{e:key-estimate-intro}
\|B_{\mathcal X}\|_{\mathcal H\to\mathcal H}=\mathcal O({\mathbf N}^{-\beta})\quad\text{for some}\quad \beta>0\quad\text{as}\quad
\mathbf N\to\infty.
\end{equation}
Here the exponent $\beta$ only depends on the matrix $A$ and the choice of the partition~\eqref{e:partitor-intro}.

Together with~\eqref{e:puttor-intro}, the key estimate~\eqref{e:key-estimate-intro}
implies the following weaker version of Theorem~\ref{t:general}:
\begin{equation}
  \label{e:lossor-intro}
\|u\|_{\mathcal H}\leq C\log\mathbf N\|\Op_{\mathbf N,\theta}(a)u\|_{\mathcal H}.
\end{equation}
To remove the $\log\mathbf N$ prefactor, we revise the decomposition $I=B_{\mathcal X}+B_{\mathcal Y}$ in the same way as in~\cite{meassupp,varfup} (which in turn
was inspired by~\cite{Anantharaman-Entropy}), including more terms into $B_{\mathcal X}$
and using that the norm bound in~\eqref{e:key-estimate-intro} (or rather, its slight generalization) is a negative power of $\mathbf N$. See~\S\ref{s:partition-control} for details.

\subsubsection{Microlocal structure of the propagated operators}

We now give an outline of the proof of the estimate~\eqref{e:key-estimate-intro}.
To simplify the presentation, we assume the following stronger version of the gap condition~\eqref{e:spectral-gap}:
\begin{equation}
  \label{e:spectral-gap-strong}
\text{$A$ has a simple eigenvalue $\lambda_+$ 
such that}\quad
\max_{\lambda\in \Spec(A)\setminus \{\lambda_+\}}|\lambda|<\sqrt{|\lambda_+|}.
\end{equation}
Let $E_\pm\subset\mathbb R^{2n}$ be the one-dimensional eigenspaces 
of $A$ corresponding to $\lambda_+$ and $\lambda_-:=\lambda_+^{-1}$.
We also define $L_\mp$ to be the sum of the generalized eigenspaces
corresponding to all eigenvalues of $A$ other than $\lambda_\pm$.
Then
$$
\mathbb R^{2n}=E_+\oplus L_-=E_-\oplus L_+.
$$
By~\eqref{e:spectral-gap-strong} we have the following norm bounds
as $T\to\infty$:
\begin{equation}
  \label{e:normer-intro}
\|A^{\pm T}\|=\mathcal O(|\lambda_+|^T),\quad
\|A^{\pm T}|_{L_\mp}\|=\mathcal O(|\lambda_+|^{T\over 2}).
\end{equation}
Returning to the proof of~\eqref{e:key-estimate-intro}, conjugating
by $M_{\mathbf N,\theta}^{T_1}$ we reduce it to the bound
$$
\begin{gathered}
\|B_-B_+\|_{\mathcal H\to\mathcal H}=\mathcal O(\mathbf N^{-\beta})\quad\text{for some}\quad \beta>0\\
\text{where}\quad
B_-:=B_2(T_1-1)\cdots B_2(0),\quad
B_+:=B_2(-1)\cdots B_2(-T_1).
\end{gathered}
$$
We would like to write the operators $B_\pm$ as quantizations
of some symbols.
By the exact Egorov's Theorem~\eqref{e:egorov-intro}, we have
$B_2(j)=\Op_{\mathbf N,\theta}(b_2\circ A^j)$ for all $j$.
However, when $j$ is too large the derivatives of the symbols
$b_2\circ A^j$ grow too fast with $\mathbf N$ and the methods
of semiclassical analysis no longer apply.

If $0\leq j<T_1$, then by~\eqref{e:normer-intro}
the derivatives of $b_2\circ A^j$ in the eigendirection~$E_+$ are bounded by $|\lambda_+|^{T_1}$ but the derivatives along the complementary hyperplane~$L_-$ are bounded by $|\lambda_+|^{T_1/2}$. Thus
 $b_2\circ A^j$ lies in the anisotropic symbol class
$S_{L_-,\rho,{\rho\over 2}}(\mathbb T^{2n})$, consisting of $\mathbf N$-dependent functions
in $C^\infty(\mathbb T^{2n})$
such that
each derivative along~$L_-$ gives an $\mathbf N^{\rho\over 2}$ growth
and derivatives in other directions give an $\mathbf N^{\rho}$ growth~-- see~\S\S\ref{s:symbol-coisotropic},\ref{s:quantization-on-torus} for details. In order
for the standard properties of semiclassical quantization to hold,
we need $\rho+{\rho\over 2}<1$, that is
\begin{equation}
  \label{e:rho-restr-intro}
0\leq \rho<\textstyle{2\over 3}.
\end{equation}
For such $\rho$ we can then show (see Lemma~\ref{l:long_logarithmic_words}) that
for any $\delta\in (0,1-{3\rho\over 2})$
$$
B_-=\Op_{\mathbf N,\theta}(b_-)+\mathcal O(\mathbf N^{-\delta})_{\mathcal H\to\mathcal H}\quad\text{where}\quad b_-:=\prod_{j=0}^{T_1-1}b_2\circ A^j.
$$
Reversing the direction of propagation and replacing $L_-$ with $L_+$, we similarly have
$$
B_+=\Op_{\mathbf N,\theta}(b_+)+\mathcal O(\mathbf N^{-\delta})_{\mathcal H\to\mathcal H}
\quad\text{where}\quad b_+:=\prod_{j=1}^{T_1}b_2\circ A^{-j}.
$$
Then the key estimate~\eqref{e:key-estimate-intro} reduces to
\begin{equation}
  \label{e:key-estimate-intro-2}
\|\Op_{\mathbf N,\theta}(b_-)\Op_{\mathbf N,\theta}(b_+)\|_{\mathcal H\to\mathcal H}=\mathcal O(\mathbf N^{-\beta})\quad\text{for some}\quad \beta>0.
\end{equation}
The symbols $b_\pm$ lie in the classes $S_{L_\pm,\rho+\varepsilon,{\rho\over 2}+\varepsilon}(\mathbb T^{2n})$ for each $\varepsilon>0$ but when $\rho>{1\over 2}$
they cannot be put in the same symbol calculus. This is important
because otherwise the norm of~$\Op_{\mathbf N,\theta}(b_-)\Op_{\mathbf N,\theta}(b_+)$ would
converge as $\mathbf N\to\infty$ to $\sup|b_-b_+|$, which would be equal to~1
as long as $\supp b_2$ contains at least one periodic trajectory of~$A$.

\subsubsection{Applying the fractal uncertainty principle}

Of course, just not being in the same symbol calculus does not give
the bound~\eqref{e:key-estimate-intro-2}. This is where the \emph{fractal uncertainty principle} of~\cite{fullgap} enters.

We first discuss the case
$n=1$, that is, the matrix $A$ has size $2\times 2$.
Then the fractal uncertainty principle shows the decay bound~\eqref{e:key-estimate-intro-2}
if the sets $\pi_\pm(\supp b_\pm)$ are \emph{porous} on scales
$C\mathbf N^{-\rho}$ to~1 in the sense of Definition~\ref{d:porous} below,
and $\rho>{1\over 2}$.
Here we think of $\supp b_\pm$ as subsets of $[0,1]^2\subset\mathbb R^{2}$ and
$\pi_-,\pi_+:\mathbb R^{2}\to\mathbb R$ are two linearly independent linear functionals.
(The original result of~\cite{fullgap} is stated in the particular case $\pi_-(x,\xi)=x$, $\pi_+(x,\xi)=\xi$ and the general case follows via conjugation by a metaplectic transformation.)

\begin{figure}
\includegraphics[height=8cm]{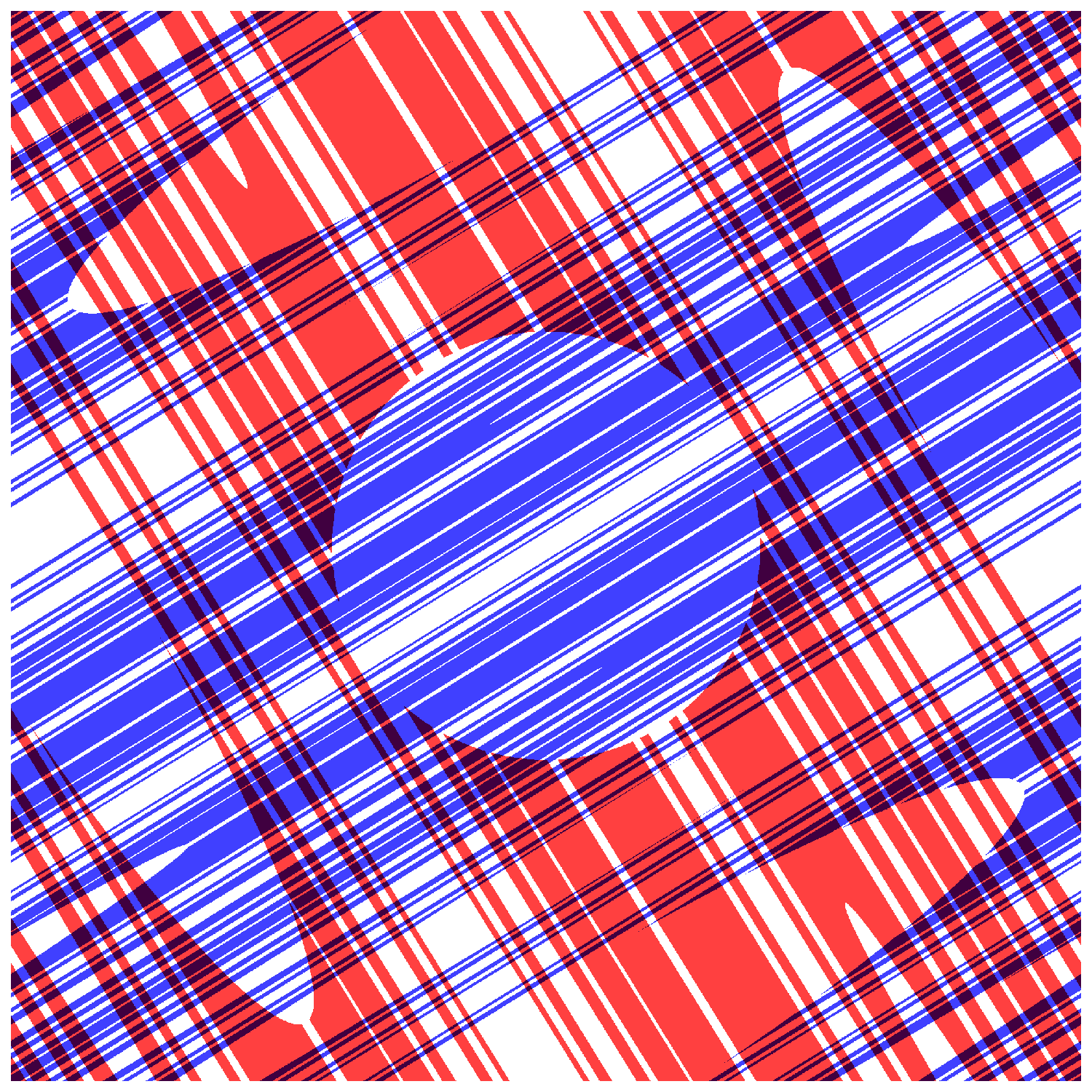}
\caption{A numerical illustration of the sets $\supp b_-$ (red) and
$\supp b_+$ (blue) for a 2-dimensional cat map. Each of these sets
is `smooth' in one of the eigendirections of $A$ and porous in the other eigendirection.}
\label{f:porosity}
\end{figure}

The required porosity property holds
if we choose the projections $\pi_\pm$ such that their kernels are given by the eigenspaces
$E_\pm=L_\pm$. This is illustrated on Figure~\ref{f:porosity},
and can be proved (taking the case of $\supp b_-$ to fix notation)
roughly speaking by combining the following observations:
\begin{enumerate}
\item Whether or not some $z\in\mathbb T^2$ belongs to $\supp b_-$
depends on the forward trajectory $\{A^jz\mid 0\leq j<T_1\}$,
more specifically on whether the points in this trajectory all lie
in $\supp b_2$.
\item The intersection of $\supp b_-$ with any line in the direction $E_+$
is porous on scales $C\mathbf N^{-\rho}$ to~1. Indeed,
the projection of $E_+$ to $\mathbb T^2$ is dense,
so any sufficiently large line segment in the direction of $E_+$
intersects the complement of $\supp b_2$ (recalling~\eqref{e:partitor-intro}).
This creates the pores on scale 1. To get pores on smaller scales, we use
that segments in the direction of $E_+$ are expanded by the map~$A^j$ by the factor $|\lambda_+|^j$, so the condition that $\supp b_-\subset A^{-j}(\supp b_2)$
creates pores on scales $|\lambda_+|^{-j}\in [\mathbf N^{-\rho},1]$.
\item If $z,w\in\mathbb T^2$ lie on the same line segment in the direction $E_-$
(of bounded length), then the forward trajectories $A^jz$, $A^jw$
converge to each other as $j\to\infty$. Thus the projection $\pi_-(\supp b_-)$
looks similar to the intersection of $\supp b_-$ with any fixed line segment
in the direction of $E_+$, which we already know is porous.
\end{enumerate}
Now, if we take ${1\over 2}<\rho<{2\over 3}$, then the fractal uncertainty principle
applies and gives the key estimate. This roughly corresponds
to the proof in~\cite{Schwartz-cat}.

We now move on to the case of higher dimensions $n>1$ which is the main novelty of this paper.
As mentioned above, the higher dimensional version of fractal uncertainty principle is not
available. However, we can still derive the key estimate~\eqref{e:key-estimate-intro-2}
from the \emph{one-dimensional} fractal uncertainty principle as long
as the projections $\pi_\pm(\supp b_\pm)$ are porous on scales $C\mathbf N^{-\rho}$ to~1 for some $\rho>{1\over 2}$,
where $\pi_\pm:\mathbb R^{2n}\to\mathbb R$ are some fixed linear maps such that
$\ker \pi_+\cap \ker \pi_-$ is a codimension~2 symplectic subspace of $\mathbb R^{2n}$.
Following the case $n=1$, it is natural to take $\pi_\pm$ such that
their kernels are given by the spaces $L_\pm$. However, we cannot expect
$\pi_\pm(\supp b_\pm)$ to be porous because the observation~(3) above is no longer valid:
there exist $z,w$ such that $z-w\in L_-$ but $A^jz-A^jw\not\to 0$ as $j\to\infty$,
for example one can take $z-w$ to be an eigenvector of $A$ with any eigenvalue
$\lambda\neq\lambda_+$ such that $|\lambda|\geq 1$.

To deal with this issue, we split the product
of operators in~\eqref{e:key-estimate-intro-2} into a sum of many pieces:
\begin{equation}
  \label{e:splittor-intro}
\Op_{\mathbf N,\theta}(b_-)\Op_{\mathbf N,\theta}(b_+)=\sum_k \Op_{\mathbf N,\theta}(b_-)
\Op_{\mathbf N,\theta}(\psi_k)\Op_{\mathbf N,\theta}(b_+)
\end{equation}
where $\{\psi_k\}$ is a partition of unity on $\mathbb T^{2n}$ such that
each $\supp\psi_k$ looks like a ball of radius $\mathbf N^{-{\rho\over 2}}$. We then observe that:
\begin{figure}
\includegraphics[height=7cm]{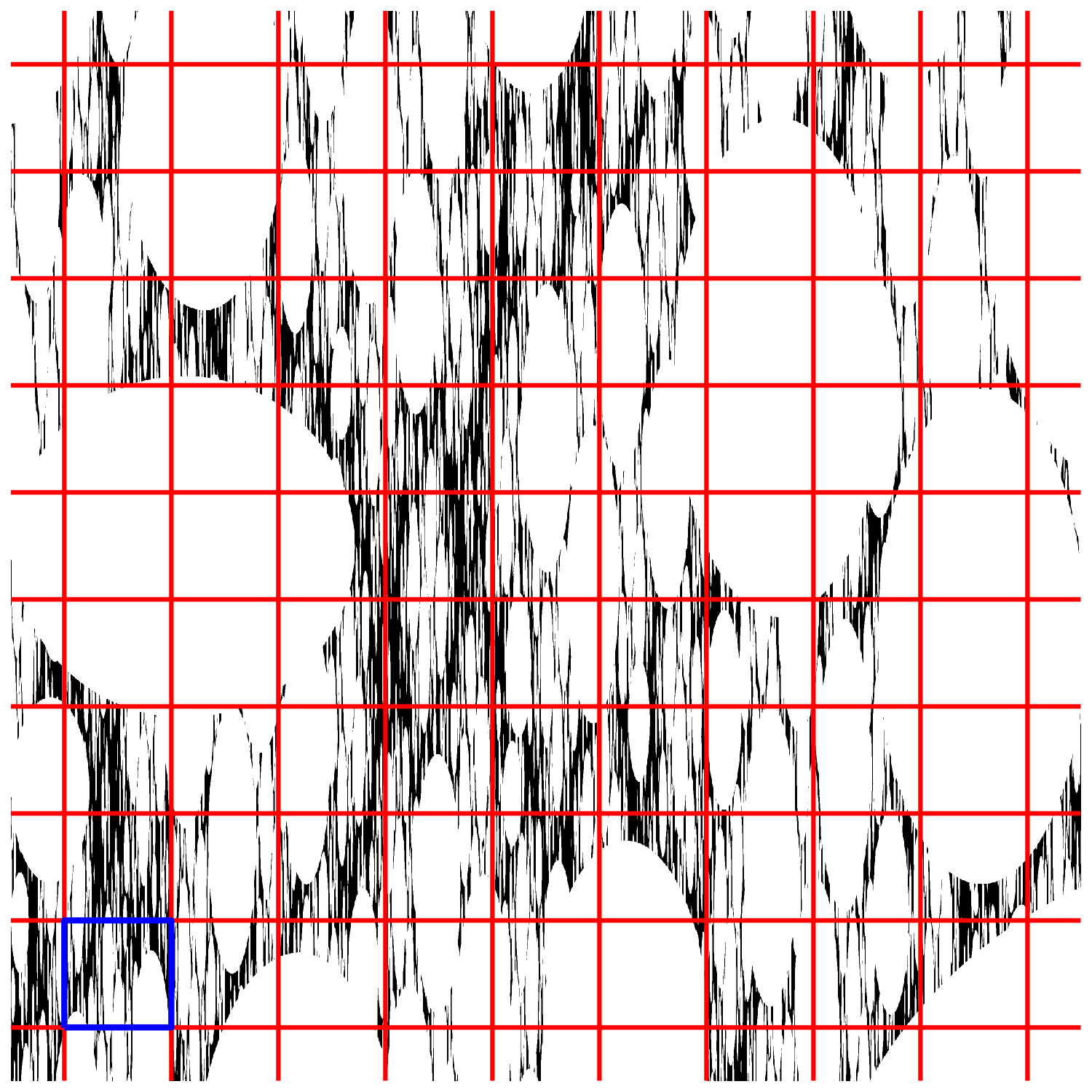}
\qquad
\includegraphics[height=7cm]{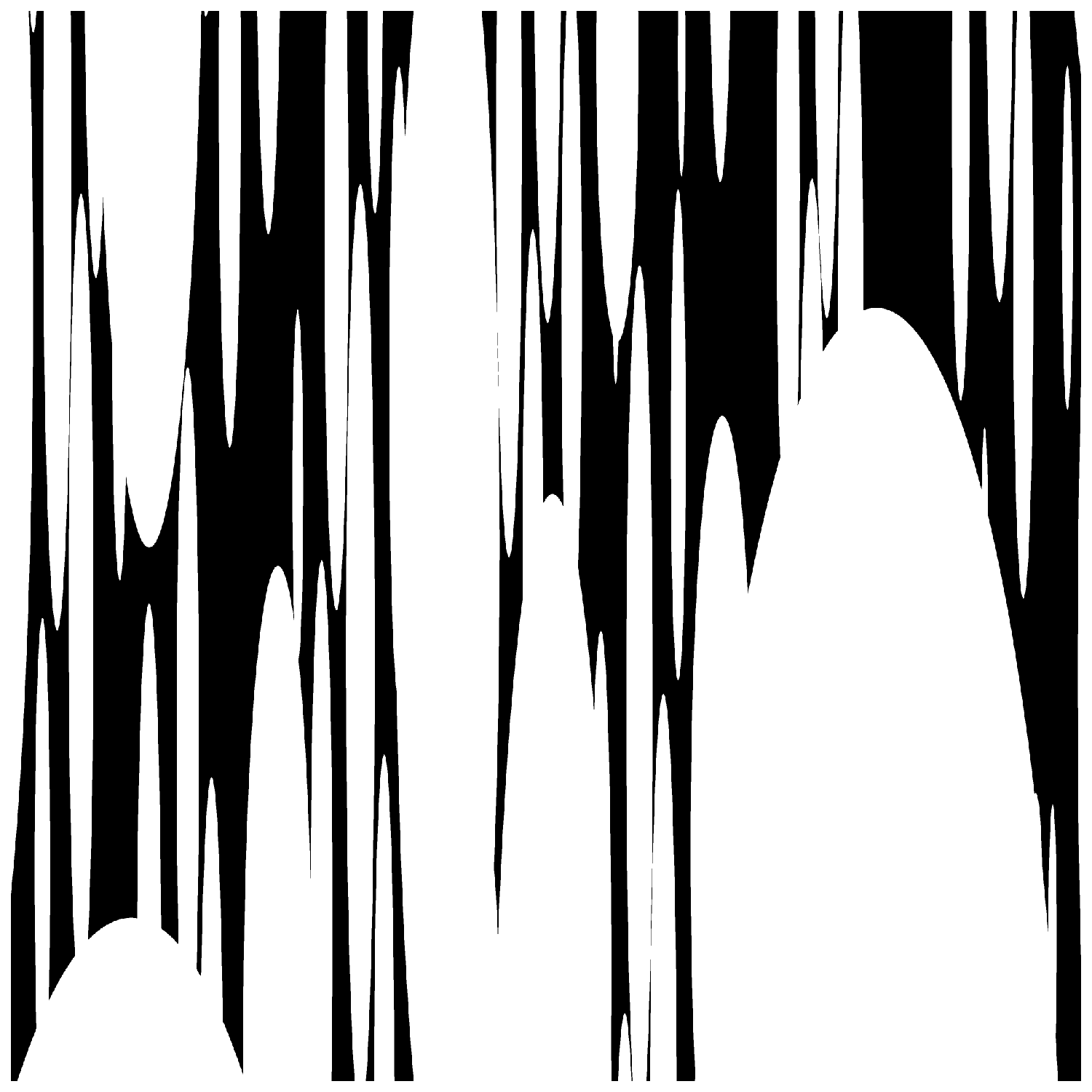}
\caption{Left: A numerical illustration of the set $\supp b_-$
for a 4-dimensional cat map, intersected with some 2-dimensional plane.
The little squares correspond to the supports of the functions $\psi_k$.
Right: The zoomed in image of the blue square on the left. The projection
onto the horizontal direction is $\pi_-$ and the set $\pi_-(\supp b_-\cap\supp\psi_k)$
is porous.}
\label{f:porosity2}
\end{figure}
\begin{itemize}
\item the terms $B_{(k)}$ in the sum in~\eqref{e:splittor-intro} form an almost orthogonal
family, namely
$$
B_{(k)}^*B_{(\ell)}=\mathcal O(\mathbf N^{-\infty}),\quad
B_{(k)}B_{(\ell)}^*=\mathcal O(\mathbf N^{-\infty})\quad\text{when}\quad
\supp\psi_k\cap\supp\psi_{\ell}=\emptyset.
$$
This follows from the nonintersecting support property of semiclassical calculus,
where we use that the symbols $\psi_k$, $\psi_{\ell}$ grow by $\mathbf N^{{\rho\over 2}}$
with each differentiation and thus belong to both
the calculi $S_{L_\pm,\rho,{\rho\over 2}}(\mathbb T^{2n})$ where the symbols $b_\pm$ lie.
By the Cotlar--Stein Theorem,
this reduces the estimate~\eqref{e:key-estimate-intro-2} to a bound on the
norms of the individual summands
\begin{equation}
  \label{e:key-estimate-intro-3}
\|\Op_{\mathbf N,\theta}(b_-)
\Op_{\mathbf N,\theta}(\psi_k)\Op_{\mathbf N,\theta}(b_+)\|_{\mathcal H\to\mathcal H}=\mathcal O(\mathbf N^{-\beta})\quad\text{for some}\quad \beta>0.
\end{equation}
\item Each of the localized symbols $b_\pm\psi_k$ does satisfy the porosity condition, namely
$\pi_\pm(\supp(b_\pm\psi_k))$ is porous on scales $C\mathbf N^{-\rho}$ to~1. This is because
a localized version of the observation~(3) above holds:
if $z,w\in\mathbb R^{2n}$ and $z-w\in L_-$ has length $\leq \mathbf N^{-{\rho\over 2}}$,
then for $0\leq j\leq T_1$ the difference $A^jz-A^jw$ is small.
Indeed, by~\eqref{e:spectral-gap-strong} the expansion rate of $A^j$ in the direction
of $L_-$ is slower than $|\lambda_+|^{j\over 2}$, and
$\mathbf N^{-{\rho\over 2}}|\lambda_+|^{T_1\over 2}\leq 1$.
(See Figure~\ref{f:porosity2}.)
On the other hand, an analogue of the observation~(2) above holds as well:
the closure of the projection of $E_+$ onto $\mathbb T^{2n}$ is equal to the subtorus
$\T_+$, all of whose shifts intersect the complement of $\supp b_2$
(see the beginning of~\S\ref{s:outline}).
Thus the one-dimensional
fractal uncertainty principle can be applied to give~\eqref{e:key-estimate-intro-3}
and thus finish the proof.
\end{itemize}
In the above argument we again fix $\rho\in ({1\over 2},{2\over 3})$.
More precisely we need $\rho>{1\over 2}$ so that the fractal uncertainty principle
can be applied (i.e. we get porosity down to some scale below~$\mathbf N^{-{1\over 2}}$),
and we need $\rho<{2\over 3}$ so that semiclassical calculus could be used
to get the almost orthogonality property.

There are several differences of the above outline from the actual proof of the key estimate. First of all, since we make the weaker spectral gap assumption~\eqref{e:spectral-gap} instead of~\eqref{e:spectral-gap-strong}, we need to revise the choice of $T_1$
and the parameters $\rho,\rho'$ in the calculus~-- see~\S\ref{s:propagation-times}. Secondly, at a certain point in the argument (see~\S\ref{s:reduction}) we pass from the toric quantization $\Op_{\mathbf N,\theta}$ to the Weyl quantization $\Op_h$ on $\mathbb R^n$, with $h:=(2\pi \mathbf N)^{-1}$.

\subsection{Structure of the paper}

\begin{itemize}
\item In~\S\ref{s:preliminaries} we review semiclassical quantization
and metaplectic operators on $\mathbb R^{2n}$ and on the torus $\mathbb T^{2n}$,
as well as the anisotropic symbol classes $S_{L,\rho,\rho'}$.
\item In~\S\ref{s:proofs-big} we prove Theorems~\ref{t:general}--\ref{t:measures}
modulo the key estimate given by Proposition~\ref{l:FUP_type}.
\item In~\S\ref{s:decay_long_words} we prove Proposition~\ref{l:FUP_type}
using the fractal uncertainty principle.
\item Finally, in Appendix~\ref{appendix:symplectic_matrices}, we discuss the algebraic properties of the spaces $V_+$ and $V_-$ and investigate the sharpness of Theorem~\ref{t:measures} in different situations.
\end{itemize}

\noindent\textbf{Notation:}
if $(F,\|\bullet\|_F)$ is a normed vector space
and $f_h\in F$ is a family depending on a parameter $h>0$, then
we say that $f_h=\mathcal O(h^{\alpha})_F$ if
$\|f_h\|_F=\mathcal O(h^{\alpha})$.
We write $f_h=\mathcal O_g (h^\alpha)_F$ if
the constant in $\mathcal O(\bullet)$ depends
on some additional parameter(s) $g$.

\section{Preliminaries}
\label{s:preliminaries}

\subsection{Semiclassical quantization}
\label{s:semi-quant}

Here we briefly review semiclassical quantization on~$\mathbb R^n$,
referring the reader to the book of Zworski~\cite{Zworski-Book} for details.

\subsubsection{Basic properties}
\label{s:semi-quant-basic}

Let $h\in (0,1]$ be the semiclassical parameter.
Throughout the paper we denote by $\Op_h$ the Weyl quantization
on $\mathbb R^n$.
For a Schwartz class symbol $a\in\mathscr S(\mathbb R^{2n})$ it can be defined as follows~\cite[\S4.1.1]{Zworski-Book}:
\begin{equation}
  \label{e:weyl-quantization}
\Op_h(a)f(x)=(2\pi h)^{-n}\int_{\mathbb R^{2n}} e^{{i\over h}\langle x-x',\xi\rangle}a\big(\textstyle{x+x'\over 2},\xi\big)f(x')\,dx'd\xi,\quad
f\in\mathscr S(\mathbb R^n). 
\end{equation}
One can extend this definition to $a$ belonging to the set of symbols
(see~\cite[\S4.4.1]{Zworski-Book})
\begin{equation}
  \label{e:S-1-def}
S(1):=\{a\in C^\infty(\mathbb R^{2n})\colon \sup|\partial^\alpha_{(x,\xi)}a|<\infty\ \text{for all multiindices}\ \alpha\}.
\end{equation}
A natural set of seminorms on $S(1)$ is given by
\begin{equation}
  \label{e:C-m-def}
\|a\|_{C^m}:=\max_{|\alpha|\leq m}\sup_{\mathbb R^{2n}}|\partial^\alpha_{(x,\xi)} a|,\quad
m\in\mathbb N_0.
\end{equation}
For any $a\in S(1)$, the operator $\Op_h(a)$ acts on the space of Schwartz functions $\mathscr S(\mathbb R^n)$
and on the dual space of tempered distributions $\mathscr S'(\mathbb R^n)$, see~\cite[Theorem~4.16]{Zworski-Book}.

We define the standard symplectic form $\sigma$ on $\mathbb R^{2n}$ by
\begin{equation}
  \label{e:symplectic-def}
\sigma(z,w):=\langle \xi,y\rangle-\langle x,\eta\rangle\quad\text{for all}\quad
z=(x,\xi),\ w=(y,\eta)\in\mathbb R^{2n}
\end{equation}
where $\langle\bullet,\bullet\rangle$ denotes the Euclidean inner product on~$\mathbb R^n$.

We now list several standard properties of the Weyl quantization:
\begin{itemize}
\item Composition formula~\cite[Theorems~4.11--4.12, 4.17--4.18]{Zworski-Book}: for $a,b\in S(1)$
\begin{equation}
  \label{e:basic-product}
\Op_h(a)\Op_h(b)=\Op_h(a\#b)\quad\text{for some $h$-dependent}\quad a\#b\in S(1),
\end{equation}
the function $a\#b$ is bounded in~$S(1)$ uniformly in~$h$
and satisfies the following asymptotic expansion as $h\to 0$ for all $N\in\mathbb N_0$:
\begin{equation}
  \label{e:Op-h-composition-expansion}
a\#b(z)=\sum_{k=0}^{N-1} {(-ih)^k\over 2^k k!}\big(\sigma(\partial_z,\partial_w)^k(a(z)b(w))\big)\big|_{w=z}+\mathcal O(h^N)_{S(1)}
\end{equation}
where $\sigma(\partial_z,\partial_w)=\langle\partial_\xi,\partial_y\rangle-\langle\partial_x,\partial_\eta\rangle$ is a second order differential operator in $z=(x,\xi)$, $w=(y,\eta)$.
Each $C^m$-seminorm of the $\mathcal O(h^N)$ remainder can be bounded
by $C\|a\|_{C^{N'}}\|b\|_{C^{N'}}h^N$ for some $C$ and $N'$ depending
only on $n,N,m$.
\item Adjoint formula~\cite[Theorem~4.1]{Zworski-Book}: if $a\in S(1)$, then
\begin{equation}
  \label{e:basic-adjoint}
\Op_h(a)^*=\Op_h(\overline a).
\end{equation}
\item $L^2$-boundedness~\cite[Theorem~4.23]{Zworski-Book}:
For $a\in S(1)$, the operator $\Op_h(a):L^2(\mathbb R^n)\to L^2(\mathbb R^n)$
is bounded uniformly in $h$, more precisely
its norm is bounded by $C\|a\|_{C^N}$ for some $C,N$ depending only on~$n$.
\item
Sharp G\r arding inequality~\cite[Theorem~4.32]{Zworski-Book}:
if $a\in S(1)$ and $a(x,\xi)\geq 0$ for all $(x,\xi)$, then
$$
\langle\Op_h(a) f,f\rangle_{L^2}\geq -C_ah\|f\|_{L^2}^2\quad\text{for all}\quad
f\in L^2(\mathbb R^n)
$$
where the constant $C_a$ has the form $C\|a\|_{C^N}$ for some $C,N$ depending only on~$n$.
\end{itemize}
The above properties imply in particular the nonintersecting support property
\begin{equation}
  \label{e:nonint-sup}
\Op_h(a)\Op_h(b)=\mathcal O(h^\infty)_{L^2(\mathbb R^n)\to L^2(\mathbb R^n)}\quad\text{if}\quad
a,b\in S(1),\quad
\supp a\cap\supp b=\emptyset.
\end{equation}
The next lemma gives an improved version of~\eqref{e:nonint-sup} in
a special case when $\supp a,\supp b$ are separated far from each other
in some direction. Lemma \ref{l:nonint-sup-1} is needed for the proof of Lemma \ref{l:nonint-sup-2}, which is itself used in the proof of Lemma \ref{l:non_stationary}.
\begin{lemm}
  \label{l:nonint-sup-1}
Let $q:\mathbb R^{2n}\to\mathbb R$ be a linear functional of norm~1
and assume that $a,b\in S(1)$ satisfy for some $r>0$
\begin{equation}
  \label{e:nonintsup-1}
\supp a\subset \{z\in\mathbb R^{2n}\mid q(z)\leq -r\},\quad
\supp b\subset \{z\in\mathbb R^{2n}\mid q(z)\geq r\}.
\end{equation}
Then for each $N>0$ we have
\begin{equation}
  \label{e:nonintsup-1-conc}
\|\Op_h(a)\Op_h(b)\|_{L^2(\mathbb R^n)\to L^2(\mathbb R^n)}\leq C_Nh^N(1+r)^{-N}
\end{equation}
where the constant $C_N$ depends only on $n$, $N$, and some seminorms
$\|a\|_{C^{N'}}$, $\|b\|_{C^{N'}}$ where $N'$ depends only on~$n,N$.
\end{lemm}
\begin{proof}
If $0<r<1$, then~\eqref{e:nonintsup-1-conc} follows
immediately from~\eqref{e:nonint-sup}. 
The constant $C_N$ does not depend on $r$ or~$q$: we only use that all the terms in the expansion~\eqref{e:Op-h-composition-expansion}
for~$a\# b$ are equal to 0.
Thus we henceforth assume that
$r\geq 1$.

\noindent 1. Fix a function
$$
F\in C^\infty(\mathbb R;(0,\infty)),\quad
F(\rho)=\begin{cases}
\rho,&\text{if}\quad \rho\geq 1;\\
|\rho|^{-1},&\text{if}\quad \rho\leq -1.
\end{cases}
$$
Then there exists a constant $C_0$ such that
\begin{equation}
  \label{e:nonintsup-1-1}
F(\rho')\leq C_0 F(\rho)(1+|\rho-\rho'|^2)\quad\text{for all}\quad
\rho,\rho'\in\mathbb R.
\end{equation}
Now define the function
$$
\mathbf m:=F\circ q\ \in\ C^\infty(\mathbb R^{2n};(0,\infty)).
$$
It follows from~\eqref{e:nonintsup-1-1} that
the function $\mathbf m$, as well as any its power $\mathbf m^s$,
is an \emph{order function} in the sense of~\cite[\S4.4.1]{Zworski-Book}.

Following~\cite[(4.4.2)]{Zworski-Book}, we define for an order function $\widetilde{\mathbf m}$
the class of symbols $S(\widetilde{\mathbf m})$ consisting of functions $a\in C^\infty(\mathbb R^{2n})$ satisfying the derivative bounds
$$
|\partial^\alpha_z a(z)|\leq C_\alpha \widetilde{\mathbf m}(z)\quad\text{for all}\quad
z\in\mathbb R^{2n}.
$$
The space $S(\widetilde{\mathbf m})$ has a natural family of seminorms defined
similarly to~\eqref{e:C-m-def}.

\noindent 2. Take arbitrary $N\in\mathbb N$.
From the support property~\eqref{e:nonintsup-1} we see that
$$
a=\mathcal O(r^{-N})_{S(\mathbf m^{-N})},\quad
b=\mathcal O(r^{-N})_{S(\mathbf m^N)}
$$
where the constants in $\mathcal O(\bullet)$ depend only
on some $S(1)$-seminorms of $a,b$. The composition formula~\eqref{e:Op-h-composition-expansion} is true for $a\in S(\mathbf m^{-N})$, $b\in S(\mathbf m^N)$
with the remainder in the expansion still in $S(1)$, see~\cite[Theorems~4.17--4.18]{Zworski-Book}.
Since $\supp a\cap\supp b=\emptyset$, all the terms in the asymptotic
expansion~\eqref{e:Op-h-composition-expansion} are zero, so in particular
$$
\Op_h(a)\Op_h(b)=\Op_h(a\#b)\quad\text{where}\quad a\# b=\mathcal O(h^Nr^{-2N})_{S(1)}.
$$
Now~\eqref{e:nonintsup-1-conc} follows from the $L^2$-boundedness property
of quantizations of symbols in $S(1)$.
\end{proof}

\subsubsection{Quantum translations}
\label{s:quantum-translations}

For $w\in \mathbb R^{2n}$, consider the operator
\begin{equation}
  \label{e:U-w-def}
U_w:=\Op_h(a_w)\quad\text{where}\quad a_w(z)=\exp\big(\textstyle{i\over h} \sigma(w,z)\big).
\end{equation}
Here the symbol $a_w$ lies in $S(1)$, though its seminorms are not bounded uniformly in~$h$.
By~\cite[Theorem~4.7]{Zworski-Book} we have
\begin{equation}
  \label{e:U-w-form}
U_wf(x)=e^{{i\over h}\langle \eta,x\rangle-{i\over 2h}\langle y,\eta\rangle}f(x-y),\quad
f\in \mathscr S'(\mathbb R^n)\quad\text{where}\quad w=(y,\eta).
\end{equation}
In particular, $U_w$ is a unitary operator on $L^2(\mathbb R^n)$.

We call $U_w$ a \emph{quantum translation} because it satisfies the following exact Egorov's Theorem (which is easy to check for $a\in\mathscr S(\mathbb R^{2n})$ using~\eqref{e:weyl-quantization} and extends to the general case by density):
\begin{equation}
  \label{e:conjugate-translation}
U_w^{-1}\Op_h(a)U_w=\Op_h(\tilde a)\quad\text{for all}\quad a\in S(1),\quad
\tilde a(z):=a(z+w).
\end{equation}
The map $w\mapsto U_w$ is not a group homomorphism, instead we have
\begin{equation}
  \label{e:U-w-prod}
U_wU_{w'}=e^{{i\over 2h}\sigma(w,w')}U_{w+w'}.
\end{equation}
This in particular implies the commutator formula
\begin{equation}
  \label{e:U-w-comm}
U_wU_{w'}=e^{{i\over h}\sigma(w,w')}U_{w'}U_w.  
\end{equation}

\subsubsection{Metaplectic transformations}
\label{s:metaplectic}

Denote by $\Sp(2n,\mathbb R)$ the group of real symplectic $2n\times 2n$ matrices,
i.e. linear isomorphisms $A:\mathbb R^{2n}\to\mathbb R^{2n}$ such that,
with the symplectic form $\sigma$ defined in~\eqref{e:symplectic-def},
$$
\sigma(Az,Aw)=\sigma(z,w)\quad\text{for all}\quad
z,w\in\mathbb R^{2n}.
$$
For each $A\in\Sp(2n,\mathbb R)$, denote by $\mathcal M_A$ the set of all
unitary transformations $M:L^2(\mathbb R^n)\to L^2(\mathbb R^n)$ satisfying
the following exact Egorov's theorem:
\begin{equation}
  \label{e:egorov-metaplectic}
M^{-1}\Op_h(a)M=\Op_h(a\circ A)\quad\text{for all}\quad a\in S(1).  
\end{equation}
Such transformations exist and are unique up to multiplication by a unit complex number,
see~\cite[Theorem~11.9]{Zworski-Book}. Moreover, they act
$\mathscr S(\mathbb R^n)\to\mathscr S(\mathbb R^n)$ and $\mathscr S'(\mathbb R^n)\to\mathscr S'(\mathbb R^n)$. The set
$$
\mathcal M:=\bigcup_{A\in\Sp(2n,\mathbb R)}\mathcal M_A
$$
is a subgroup of the group of unitary transformations of $L^2(\mathbb R^n)$
called the \emph{metaplectic group} and the map
$M\mapsto A$ is a group homomorphism $\mathcal M\to \Sp(2n,\mathbb R)$.

The metaplectic transformations and the quantum translations
are intertwined by the following corollary of~\eqref{e:egorov-metaplectic}:
\begin{equation}
  \label{e:metaplectic-shift}
M^{-1}U_wM=U_{A^{-1}w}\quad\text{for all}\quad
M\in \mathcal M_A,\
w\in\mathbb R^{2n}.
\end{equation}

\subsubsection{Symbol calculus associated to a coisotropic space}
\label{s:symbol-coisotropic}

We now introduce an exotic calculus corresponding to a linear foliation
by coisotropic spaces. For each subspace $L\subset \mathbb R^{2n}$, define
its symplectic complement $L^{\perp\sigma}$ as
\begin{equation}
  \label{e:perp-sigma-def}
L^{\perp\sigma}=\{w\in \mathbb R^{2n}\mid \sigma(w,z)=0\text{ for all }z\in L\}.
\end{equation}
Note that $L^{\perp\sigma}$ is a subspace of dimension $2n-\dim L$.

We call $L$ \emph{coisotropic} if $L^{\perp\sigma}\subset L$. This in particular
implies that $\dim L\geq n$. The case $\dim L=n$ corresponds
to $L$ being \emph{Lagrangian}, that is $L^{\perp\sigma}=L$.
On the opposite end, both the whole $\mathbb R^{2n}$ and any codimension 1 subspace
of it are coisotropic.

Our exotic calculus will use symbols in the class $S_{L,\rho,\rho'}$,
defined as follows:
\begin{defi}
\label{d:S-L}
Let $L\subset\mathbb R^{2n}$ be a coisotropic subspace. Fix
$$
0\leq\rho'\leq \rho\quad\text{such that}\quad \rho+\rho'<1.
$$
We say that an $h$-dependent symbol $a(x,\xi;h)\in C^\infty(\mathbb R^{2n})$ lies in
$S_{L,\rho,\rho'}(\mathbb R^{2n})$ if
for any choice of constant vector fields $X_1,\dots,X_k,Y_1,\dots,Y_m$ on $\mathbb R^{2n}$
such that $Y_1,\dots,Y_m$ are tangent to~$L$ there exists a constant $C$ such that
for all $h\in (0,1]$
\begin{equation}
  \label{e:S-L-def}
\sup_{\mathbb R^{2n}}|X_1\dots X_k Y_1\dots Y_m a|\leq Ch^{-\rho k-\rho'm}.
\end{equation}
\end{defi}
\Remarks
1. The derivative bounds~\eqref{e:S-L-def} can be interpreted as follows:
the symbol~$a$ can grow by~$h^{-\rho'}$ when differentiated along $L$ and by~$h^{-\rho}$
when differentiated in other directions.

\noindent 2. If $L$ is Lagrangian, then a version of the class $S_{L,\rho,\rho'}$ corresponding
to compactly supported symbols but an arbitrary (not necessarily constant) Lagrangian foliation previously appeared in~\cite{meassupp}
which inspired part of the argument in the present paper. In the
important special case $\rho'=0$ this class was introduced in~\cite{hgap}.

\noindent 3. In the case $\rho=\rho'<{1\over 2}$ the class $S_{L,\rho,\rho'}$ does not depend
on~$L$ and becomes the standard mildly exotic pseudodifferential class $S_\rho(1)$,
see~\cite[\S4.4.1]{Zworski-Book}. In particular, if $\rho=\rho'=0$ then we recover the class $S(1)$ defined
in~\eqref{e:S-1-def}.

\noindent 4. In~\cite{hgap,meassupp} the value of $\rho$ was taken close to~1 and $\rho'$ was either~0 or very small. In the present paper
we choose $\rho,\rho'$ in a more complicated way depending
on the size of the spectral gap of the matrix~$A$,
see~\S\ref{s:propagation-times}.

Each $a\in S_{L,\rho,\rho'}(\mathbb R^{2n})$ lies in $S(1)$ for any fixed value of $h>0$.
Therefore, the quantization $\Op_h(a)$ defines an operator
on $L^2(\mathbb R^n)$. The $S(1)$-seminorms of $a$ are not bounded uniformly
as $h\to 0$, so the standard properties of quantization from~\S\ref{s:semi-quant-basic}
do not apply. However, using that $\rho+\rho'<1$ and the fact that $L$ is coisotropic,
we can establish analogues of these properties with weaker remainders:
\begin{lemm}
  \label{l:quantization-properties}
1. Assume that $a,b\in S_{L,\rho,\rho'}(\mathbb R^{2n})$. Then $\Op_h(a)\Op_h(b)=\Op_h(a\#b)$
where $a\#b\in S_{L,\rho,\rho'}(\mathbb R^{2n})$ satisfies
the following asymptotic expansion as $h\to 0$ for all $N\in\mathbb N_0$:
\begin{equation}
  \label{e:S-L-prod}
a\#b(z)=\sum_{k=0}^{N-1} {(-ih)^k\over 2^k k!}\big(\sigma(\partial_z,\partial_w)^k(a(z)b(w))\big)\big|_{w=z}+\mathcal O(h^{(1-\rho-\rho')N})_{S_{L,\rho,\rho'}}.
\end{equation}

2. Assume that $a\in S_{L,\rho,\rho'}(\mathbb R^{2n})$. Then $\|\Op_h(a)\|_{L^2\to L^2}$ is bounded uniformly in $h\in (0,1]$.

3. Assume that $a\in S_{L,\rho,\rho'}(\mathbb R^{2n})$ and $a\geq 0$ everywhere. Then
there exists a constant~$C$ such that
\begin{equation}
  \label{e:S-L-Garding}
\langle \Op_h(a)f,f\rangle_{L^2}\geq -Ch^{1-\rho-\rho'}\|f\|_{L^2}^2\quad\text{for all}\quad f\in L^2(\mathbb R^n),\quad
0<h\leq 1.
\end{equation}
The constants in the above estimates depend only on certain $S_{L,\rho,\rho'}$-seminorms
of $a,b$ similarly to the properties of the $S(1)$-calculus in~\S\ref{s:semi-quant-basic},
where the choice of the seminorms additionally depends on $\rho,\rho'$.
\end{lemm}
\begin{proof}
1. Let $\dim L=n+p$. Since $L$ is coisotropic, there exists a linear symplectomorphism
\begin{equation}
  \label{e:L0-def}
A\in\Sp(2n,\mathbb R)\quad\text{such that}\quad
A^{-1}(L)=L_0:=\Span(\partial_{x_1},\dots,\partial_{x_n},\partial_{\xi_1},\dots,\partial_{\xi_p}).
\end{equation}
Denote $\xi=(\xi',\xi'')$ where $\xi'\in\mathbb R^p$
and $\xi''\in\mathbb R^{n-p}$.
For $a(x,\xi;h)\in C^\infty(\mathbb R^{2n})$, we have
$$
a\in S_{L,\rho,\rho'}(\mathbb R^{2n})\ \iff\
a\circ A\in S_{L_0,\rho,\rho'}(\mathbb R^{2n})
$$
where the space $S_{L_0,\rho,\rho'}(\mathbb R^{2n})$ can be characterized by the following
inequalities for all multiindices $\alpha,\beta$:
\begin{equation}
  \label{e:S-L0}
\sup_{\mathbb R^{2n}}|\partial^\alpha_{(x,\xi')}\partial^\beta_{\xi''}a|\leq C_{\alpha\beta} h^{-\rho'|\alpha|-\rho|\beta|}.
\end{equation}
Fix a metaplectic operator $M\in \mathcal M_A$. Then by~\eqref{e:egorov-metaplectic}
(which applies since the symbols in question are in $S(1)$ for any fixed~$h$)
$$
M^{-1}\Op_h(a)M=\Op_h(a\circ A)\quad\text{for all}\quad a\in S_{L,\rho,\rho'}(\mathbb R^{2n}).
$$
Since $M$ is unitary on $L^2(\mathbb R^n)$ and the terms in the expansion~\eqref{e:S-L-prod} are equivariant
under composing $a,b$ with~$A$, we reduce the proof of the lemma to the case
$L=L_0$.

\noindent 2. Henceforth we assume that $L=L_0$. We use the unitary rescaling operator
$T$ on~$L^2(\mathbb R^n)$ defined by
$$
T f(x)=h^{-{\rho' n\over 2}}f(h^{-\rho'}x).
$$
A direct calculation using~\eqref{e:weyl-quantization} shows that
for all $a\in S_{L_0,\rho,\rho'}(\mathbb R^{2n})$ we have
\begin{equation}
  \label{e:dilate-Op-h}
T^{-1} \Op_h(a) T=\Op_{\tilde h}(\tilde a)\quad\text{where}\quad
\tilde h:=h^{1-\rho-\rho'},\quad
\tilde a(x,\xi):=a(h^{\rho'}x,h^\rho \xi).
\end{equation}
By~\eqref{e:S-L0} we see that $\tilde a\in S(1)$ uniformly in~$h$.
This gives parts~2--3 of the lemma.

\noindent 3. To show part~1 of the lemma, define the symbol $a\widetilde\# b$ by
the formula $\Op_{\tilde h}(a)\Op_{\tilde h}(b)=\Op_{\tilde h}(a\widetilde\#b)$ and consider the rescaling
map $B(x,\xi)=(h^{\rho'}x,h^{\rho}\xi)$. Then by~\eqref{e:dilate-Op-h}
\begin{equation}
  \label{e:tilde-moyal}
(a\# b)\circ B=\tilde a\widetilde\#\tilde b\quad\text{where}\quad
\tilde a:=a\circ B,\quad
\tilde b:=b\circ B.
\end{equation}
Define the remainder
$$
r_N(z):=a\#b(z)-\sum_{k=0}^{N-1} {(-ih)^k\over 2^k k!}\big(\sigma(\partial_z,\partial_w)^k(a(z)b(w))\big)\big|_{w=z},
$$
then by~\eqref{e:tilde-moyal}
$$
r_N(B(z))=\tilde a\widetilde\# \tilde b(z)-\sum_{k=0}^{N-1} {(-i\tilde h)^k\over 2^k k!}\big(\sigma(\partial_z,\partial_w)^k(\tilde a(z)\tilde b(w))\big)\big|_{w=z}.
$$
Since $\tilde a,\tilde b$ are bounded in $S(1)$ uniformly in~$h$, we see from~\eqref{e:Op-h-composition-expansion} 
with $h$ replaced by~$\tilde h$ that
$$
r_N(B(z))=\mathcal O(\tilde h^N)_{S(1)}
$$
which implies that for all multiindices $\alpha,\beta$
\begin{equation}
  \label{e:zebra}
\sup_{\mathbb R^{2n}}|\partial^\alpha_x\partial^\beta_\xi r_N|
=\mathcal O(h^{(1-\rho-\rho')N-\rho'|\alpha|-\rho|\beta|}).
\end{equation}
If $p=0$, i.e. $L$ is Lagrangian, then this immediately gives the expansion~\eqref{e:S-L-prod}
as~\eqref{e:zebra} shows that $r_N=\mathcal O(h^{(1-\rho-\rho')N})_{S_{L_0,\rho,\rho'}}$.
In the general case, since $1-\rho-\rho'>0$ we get that all the derivatives of the remainder become rapidly decaying in $h$ when $N\to \infty$, that is for each $N_0,\alpha,\beta$
there exists $N$ such that
$\sup_{\mathbb R^{2n}}|\partial^\alpha_x\partial^\beta_\xi r_N|=\mathcal O(h^{N_0})$.
Combining this with the fact that the $k$-th term in~\eqref{e:S-L-prod}
is $\mathcal O(h^{(1-\rho-\rho')k})_{S_{L_0,\rho,\rho'}}$ we see
that the expansion~\eqref{e:S-L-prod} holds.
\end{proof}
\Remark
Lemma~\ref{l:quantization-properties} does not hold for the standard quantization
$$
\Op^0_h(a)f(x)=(2\pi h)^{-n}\int_{\mathbb R^{2n}}e^{{i\over h}\langle x-x',\xi\rangle}
a(x,\xi)f(x')\,dx'd\xi.
$$
Indeed, consider the case $n=1$ and $L=\Span(\partial_x+\partial_\xi)$. The analog
of the expansion~\eqref{e:S-L-prod} for the standard quantization is~\cite[Theorem~4.14]{Zworski-Book}
$$
\Op^0_h(a)\Op^0_h(b)=\Op^0_h(a\#^0b),\quad
a\#^0b(x,\xi)\sim\sum_{k=0}^\infty {(-ih)^k\over k!}\partial^k_\xi a(x,\xi)\partial^k_x b(x,\xi)\quad\text{as}\quad h\to 0.
$$
Take $\rho':=0$ and put $a(x,\xi)=\chi(h^{-\rho}(x-\xi))$
for some nonzero $h$-independent function $\chi\in\CIc(\mathbb R)$. Then $a\in S_{L,\rho,0}(\mathbb R^2)$ and
the $k$-th term in the asymptotic expansion for~$a\#^0a$ is
$$
{i^k\over k!}\big(\chi^{(k)}(h^{-\rho}(x-\xi))\big)^2 h^{(1-2\rho)k}.
$$
For ${1\over 2}<\rho<1$, each successive term in the expansion grows
faster in $h$ than the previous one, which makes it impossible for this expansion to hold.
The difference between the standard and the Weyl quantization exploited in the proof of Lemma~\ref{l:quantization-properties} is that the Weyl quantization obeys the exact Egorov Theorem~\eqref{e:egorov-metaplectic}
and the related fact that the terms in the asymptotic expansion~\eqref{e:S-L-prod}
are equivariant under $\Sp(2n,\mathbb R)$.

We also prove here a statement on composition of operators
whose symbols have well-separated supports, used in the proof of Lemma~\ref{l:non_stationary}:
\begin{lemm}
  \label{l:nonint-sup-2}
Let $q:\mathbb R^{2n}\to \mathbb R$ be a linear functional of norm~1
and assume that $a,b\in S_{L,\rho,\rho'}(\mathbb R^{2n})$ satisfy
for some $r_0\in\mathbb R$ and $r>0$
\begin{equation}
  \label{e:nonint-sup-2}
\supp a\subset \{z\in\mathbb R^{2n}\mid q(z)\leq r_0-r\},\quad
\supp b\subset \{z\in\mathbb R^{2n}\mid q(z)\geq r_0+r\}.
\end{equation}
Then for each $c\in S_{L,\rho,\rho'}(\mathbb R^{2n})$ and $N>0$ we have
\begin{equation}
  \label{e:nonint-sup-2-st}
\|\Op_h(a)\Op_h(c)\Op_h(b)\|_{L^2(\mathbb R^n)\to L^2(\mathbb R^n)}\leq
C_N h^N (1+r)^{-N}
\end{equation}
where the constant $C_N$ depends only on $n$, $N$, and some $n,N$-dependent $S_{L,\rho,\rho'}(\mathbb R^{2n})$-seminorms of $a,b,c$.
\end{lemm}
\begin{proof}
1. We first show that for all $a,b$ satisfying~\eqref{e:nonint-sup-2}
and all~$N$
\begin{equation}
  \label{e:nonint-sup-2-1}
\|\Op_h(a)\Op_h(b)\|_{L^2(\mathbb R^n)\to L^2(\mathbb R^n)}\leq C_N h^N (1+r)^{-N}.
\end{equation}
We may shift the supports of $a,b$ by conjugating by a quantum
translation (see~\eqref{e:conjugate-translation}), so we may assume that $r_0=0$.
We may also conjugate by a metaplectic transformation similarly to Step~1 in the proof of Lemma~\ref{l:quantization-properties} to reduce to the case $L=L_0$ where
$L_0$ is defined in~\eqref{e:L0-def}.
Next, using the dilation formula~\eqref{e:dilate-Op-h} we see that
$$
\|\Op_h(a)\Op_h(b)\|_{L^2(\mathbb R^n)\to L^2(\mathbb R^n)}=
\|\Op_{\tilde h}(\tilde a)\Op_{\tilde h}(\tilde b)\|_{L^2(\mathbb R^n)\to L^2(\mathbb R^n)}
$$
where $\tilde h=h^{1-\rho-\rho'}$ and the rescaled symbols $\tilde a(x,\xi)=a(h^{\rho'}x,h^\rho\xi)$,
$\tilde b(x,\xi)=b(h^{\rho'}x,h^\rho\xi)$ lie in~$S(1)$ uniformly in~$h$. The support condition~\eqref{e:nonint-sup-2} implies that
$$
\supp \tilde a\subset \{z\in\mathbb R^{2n}\mid \tilde q(z)\leq -\tilde r\},\quad
\supp \tilde b\subset \{z\in\mathbb R^{2n}\mid \tilde q(z)\geq \tilde r\}
$$
where we put
$$
q'(x,\xi):=q(h^{\rho'}x,h^\rho\xi),\quad
\tilde q:={q'\over\|q'\|},\quad
\tilde r:={r\over\|q'\|}.
$$
Note that $\|q'\|\leq 1$ and thus $\tilde r\geq r$.

We now apply Lemma~\ref{l:nonint-sup-1} to $\tilde a,\tilde b,\tilde h,\tilde r$
to get for all~$N$
$$
\|\Op_{\tilde h}(\tilde a)\Op_{\tilde h}(\tilde b)\|_{L^2(\mathbb R^n)\to L^2(\mathbb R^n)}\leq C_N \tilde h^N(1+\tilde r)^{-N}
\leq C_N h^{N(1-\rho-\rho')}(1+r)^{-N}
$$
which implies~\eqref{e:nonint-sup-2-1} since $1-\rho-\rho'>0$.

2. We now prove~\eqref{e:nonint-sup-2-st}. If $0<r<1$
then~\eqref{e:nonint-sup-2-st} follows by applying twice
the composition formula~\eqref{e:S-L-prod}
and using the $L^2$ boundedness property of the class $S_{L,\rho,\rho'}$
and the fact that $\supp a\cap\supp b=\emptyset$.
Henceforth we assume that $r\geq 1$. Fix
$$
\chi\in C^\infty(\mathbb R;[0,1]),\quad
\chi=1\text{ on }[\textstyle{1\over 2},\infty),\quad
\chi=0\text{ on }(-\infty,-\textstyle{1\over 2}]
$$
and decompose
$$
c=c_1+c_2,\quad
c_1(z):=c(z)\chi\Big({q(z)-r_0\over r}\Big),\quad
c_2(z):=c(z)\bigg(1-\chi\Big({q(z)-r_0\over r}\Big)\bigg).
$$
Then the symbols
$c_1,c_2$ are bounded in $S_{L,\rho,\rho'}(\mathbb R^{2n})$ uniformly in~$r$. Moreover
$$
\supp c_1\subset \{z\in\mathbb R^{2n}\mid q(z)\geq r_0-\textstyle{r\over 2}\},\quad
\supp c_2\subset \{z\in\mathbb R^{2n}\mid q(z)\leq r_0+\textstyle{r\over 2}\}.
$$
We now write $\Op_h(c)=\Op_h(c_1)+\Op_h(c_2)$ and estimate (using $L^2$ boundedness of~$\Op_h(a)$, $\Op_h(b)$)
$$
\|\Op_h(a)\Op_h(c)\Op_h(b)\|_{L^2\to L^2}\leq
C\big(\|\Op_h(a)\Op_h(c_1)\|_{L^2\to L^2}
+\|\Op_h(c_2)\Op_h(b)\|_{L^2\to L^2}\big).
$$
We finally use~\eqref{e:nonint-sup-2-1} with
$r_0$ replaced by $r_0\pm {3r\over 4}$ and
$r$ replaced by $r\over 4$
to get
$$
\|\Op_h(a)\Op_h(c_1)\|_{L^2\to L^2},
\|\Op_h(c_2)\Op_h(b)\|_{L^2\to L^2}
\leq C_N h^N (1+r)^{-N}
$$
which finishes the proof.
\end{proof}

\subsection{Quantization on the torus}

In this section we study quantizations of functions on the torus
$$
\mathbb T^{2n}:=\mathbb R^{2n}/\mathbb Z^{2n}.
$$
Each $a\in C^\infty(\mathbb T^{2n})$ can be identified with a $\mathbb Z^{2n}$-periodic function on $\mathbb R^{2n}$.
This function lies in the symbol class $S(1)$ defined in~\eqref{e:S-1-def}
and thus its Weyl quantization $\Op_h(a)$ is an operator on $L^2(\mathbb R^n)$.
We will decompose $L^2(\mathbb R^n)$ into a direct integral
of finite dimensional spaces $\mathcal H_{\mathbf N}(\theta)$, $\theta\in\mathbb T^{2n}$,
which we call the \emph{spaces of quantum states}. The operator $\Op_h(a)$ descends to these spaces and gives a quantization
of the observable~$a\in C^\infty(\mathbb T^{2n})$. Our presentation partially follows~\cite{Bouzouina-deBievre}.

To make sure that the spaces of quantum states are nontrivial, we henceforth make
the following assumption (see~\cite[Proposition~2.1]{Bouzouina-deBievre}):
\begin{equation}
  \label{e:h-N-def}
h={1\over 2\pi \mathbf N}\quad\text{for some}\quad \mathbf N\in\mathbb N.
\end{equation}

\subsubsection{The spaces of quantum states}
\label{s:quantum-states}

Recall the quantum translations $U_w$, $w\in\mathbb R^{2n}$, defined in~\S\ref{s:quantum-translations}.
By~\eqref{e:conjugate-translation} we have the commutation relations
\begin{equation}
  \label{e:U-w-Op-a}
\Op_h(a)U_w=U_w\Op_h(a)\quad\text{for all}\quad
a\in C^\infty(\mathbb T^{2n}),\quad
w\in\mathbb Z^{2n}.
\end{equation}
This motivates the following definition of the spaces of quantum states:
for each $\theta\in\mathbb T^{2n}$, 
put
\begin{equation}
  \label{e:H-N-theta}
\mathcal H_{\mathbf N}(\theta):=\{f\in \mathscr S'(\mathbb R^n)\mid U_w f=e^{2\pi i\sigma(\theta,w)+\mathbf N\pi iQ(w)}f\text{ for all }w\in\mathbb Z^{2n}\}
\end{equation}
where the quadratic form $Q$ on $\mathbb R^n$ is defined by
\begin{equation}
  \label{e:Q-def}
Q(w):=\langle y,\eta\rangle\quad\text{where}\quad w=(y,\eta)\in\mathbb R^{2n}.
\end{equation}
Denote
$$
\mathbb Z_{\mathbf N}:=\{0,\dots,\mathbf N-1\}.
$$
The following description of the spaces $\mathcal H_{\mathbf N}(\theta)$ is a higher
dimensional version of~\cite[Proposition~2.1]{Bouzouina-deBievre}:
\begin{lemm}
The space $\mathcal H_{\mathbf N}(\theta)$ is ${\mathbf N}^n$-dimensional with a basis given by
$\mathbf e^\theta_j$, $j\in \mathbb Z_{\mathbf N}^n$, where for $\theta=(\theta_x,\theta_\xi)\in\mathbb R^{2n}$ we define
\begin{equation}
  \label{e:e-kappa-def}
\mathbf e^\theta_j(x):=\mathbf N^{-{n\over 2}}\sum_{k\in\mathbb Z^n} e^{-2\pi i\langle \theta_\xi,k\rangle}
\delta\bigg(x-{\mathbf Nk+ j-\theta_x\over \mathbf N}\bigg),\quad  j\in \mathbb Z^n.
\end{equation}
\end{lemm}
\Remark The distributions $\mathbf e^\theta_j$ satisfy the identities
\begin{equation}
  \label{e:e-theta-equiv}
\begin{aligned}
\mathbf e^{\theta+w}_j&=\mathbf e^\theta_{j-y}\quad\text{for all}\quad w=(y,\eta)\in \mathbb Z^{2n},\\
\mathbf e^\theta_{j+\mathbf N\ell}&=e^{2\pi i\langle\theta_\xi,\ell\rangle}\mathbf e^\theta_j\quad\text{for all}\quad \ell\in\mathbb Z^n.
\end{aligned}
\end{equation}
In particular, even though the space $\mathcal H_{\mathbf N}(\theta)$ is canonically defined
for $\theta$ in the torus~$\mathbb T^{2n}$, in order to define the basis
$\{\mathbf e^\theta_j\}$ we need to fix a representative $\theta_x\in\mathbb R^{n}$.
Note also that $\mathbf e^\theta_j$ is supported on the shifted lattice
$\mathbf N^{-1}(j-\theta_x)+\mathbb Z^n$.
\begin{proof}
By~\eqref{e:U-w-comm}, for each $v\in\mathbb R^{2n}$ the quantum translation
$U_{v}$ is an isomorphism from $\mathcal H_{\mathbf N}(\theta)$ onto $\mathcal H_{\mathbf N}(\theta-\mathbf Nv)$.
On the other hand, we compute for all $j\in\mathbb Z^n$ and $v\in \mathbb R^{2n}$
\begin{equation}
  \label{e:U-w-e-theta}
U_v\mathbf e^\theta_j=e^{2\pi i\langle\eta,j-\theta_x\rangle+i\pi {\mathbf N}Q(v)}\mathbf e_j^{\theta-{\mathbf N}v}\quad\text{where}\quad
v=(y,\eta).
\end{equation}
Thus it suffices to consider the case $\theta=0$.

Using $w=(0,\ell)$ and $w=(\ell,0)$ in the definition~\eqref{e:H-N-theta},
as well as~\eqref{e:U-w-prod},
we can characterize $\mathcal H_{\mathbf N}(0)$ as the set of all $f\in\mathscr S'(\mathbb R^n)$ such that
\begin{equation}
  \label{e:hair-conditioner}
e^{2\pi i \mathbf N\langle \ell,x\rangle}f(x)=f(x),\quad
f(x-\ell)=f(x)\quad\text{for all}\quad
\ell\in\mathbb Z^n.
\end{equation}
The first condition in~\eqref{e:hair-conditioner} is equivalent to $f$ being a linear combination of delta functions at the points in the lattice~$\mathbf N^{-1}\mathbb Z^n$, that is
$$
f(x)=\sum_{r\in\mathbb Z^n}f_r \delta\Big(x-{r\over\mathbf N}\Big)\quad\text{for some}\quad
(f_r\in\mathbb C)_{r\in\mathbb Z^n}.
$$
The second condition in~\eqref{e:hair-conditioner} is then equivalent to
the periodic property $f_{r-\mathbf N\ell}=f_r$ for all $\ell\in\mathbb Z^n$.
It follows that $\mathcal H_{\mathbf N}(0)$ is the span of $\{\mathbf e^0_j\mid j\in\mathbb Z_{\mathbf N}^n\}$, which finishes the proof.
\end{proof}
We fix the inner product $\langle\bullet,\bullet\rangle_{\mathcal H}$ on each~$\mathcal H_{\mathbf N}(\theta)$ by requiring that
$\{\mathbf e^\theta_j\}_{j\in\mathbb Z^n_{\mathbf N}}$ be an orthonormal basis. Note that
while the basis $\{\mathbf e^\theta_j\}$ depends on the choice
of the representative $\theta_x\in\mathbb R^n$,
the inner product only depends on $\theta\in\mathbb T^{2n}$
as follows from~\eqref{e:e-theta-equiv}.

Using the bases $\{\mathbf e^\theta_j\}$, we can consider the spaces $\mathcal H_{\mathbf N}(\theta)$ as the fibers of a smooth $\mathbf N^n$-dimensional complex vector bundle over $\mathbb T^{2n}$, which we denote by $\mathcal H_{\mathbf N}$.

\subsubsection{Decomposing $L^2$}
\label{s:decomposing-L2}

We now construct a unitary isomorphism $\Pi_{\mathbf N}$ between $L^2(\mathbb R^n)$ 
and the space of $L^2$ sections $L^2(\mathbb T^{2n};\mathcal H_{\mathbf N})$.
This gives a decomposition of $L^2(\mathbb R^n)$ into the direct integral
of the spaces $\mathcal H_{\mathbf N}(\theta)$ over $\theta\in\mathbb T^{2n}$.

Define the operators $\Pi_{\mathbf N}(\theta):\mathscr S(\mathbb R^n)\to \mathcal H_{\mathbf N}(\theta)$ by
\begin{equation}
  \label{e:Pi-N-theta-def}
\Pi_{\mathbf N}(\theta)f:=\sum_{j\in \mathbb Z_{\mathbf N}^n} \langle f,\mathbf e^\theta_j\rangle_{L^2}\,\mathbf e^\theta_j,\quad \theta\in\mathbb T^{2n}.
\end{equation}
Even though the basis $\{\mathbf e_j^\theta\}$ depends on the choice
of the preimage $\theta_x\in\mathbb R^n$, the operator $\Pi_{\mathbf N}(\theta)$
does not depend on this choice as follows from~\eqref{e:e-theta-equiv}.
We next define the operator
$$
\Pi_{\mathbf N}:\mathscr S(\mathbb R^n)\to C^\infty(\mathbb T^{2n};\mathcal H_{\mathbf N}),\quad
\Pi_{\mathbf N}f:=(\Pi_{\mathbf N}(\theta)f)_{\theta\in\mathbb T^{2n}}.
$$
We also define the operator $\Pi_{\mathbf N}^*:C^\infty(\mathbb T^{2n};\mathcal H_{\mathbf N})\to \mathscr S(\mathbb R^n)$ as follows:
for $g\in C^\infty(\mathbb T^{2n};\mathcal H_{\mathbf N})$
$$
\begin{gathered}
\Pi_{\mathbf N}^* g(x):={\mathbf N}^{{n\over 2}}\int_{\mathbb T^n}\big\langle g(-\mathbf Nx,\theta_\xi),\mathbf e^{(-\mathbf Nx,\theta_\xi)}_0\big\rangle_{\mathcal H}\,d\theta_\xi,\quad
x\in\mathbb R^n.
\end{gathered}
$$
Here one can check that $\Pi_{\mathbf N}^*g\in \mathscr S(\mathbb R^n)$ using a non-stationary phase argument and the following corollary of~\eqref{e:e-theta-equiv}:
$$
\Pi_{\mathbf N}^* g(x-\ell)=\Pi_{\mathbf N}^*(e^{2\pi i\langle \ell,\theta_\xi\rangle} g)(x)\quad\text{for all}\quad \ell\in\mathbb Z^n.
$$
\begin{lemm}
  \label{l:L2-decomposed}
The map $\Pi_{\mathbf N}$ extends to a unitary isomorphism from $L^2(\mathbb R^{n})$ to $L^2(\mathbb T^{2n};\mathcal H_{\mathbf N})$
and $\Pi_{\mathbf N}^*$ extends to its adjoint.
\end{lemm}
\begin{proof}
We argue similarly to~\cite[Proposition~2.3]{Bouzouina-deBievre}.

\noindent 1. We first show that $\Pi_{\mathbf N}$ extends to
an isometry from $L^2(\mathbb R^{n})$ to $L^2(\mathbb T^{2n};\mathcal H_{\mathbf N})$.
For that it suffices to show the identity
\begin{equation}
  \label{e:L2-dec-1}
\int_{\mathbb T^{2n}}\|\Pi_{\mathbf N}(\theta)f\|_{\mathcal H}^2\,d\theta=\|f\|_{L^2(\mathbb R^n)}^2\quad\text{for all}\quad f\in\mathscr S(\mathbb R^{n}).
\end{equation}
For $j\in\mathbb Z^n$ and $\theta_x\in\mathbb R^n$, define the function
$F_{j,\theta_x}\in C^\infty(\mathbb T^n)$ by
$$
F_{j,\theta_x}(\theta_\xi)=\langle f,\mathbf e_j^\theta\rangle_{L^2}\quad\text{where}\quad
\theta=(\theta_x,\theta_\xi).
$$
Then $F_{j,\theta_x}$ can be written as a Fourier series:
$$
F_{j,\theta_x}(\theta_\xi)={\mathbf N}^{-{n\over 2}}\sum_{k\in\mathbb Z^n}f\Big({\mathbf Nk+j-\theta_x\over\mathbf N}\Big)e^{2\pi i\langle\theta_\xi,k\rangle}.
$$
Therefore by Parseval's Theorem
$$
\int_{\mathbb T^n}|F_{j,\theta_x}(\theta_\xi)|^2\,d\theta_\xi=\mathbf N^{-n}
\sum_{k\in\mathbb Z^n}\bigg|f\Big({\mathbf Nk+j-\theta_x\over\mathbf N}\Big)\bigg|^2.
$$
We then have
$$
\begin{aligned}
\int_{\mathbb T^{2n}}\|\Pi_{\mathbf N}(\theta)f\|_{\mathcal H}^2\,d\theta
&=\int_{[0,1]^n}\int_{\mathbb T^n}\sum_{j\in\mathbb Z_{\mathbf N}^n}|F_{j,\theta_x}(\theta_\xi)|^2\,d\theta_\xi d\theta_x
\\
&=\mathbf N^{-n}\sum_{j\in\mathbb Z_{\mathbf N}^n,k\in\mathbb Z^n}\int_{[0,1]^n}
\bigg|f\Big({\mathbf Nk+j-\theta_x\over\mathbf N}\Big)\bigg|^2\,d\theta_x\\
&=\mathbf N^{-n}\sum_{r\in\mathbb Z^n}\int_{[0,1]^n}
\bigg|f\Big({r-\theta_x\over\mathbf N}\Big)\bigg|^2\,d\theta_x
=\|f\|_{L^2(\mathbb R^n)}^2
\end{aligned}
$$
which gives~\eqref{e:L2-dec-1}.

\noindent 2. It remains to show that $\Pi_{\mathbf N}$ is onto
and $\Pi_{\mathbf N}^*$ is the adjoint of $\Pi_{\mathbf N}$. For that
it suffices to prove that for each $g\in C^\infty(\mathbb T^{2n};\mathcal H_{\mathbf N})$
we have $\Pi_{\mathbf N}\Pi_{\mathbf N}^*g=g$.
We compute for all~$\theta=(\theta_x,\theta_\xi)\in\mathbb R^{2n}$ and~$j\in\mathbb Z^n$
$$
\begin{aligned}
\langle\Pi_{\mathbf N}^*g,\mathbf e_j^\theta\rangle_{L^2}&=
\mathbf N^{-{n\over 2}}\sum_{k\in\mathbb Z^n}e^{2\pi i\langle\theta_\xi,k\rangle}
\Pi_{\mathbf N}^*g\Big({\mathbf Nk+j-\theta_x\over\mathbf N}\Big)
\\
&=\sum_{k\in\mathbb Z^n}e^{2\pi i\langle\theta_\xi,k\rangle}
\int_{\mathbb T^n} \langle g(\theta_x,\tilde\theta_\xi),\mathbf e_0^{(\theta_x-\mathbf Nk-j,\tilde\theta_\xi)}\rangle_{\mathcal H}\,d\tilde \theta_\xi
\\
&=\langle g(\theta),\mathbf e^\theta_j\rangle_{\mathcal H}.
\end{aligned}
$$
Here in the last line we use that $\mathbf e_0^{(\theta_x-\mathbf Nk-j,\tilde\theta_\xi)}=e^{2\pi i\langle\tilde\theta_\xi,k\rangle}\mathbf e_j^{(\theta_x,\tilde\theta_\xi)}$ by~\eqref{e:e-theta-equiv}, as well as convergence of the Fourier series
of the function $\theta_\xi\mapsto \langle g(\theta_x,\theta_\xi),\mathbf e^{(\theta_x,\theta_\xi)}_j\rangle_{\mathcal H}$. We now compute
$\Pi_{\mathbf N}(\theta)\Pi_{\mathbf N}^*g=g(\theta)$ for
all $\theta\in\mathbb T^{2n}$, finishing the proof.
\end{proof}
By duality, we may extend $\Pi_{\mathbf N},\Pi^*_{\mathbf N}$ to operators
$$
\Pi_{\mathbf N}:\mathscr S'(\mathbb R^n)\to \mathcal D'(\mathbb T^{2n};\mathcal H_{\mathbf N}),\quad
\Pi_{\mathbf N}^*:\mathcal D'(\mathbb T^{2n};\mathcal H_{\mathbf N})\to
\mathscr S'(\mathbb R^n).
$$
We then have the natural formula
\begin{equation}
  \label{e:delta-galore}
\Pi_{\mathbf N}^*(\delta(\theta-\theta_0)f)=f\quad\text{for all}\quad
\theta_0\in\mathbb T^{2n},\quad
f\in \mathcal H_{\mathbf N}(\theta_0)\subset\mathscr S'(\mathbb R^n),
\end{equation}
which follows by duality from the identity
\begin{equation*}
\begin{split}
\langle\delta(\theta-\theta_0)f,\Pi_{\mathbf N}\tilde f\rangle_{L^2(\mathbb T^{2n};\mathcal H_{\mathbf N})} & =\langle f,\Pi_{\mathbf N}(\theta_0)\tilde f\rangle_{\mathcal H}  = \sum_{j\in \mathbb Z_{\mathbf N}^n}   \langle f ,\mathbf e^{\theta_0}_j \rangle_{\mathcal{H}}
\,\langle \mathbf e^{\theta_0}_j,\tilde f\rangle_{L^2}
  \\
& =\langle f,\tilde f\rangle_{L^2(\mathbb R^n)}\quad\text{for all}\quad
\tilde f\in\mathscr S(\mathbb R^n).
\end{split}
\end{equation*}

\subsubsection{Semiclassical quantization}
\label{s:quantization-on-torus}

Fix $\mathbf N\in\mathbb N$ and put $h:=(2\pi\mathbf N)^{-1}$ as before.
Let $a\in C^\infty(\mathbb T^{2n})$.
By~\eqref{e:U-w-Op-a} and~\eqref{e:H-N-theta}, the operator $\Op_h(a)$
maps each of the spaces $\mathcal H_{\mathbf N}(\theta)$
to itself. This defines the quantizations
$$
\Op_{\mathbf N,\theta}(a):=\Op_h(a)|_{\mathcal H_{\mathbf N}(\theta)}:\mathcal H_{\mathbf N}(\theta)\to\mathcal H_{\mathbf N}(\theta),\quad
\theta\in\mathbb T^{2n}
$$
which depend smoothly on~$\theta$.

A special case is given by $a(x,\xi)=a(x)$ which is independent of $\xi$ and $\mathbb Z^n$-periodic in~$x$. In this case $\Op_h(a)$ is the multiplication operator by~$a$
(see~\cite[Theorem~4.3]{Zworski-Book}), so by~\eqref{e:e-kappa-def}
\begin{equation}
  \label{e:Op-a-special}
\Op_{\mathbf N,\theta}(a)\mathbf e^\theta_j=a\Big({j-\theta_x\over\mathbf N}\Big) \mathbf e^\theta_j\quad\text{for all}\quad
\theta=(\theta_x,\theta_\xi)\in\mathbb R^{2n},\quad
j\in\mathbb Z^n.
\end{equation}
In particular, $\Op_{\mathbf N,\theta}(1)$ is the identity operator on $\mathcal H_{\mathbf N}(\theta)$.

Let $\Pi_{\mathbf N},\Pi_{\mathbf N}^*$ be the unitary operators
constructed in~\S\ref{s:decomposing-L2}. By~\eqref{e:delta-galore}, they
relate the operator $\Op_h(a):\mathscr S'(\mathbb R^n)\to \mathscr S'(\mathbb R^n)$
to its restrictions $\Op_{\mathbf N,\theta}(a)$ as follows:
\begin{equation}
  \label{e:relator}
\Pi_{\mathbf N}\Op_h(a)\Pi_{\mathbf N}^* g(\theta)=\Op_{\mathbf N,\theta}(a)g(\theta)\quad\text{for all}\quad g\in \mathcal D'(\mathbb T^{2n};\mathcal H_{\mathbf N}).
\end{equation}
Notice that \eqref{e:relator} may also be deduced from the explicit expression
(verified by an explicit computation using~\eqref{e:U-w-form}, \eqref{e:e-kappa-def}, \eqref{e:Pi-N-theta-def}, and the Poisson summation formula;
the series below converges in~$\mathscr S'(\mathbb R^n)$)
\begin{equation}
   \label{e:explicit_projector}
\Pi_{\mathbf N}(\theta) f = \sum_{w \in \mathbb{Z}^{2n}} e^{i \pi \mathbf{N} Q(w) - 2 i \pi \sigma(\theta,w)} U_w f \quad \textup{ for all } \quad f \in \mathscr{S}(\R^n)
\end{equation}
and the commutation identity~\eqref{e:U-w-Op-a}.

Since $\Op_{\mathbf N,\theta}(a)$ depends smoothly on~$\theta$, it follows from \eqref{e:relator} and Lemma~\ref{l:L2-decomposed} that
\begin{equation}
  \label{e:op-norm-decomposed}
\max_{\theta\in\mathbb T^{2n}}\|\Op_{\mathbf N,\theta}(a)\|_{\mathcal H_{\mathbf N}(\theta)\to\mathcal H_{\mathbf N}(\theta)}=\|\Op_h(a)\|_{L^2(\mathbb R^n)\to L^2(\mathbb R^n)}.
\end{equation}
Recall from Definition~\ref{d:S-L} the symbol class $S_{L,\rho,\rho'}(\mathbb R^{2n})$ 
where $L\subset\mathbb R^{2n}$ is a coisotropic subspace and $0\leq \rho'\leq\rho$,
$\rho+\rho'<1$. We similarly define the corresponding
symbol class on the torus
$$
S_{L,\rho,\rho'}(\mathbb T^{2n})
$$
whose elements are the $\mathbb Z^{2n}$-periodic symbols in $S_{L,\rho,\rho'}(\mathbb R^{2n})$.
Note that putting $\rho=\rho'=0$ we obtain the standard symbol class
$S(\mathbb T^{2n})$ consisting of functions in $C^\infty(\mathbb T^{2n})$ with all derivatives
bounded uniformly in~$h$.

Using~\eqref{e:relator} and~\eqref{e:op-norm-decomposed},
we see that Lemma~\ref{l:quantization-properties} applies
to the quantization $\Op_{\mathbf N,\theta}(a)$. In particular,
we have the product formula for all $a,b\in S_{L,\rho,\rho'}(\mathbb T^{2n})$
\begin{equation}
  \label{e:our-product}
\Op_{\mathbf N,\theta}(a)\Op_{\mathbf N,\theta}(b)=\Op_{\mathbf N,\theta}(a\# b)
\end{equation}
where $a\#b\in S_{L,\rho,\rho'}(\mathbb T^{2n})$ satisfies the expansion~\eqref{e:S-L-prod},
the adjoint formula (following from~\eqref{e:basic-adjoint})
\begin{equation}
  \label{e:our-adjoint}
\Op_{\mathbf N,\theta}(a)^*=\Op_{\mathbf N,\theta}(\bar a),
\end{equation}
the norm $\|\Op_{\mathbf N,\theta}(a)\|_{\mathcal H_{\mathbf N}(\theta)\to\mathcal H_{\mathbf N}(\theta)}$ is bounded by some $S_{L,\rho,\rho'}$-seminorm of~$a$,
and we have the sharp G\r arding inequality
for all $a\in S_{L,\rho,\rho'}(\mathbb T^{2n})$
such that $a\geq 0$ everywhere
\begin{equation}
  \label{e:our-garding}
\langle \Op_{\mathbf N,\theta}(a) f,f\rangle_{\mathcal H}\geq -C\mathbf N^{\rho+\rho'-1}\|f\|_{\mathcal H}^2\quad\text{for all}\quad f\in \mathcal H_{\mathbf N}(\theta)
\end{equation}
where $C$ is some $S_{L,\rho,\rho'}$-seminorm of~$a$.
The choice of seminorms above depends on~$\rho,\rho'$ but not on $\mathbf N$ or~$\theta$. The inequality \eqref{e:our-garding} follows from the usual sharp G\r arding inequality, and Lemma \ref{l:L2-decomposed} that implies that $\Op_{\mathbf N,\theta}(a)$ is self-adjoint and that its spectrum is contained in the $L^2$ spectrum of $\Op_h(a)$.

We now give several corollaries of the basic calculus above.
First of all, from~\eqref{e:our-product}, the expansion~\eqref{e:S-L-prod} with $N=1$,
and the boundeness of the operator norm of $\Op_{\mathbf N,\theta}(\bullet)$
we get for all $a,b\in S_{L,\rho,\rho'}(\mathbb T^{2n})$
\begin{equation}
  \label{e:our-product-2}
\Op_{\mathbf N,\theta}(a)\Op_{\mathbf N,\theta}(b)=\Op_{\mathbf N,\theta}(ab)
+\mathcal O(\mathbf N^{\rho+\rho'-1})_{\mathcal H_{\mathbf N}(\theta)\to\mathcal H_{\mathbf N}(\theta)}
\end{equation}
where the constant in $\mathcal O(\bullet)$ depends only on some
$S_{L,\rho,\rho'}$-seminorms of $a,b$.

Next, we have the following inequality of norms:
\begin{lemm}
\label{l:advanced-norm}
Assume that $a,b\in S_{L,\rho,\rho'}(\mathbb T^{2n})$ and $|a|\leq |b|$ everywhere.
Then we have for all $f\in \mathcal H_{\mathbf N}(\theta)$
\begin{equation}
  \label{e:advanced-norm}
\|\Op_{\mathbf N,\theta}(a)f\|_{\mathcal H}
\leq \|\Op_{\mathbf N,\theta}(b)f\|_{\mathcal H}+C\mathbf N^{\rho+\rho'-1\over 2}\|f\|_{\mathcal H}
\end{equation}
where the constant $C$ depends only on some $S_{L,\rho,\rho'}$-seminorms of~$a,b$.
\end{lemm}
\Remark Taking $b$ to be the constant symbol $b:=\sup|a|$, we get
\begin{equation}
  \label{e:basic-norm}
\|\Op_{\mathbf N,\theta}(a)\|_{\mathcal H_{\mathbf N}(\theta)\to\mathcal H_{\mathbf N}(\theta)}\leq \sup_{z\in\mathbb T^{2n}}|a(z)|+ C\mathbf N^{\rho+\rho'-1\over 2}
\end{equation}
where $C$ only depends on some $S_{L,\rho,\rho'}$-seminorm of~$a$.
\begin{proof}
By~\eqref{e:our-adjoint} and~\eqref{e:our-product-2} we have
\begin{equation}
  \label{e:advn-1}
\begin{gathered}
\Op_{\mathbf N,\theta}(b)^*\Op_{\mathbf N,\theta}(b)
-\Op_{\mathbf N,\theta}(a)^*\Op_{\mathbf N,\theta}(a)\\
=\Op_{\mathbf N,\theta}(|b|^2-|a|^2)
+\mathcal O(\mathbf N^{\rho+\rho'-1})_{\mathcal H_{\mathbf N}(\theta)\to\mathcal H_{\mathbf N}(\theta)}.
\end{gathered}
\end{equation}
Since $|a|\leq |b|$ everywhere, we have $|b|^2-|a|^2\geq 0$ everywhere, so by~\eqref{e:our-garding}
\begin{equation}
  \label{e:advn-2}
\langle \Op_{\mathbf N,\theta}(|b|^2-|a|^2)f,f\rangle_{\mathcal H}\geq
-C\mathbf N^{\rho+\rho'-1}\|f\|_{\mathcal H}^2.
\end{equation}
Together~\eqref{e:advn-1} and~\eqref{e:advn-2} give
$$
\|\Op_{\mathbf N,\theta}(b)f\|_{\mathcal H}^2-
\|\Op_{\mathbf N,\theta}(a)f\|_{\mathcal H}^2\geq 
-C\mathbf N^{\rho+\rho'-1}\|f\|_{\mathcal H}^2
$$
which implies~\eqref{e:advanced-norm}.
\end{proof}
Finally, we record here the following lemma regarding products of many quantized
observables, which is analogous to~\cite[Lemmas~A.1 and~A.6]{meassupp}:
\begin{lemm}\label{l:long_product}
Assume that $a_1,\dots,a_R\in S_{L,\rho,\rho'}(\mathbb T^{2n})$,
where $R\leq C_0\log\mathbf N$, satisfy $\sup_{\mathbb T^{2n}}|a_j|\leq 1$
and each $S_{L,\rho,\rho'}$-seminorm of $a_j$ is bounded uniformly in~$j$.
Then:

\noindent 1. The product $a_1\cdots a_R$ lies in $S_{L,\rho+\varepsilon,\rho'+\varepsilon}(\mathbb T^{2n})$
for all small $\varepsilon>0$.

\noindent 2. We have for all $\varepsilon>0$
$$
\Op_{\mathbf N,\theta}(a_1)\cdots \Op_{\mathbf N,\theta}(a_R)=
\Op_{\mathbf N,\theta}(a_1\cdots a_R)+\mathcal O(\mathbf N^{\rho+\rho'-1+\varepsilon})_{\mathcal H_{\mathbf N}(\theta)\to\mathcal H_{\mathbf N}(\theta)}.
$$
Here the constant in $\mathcal O(\bullet)$ depends only on $\rho,\rho',\varepsilon,C_0$,
and on the maximum over~$j$ of a certain $S_{L,\rho,\rho'}$-seminorm of~$a_j$.
\end{lemm}
\begin{proof}
1. We have $\sup_{\mathbb T^{2n}}|a_1\cdots a_R|\leq 1$, so it suffices to show
that for each constant vector fields $X_1,\dots,X_k,Y_1,\dots,Y_m$ on $\mathbb T^{2n}$
such that $Y_1,\dots,Y_m$ are tangent to~$L$ we have
\begin{equation}
  \label{e:multiplicator-int-1}
\sup_{\mathbb T^{2n}}|X_1\dots X_k Y_1\dots Y_m(a_1\cdots a_R)|
=\mathcal O(h^{-\rho k-\rho'm-}).
\end{equation}
Using the Leibniz Rule, we write $X_1\dots X_k Y_1\dots Y_m(a_1\cdots a_R)$
as a sum of $R^{m+k}=\mathcal O(h^{0-})$ terms. Each of these is a product
of the form $(P_1 a_1)\cdots (P_Ra_R)$ where each $P_j$ is a product
of some subset of $X_1,\dots,X_k,Y_1,\dots,Y_m$ and $P_1\dots P_R=X_1\dots X_k Y_1\dots Y_m$.
Using that $\sup |a_j|\leq 1$ and each $a_j$ is bounded in $S_{L,\rho,\rho'}$ uniformly in~$j$,
we get~\eqref{e:multiplicator-int-1}.

\noindent 2. We write
$$
\Op_{\mathbf N,\theta}(a_1)\cdots \Op_{\mathbf N,\theta}(a_R)-
\Op_{\mathbf N,\theta}(a_1\cdots a_R)=\sum_{j=1}^R B_j\Op_{\mathbf N,\theta}(a_{j+1})\cdots \Op_{\mathbf N,\theta}(a_R)
$$
where $B_j:=\Op_{\mathbf N,\theta}(a_1\cdots a_{j-1})\Op_{\mathbf N,\theta}(a_j)-\Op_{\mathbf N,\theta}(a_1\cdots a_j)$. By~\eqref{e:basic-norm} and since $\sup|a_j|\leq 1$ we have
$\|\Op_{\mathbf N,\theta}(a_j)\|_{\mathcal H_{\mathbf N}(\theta)\to\mathcal H_{\mathbf N}(\theta)}\leq 1+C\mathbf N^{\rho+\rho'-1\over 2}$ where the constant~$C$
is uniform in~$j$. Therefore (assuming $\mathbf N$ is large enough)
$$
\|\Op_{\mathbf N,\theta}(a_1)\cdots \Op_{\mathbf N,\theta}(a_R)-
\Op_{\mathbf N,\theta}(a_1\cdots a_R)\|_{\mathcal H_{\mathbf N}(\theta)\to\mathcal H_{\mathbf N}(\theta)}\leq 2\sum_{j=1}^R \|B_j\|_{\mathcal H_{\mathbf N}(\theta)\to\mathcal H_{\mathbf N}(\theta)}.
$$
It remains to use that $a_1\cdots a_j$ is bounded in $S_{L,\rho+\varepsilon,\rho'+\varepsilon}(\mathbb T^{2n})$ uniformly in~$j$ for all~$\varepsilon>0$, so by~\eqref{e:our-product-2} we have
$\|B_j\|_{\mathcal H_{\mathbf N}(\theta)\to\mathcal H_{\mathbf N}(\theta)}=\mathcal O(\mathbf N^{\rho+\rho'-1+})$ uniformly in~$j$.
\end{proof}

\subsubsection{Quantization of toric automorphisms}
\label{s:quantization-toric-auto}

Let $\Sp(2n,\mathbb Z)\subset \Sp(2n,\mathbb R)$ be the subgroup of integer symplectic matrices, i.e. elements
of $\Sp(2n,\mathbb R)$ which preserve the lattice $\mathbb Z^{2n}$.
We will quantize elements of $\Sp(2n,\mathbb Z)$ as unitary operators
on the spaces $\mathcal H_{\mathbf N}(\theta)$ provided that
$\mathbf N,\theta$ satisfy the condition~\eqref{e:theta-quantize-condition} below.
To do this we need the following
\begin{lemm}
  \label{l:parity}
Denote by
$\mathbb Z_2:=\mathbb Z/2\mathbb Z$ the field of order~2.
Then for each $A\in\Sp(2n,\mathbb Z)$ there exists unique $\varphi_A\in \mathbb Z_2^{2n}$ such that, with the symplectic form $\sigma$ defined in~\eqref{e:symplectic-def}
and the quadratic form $Q$ defined in~\eqref{e:Q-def}
\begin{equation}
  \label{e:parity}
Q(A^{-1}w)-Q(w)=\sigma(\varphi_A,w)\mod 2\mathbb Z\quad\text{for all}\quad w\in\mathbb Z^{2n}.
\end{equation}
\end{lemm}
\Remark Note that the map $A\mapsto\varphi_A$ satisfies for all $A,B\in\Sp(2n,\mathbb Z)$
$$
\varphi_{AB}=\varphi_A+A\varphi_B,\quad
\varphi_{A^{-1}}=A^{-1}\varphi_A.
$$
\begin{proof}
For $w\in\mathbb Z^{2n}$, denote
$$
Z(w):=\big(Q(A^{-1}w)-Q(w)\big)\bmod 2\mathbb Z\ \in\ \mathbb Z_2.
$$
We have
$$
Q(w+w')=Q(w)+Q(w')+\sigma(w,w')\mod 2\mathbb Z\quad\text{for all}\quad w,w'\in\mathbb Z^{2n}
$$
which together with the fact that $A\in\Sp(2n,\mathbb Z)$ implies that
$$
Z(w+w')=Z(w)+Z(w')\quad\text{for all}\quad
w,w'\in\mathbb Z^{2n}.
$$
Thus $Z$ is a group homomorphism $\mathbb Z^{2n}\to\mathbb Z_2$,
which gives the existence and uniqueness of $\varphi_A$ such that~\eqref{e:parity} holds.
\end{proof}
Now, fix $A\in\Sp(2n,\mathbb Z)$ and choose a metaplectic operator (see~\S\ref{s:metaplectic})
$$
M\in\mathcal M_A.
$$
Here we put $h:=(2\pi\mathbf N)^{-1}$ as before.
Using~\eqref{e:metaplectic-shift} and~\eqref{e:H-N-theta} we see that
$$
M(\mathcal H_{\mathbf N}(\theta))\subset \mathcal H_{\mathbf N}\big(A\theta+\textstyle{\mathbf N\varphi_A\over 2}\big)\quad\text{for all}\quad
\theta\in\mathbb T^{2n}.
$$
Denote
$$
M_{\mathbf N,\theta}:=M|_{\mathcal H_{\mathbf N}(\theta)}:
\mathcal H_{\mathbf N}(\theta)\to\mathcal H_{\mathbf N}\big(A\theta+\textstyle{\mathbf N\varphi_A\over 2}\big)
$$
which depends smoothly on $\theta\in\mathbb T^{2n}$.
Using~\eqref{e:delta-galore} or \eqref{e:explicit_projector}, we see that the action of $M$
on~$L^2(\mathbb R^n)$ is intertwined with the operators $M_{\mathbf N,\theta}$
as follows:
\begin{equation}\label{e:intertwining}
\Pi_{\mathbf N}M\Pi_{\mathbf N}^* g\big(A\theta+\textstyle{\mathbf N\varphi_A\over 2}\big)=M_{\mathbf N,\theta}\,g(\theta)\quad\text{for all}\quad
g\in\mathcal D'(\mathbb T^{2n};\mathcal H_{\mathbf N}).
\end{equation}
Since $M$ is unitary on $L^2(\mathbb R^n)$, it follows
that each $M_{\mathbf N,\theta}$ is a unitary operator as well.

We will be interested in the spectrum of $M_{\mathbf N,\theta}$, so we need
its domain and range to be the same space $\mathcal H_{\mathbf N}(\theta)$. This is true if
we choose $\theta\in\mathbb T^{2n}$ such that the following \emph{quantization condition} holds:
\begin{equation}
  \label{e:theta-quantize-condition}
(I-A)\theta={\mathbf N\varphi_A\over 2}\mod \mathbb Z^{2n}.
\end{equation}
Note that when $\mathbf N$ is even or $\varphi_A=0$, the equation~\eqref{e:theta-quantize-condition} is satisfied in particular when~$\theta=0$.

Assuming~\eqref{e:theta-quantize-condition}, from~\eqref{e:egorov-metaplectic} we get the following exact Egorov's Theorem:
\begin{equation}
  \label{e:our-egorov}
M_{\mathbf N,\theta}^{-1}\Op_{\mathbf N,\theta}(a)M_{\mathbf N,\theta}
=\Op_{\mathbf N,\theta}(a\circ A)\quad\text{for all}\quad
a\in C^\infty(\mathbb T^{2n}).
\end{equation}

\subsubsection{Explicit formulas}
\label{s:explicit-formulas}

Here we give some explicit formulas for the operators $\Op_{\mathbf N,\theta}(a)$
and $M_{\mathbf N,\theta}$. These are not used in the proofs but are
helpful for implementing numerics. For simplicity we assume in this subsection that
$\theta=0$.

For $f\in\mathcal H_{\mathbf N}(0)$ define the coordinates
$$
f_j:=\langle f,\mathbf e^0_j\rangle_{\mathcal H},\quad
j\in\mathbb Z^n
$$
where $\mathbf e^0_j$ is defined in~\eqref{e:e-kappa-def}.
By~\eqref{e:e-theta-equiv} we have $f_{j+\mathbf N\ell}=f_j$
for all $\ell\in\mathbb Z^n$.

Our first statement computes the expression $\langle\Op_{\mathbf N,0}(a)f,f\rangle_{\mathcal H}$ in terms of the values of $a$ at the points in the lattice ${1\over 2\mathbf N}\mathbb Z^{2n}$. As before we define
$\mathbb Z_{\mathbf N}:=\{0,\dots,\mathbf N-1\}$
and similarly $\mathbb Z_{2\mathbf N}:=\{0,\dots,2\mathbf N-1\}$.
\begin{prop}\label{p:explicit_pseudors}
Let $a\in C^\infty(\mathbb T^{2n})$. Then for all $f\in\mathcal H_{\mathbf N}(0)$ we have
\begin{equation}
  \label{e:Op-h-paired}
\langle\Op_{\mathbf N,0}(a)f,f\rangle_{\mathcal H}=(2\mathbf N)^{-n}
\sum_{j\in \mathbb Z_{\mathbf N}^n,\  k, \ell\in \mathbb Z_{2\mathbf N}^n}
e^{{\pi i\over \mathbf N}\langle  k, \ell\rangle}
a\Big({2j+ k\over 2\mathbf N},{ \ell\over 2\mathbf N}\Big)f_{j}\overline{f_{j+ k}}.
\end{equation}
\end{prop}
\begin{proof}
Since trigonometric polynomials are dense in $C^\infty(\mathbb T^{2n})$, it suffices to consider the case of
$$
a(z)=e^{2\pi i\sigma(w,z)}\quad\text{for some}\quad
w=(y,\eta)\in \mathbb Z^{2n}.
$$
Using~\eqref{e:U-w-def}, \eqref{e:U-w-e-theta}, and~\eqref{e:e-theta-equiv}, we compute for all $j\in\mathbb Z^n$
$$
\Op_{\mathbf N,0}(a)\mathbf e^0_j=U_{w/\mathbf N}\,\mathbf e^0_j=e^{{2\pi i\over \mathbf N}\langle \eta,j\rangle+{\pi i\over \mathbf N}\langle y,\eta\rangle}\mathbf e^0_{j+y}.
$$
Therefore the left-hand side of~\eqref{e:Op-h-paired} is
$$
\langle\Op_{\mathbf N,0}(a)f,f\rangle_{\mathcal H}=\sum_{j\in\mathbb Z_{\mathbf N}^n} f_j\langle\Op_{\mathbf N,0}(a)\mathbf e_j^0,f\rangle_{\mathcal H}
=\sum_{j\in \mathbb Z_{\mathbf N}^n} e^{{2\pi i\over \mathbf N}\langle \eta,j\rangle+{\pi i\over \mathbf N}\langle y,\eta\rangle} f_j\overline{f_{j+y}}.
$$
On the other hand, the right-hand side of~\eqref{e:Op-h-paired} is equal to
$$
(2\mathbf N)^{-n}\sum_{j\in \mathbb Z_{\mathbf N}^n,\
k, \ell\in \mathbb Z_{2\mathbf N}^n}
e^{{\pi i\over \mathbf N}(\langle k-y,\ell\rangle+\langle\eta,2j+k\rangle)}f_j\overline{f_{j+k}}.
$$
The sum over $\ell$ is equal to~0 unless $k-y\in 2\mathbf N\mathbb Z^n$
which happens
for exactly one value of $k\in\mathbb Z_{2\mathbf N}^n$. Using that
$f_{j+k}=f_{j+y}$ for this~$k$ we write the right-hand side of~\eqref{e:Op-h-paired} as
$$
\sum_{j\in \mathbb Z_{\mathbf N}^n}
e^{{\pi i\over \mathbf N}\langle \eta,2j+y\rangle}f_j\overline{f_{j+y}}.
$$
This equals the left-hand side of~\eqref{e:Op-h-paired} which finishes the proof.
\end{proof}
Proposition~\ref{p:explicit_pseudors} can be interpreted as follows:
$$
\langle\Op_{\mathbf N,0}(a)f,f\rangle_{\mathcal H}
=(2\mathbf N)^{-2n}\sum_{p,q\in\mathbb Z_{2\mathbf N}^n}a\Big({p\over2\mathbf N},
{q\over 2\mathbf N}\Big) W(f)_{pq}
$$
where the \emph{Wigner matrix} $W(f)_{pq}$ of $f$ is given by
\begin{equation}
  \label{e:Wigner}
W(f)_{pq}=(2\mathbf N)^n\sum_{j\in\mathbb Z_{\mathbf N}^n}
e^{{\pi i\over\mathbf N}\langle p-2j,q\rangle}f_j \overline{f_{p-j}},\quad
p,q\in\mathbb Z_{2\mathbf N}^n.
\end{equation}

We now compute the action of metaplectic transformations
$M_{\mathbf N,0}:\mathcal H_{\mathbf N}(0)\to\mathcal H_{\mathbf N}(0)$
where $M\in \mathcal M_A$ for some $A\in\Sp(2n,\mathbb Z)$ and
we assume that $\mathbf N$ is even so that~\eqref{e:theta-quantize-condition}
is satisfied for $\theta=0$ and all~$A$. The general formulas are complicated,
so instead we follow the approach of~\cite[\S1.2.1]{Kelmer-cat}.
As proved for instance in~\cite{Kohnen} the group $\Sp(2n,\mathbb Z)$ is
generated by matrices of the following block form, where
$E^{-T}$ denotes the transpose of $E^{-1}$:
\begin{equation}
  \label{e:special-sp}
\begin{aligned}
S_B&:=\begin{pmatrix} I & 0 \\ B & I\end{pmatrix},\quad
B\text{ is a symmetric $n\times n$ integer matrix};\\
L_E&:=\begin{pmatrix} E & 0 \\ 0 & E^{-T}\end{pmatrix},\quad
E\in \GL(n,\mathbb Z),\quad |\det E|=1;\\
F&:=\begin{pmatrix} 0 & I \\ -I & 0 \end{pmatrix}.
\end{aligned}
\end{equation}
Since the map $M\in\mathcal M\mapsto A\in\Sp(2n,\mathbb R)$ such that
$M\in\mathcal M_A$ is a group homomorphism, it suffices to compute
the operators $M_{\mathbf N,0}$ for $M\in\mathcal M_A$ where $A$
is in one of the forms~\eqref{e:special-sp}. This is done in
\begin{prop}\label{p:explicit_quantization}
Assume that $\mathbf N$ is even. Then
there exist
\begin{align}
\label{e:explicit-meta-1}
M\in\mathcal M_{S_B}&\quad\text{such that}\quad (M_{\mathbf N,0}f)_j=e^{{\pi i\over \mathbf N}\langle Bj,j\rangle}f_j,\\
\label{e:explicit-meta-2}
M\in\mathcal M_{L_E}&\quad\text{such that}\quad (M_{\mathbf N,0}f)_j=f_{E^{-1}j},\\
\label{e:explicit-meta-3}
M\in\mathcal M_{F}&\quad\text{such that}\quad (M_{\mathbf N,0}f)_{j}=\mathbf N^{-{n\over 2}}\sum_{k\in \mathbb Z_{\mathbf N}^{n}}e^{-{2\pi i\over \mathbf N}\langle j,k\rangle}f_{k}
\end{align}
for all $f\in \mathcal H_\mathbf N(0)$ and $j\in\mathbb Z^n$ where $f_j:=\langle f,\mathbf e^0_j\rangle_{\mathcal H}$.
\end{prop}
\Remark The evenness assumption on $\mathbf N$ is only required in~\eqref{e:explicit-meta-1}.
The formulas~\eqref{e:explicit-meta-2}--\eqref{e:explicit-meta-3} are valid
for all~$\mathbf N$ and one can check that $\varphi_{L_E}=\varphi_{F}=0$ where
$\varphi_A$ is defined in~\eqref{e:parity}.
\begin{proof}
This follows from the definition~\eqref{e:e-kappa-def}
of $\mathbf e^0_j$, the Poisson summation formula (in case of~\eqref{e:explicit-meta-3}),
and the following formulas for metaplectic transformations
for which~\eqref{e:egorov-metaplectic} can be verified directly using~\eqref{e:weyl-quantization}:
\begin{itemize}
\item we have $M\in\mathcal M_{S_B}$ where
$$
Mf(x)=e^{{i\over 2h}\langle Bx,x\rangle}f(x);
$$
\item we have $M\in\mathcal M_{L_E}$ where 
$$
Mf(x)=f(E^{-1}x);
$$
\item
we have $\mathcal F_h\in \mathcal M_{F}$ where
\begin{equation}
  \label{e:F-h-def-2}
\mathcal F_hf(x)=(2\pi h)^{-{n\over 2}}\int_{\mathbb R^{n}}e^{-{i\over h}\langle x,\eta\rangle}f(\eta)\,d\eta.
\end{equation}
\end{itemize}
\end{proof}

\section{Proofs of Theorems~\ref{t:general}--\ref{t:measures}}
\label{s:proofs-big}

In this section, we reduce the proofs of Theorems~\ref{t:general} and~\ref{t:damped} to a decay statement for long words, Proposition~\ref{l:FUP_type}, that will be proved in~\S \ref{s:decay_long_words}. In \S\ref{s:words}, we introduce notation that will be used in the proofs of Theorems \ref{t:general} and \ref{t:damped} and state the main estimates that we will need to write these proofs: Lemmas~\ref{l:controlled_support},
\ref{l:controlled_damped}, and~\ref{l:uncontrolled}. In \S \ref{s:proofs}, we explain how these estimates allow us to prove Theorems~\ref{t:general} and~\ref{t:damped}. 
In~\S\ref{s:proof-measures}, we derive Theorem~\ref{t:measures} from Theorem~\ref{t:general}.
In~\S \ref{s:controlled}, we prove Lemmas \ref{l:controlled_support} and \ref{l:controlled_damped}. Finally, in \S \ref{s:reduction}, we reduce Lemma \ref{l:uncontrolled} to Proposition~\ref{l:FUP_type}.

Our strategy in this section generally follows~\cite{meassupp,Jin-DWE,varfup}.
However, the proofs have to be modified to adapt to the setting of quantum maps
used here and to the assumption of geometric control transversally to~$\mathbb T_\pm$.
Another important difference is the choice of propagation times, see~\S\ref{s:propagation-times}.

Throughout this section we fix $A \in \Sp(2n,\Z)$ that satisfies the spectral gap condition~\eqref{e:spectral-gap} and choose a metaplectic operator $M \in \mathcal{M}_A$ quantizing $A$. 
We take $\mathbf N$ large and $\theta\in\mathbb T^{2n}$ satisfying the quantization condition~\eqref{e:theta-quantize-condition} and study the restrictions
$M_{\mathbf N,\theta}$ of $M$ to the spaces of quantum states
$\mathcal H_{\mathbf N}(\theta)$, see~\S\ref{s:quantization-toric-auto}.
Following~\eqref{e:h-N-def} we put $h:=(2\pi\mathbf N)^{-1}$.

\subsection{Words decomposition}\label{s:words}

In the proofs of both Theorems \ref{t:general} and \ref{t:damped}, we will consider two $\mathbf N$-independent functions $b_1,b_2$ on the torus~$\mathbb T^{2n}$. The choice of these functions will differ in the proof of each theorem, but we will always assume that they satisfy
\begin{equation}
\label{e:b-assumptions}
b_1,b_2\in C^\infty(\mathbb T^{2n}),\quad
|b_1|+|b_2|\leq 1.
\end{equation}
The functions $b_1$ and $b_2$ are supposed given for now, we will explain in the proofs of Theorems~\ref{t:general} and~\ref{t:damped} how they are constructed. We will always explicitly point out when specific properties of $b_1$ and $b_2$ are required.

Let us write $b := b_1 + b_2$ and take the quantizations (see~\S\ref{s:quantization-on-torus})
\begin{equation*}
B := \Op_{\mathbf N,\theta}(b),\quad B_1 := \Op_{\mathbf N,\theta}(b_1),\quad B_2 := \Op_{\mathbf{N},\theta}(b_2).
\end{equation*}
For any operator $L$ on $\mathcal{H}_{\mathbf{N}}(\theta)$ and $T \in \Z$, we define the conjugated operator
\begin{equation}
\label{e:conjugation-def}
L(T) := M_{\mathbf{N},\theta}^{-T}\, L\, M^T_{\mathbf{N},\theta}:
\mathcal H_{\mathbf N}(\theta)\to \mathcal H_{\mathbf N}(\theta).
\end{equation}
For $m \in \N$, we introduce the set of words
\begin{equation*}
\begin{split}
\mathcal{W}(m) := \set{1,2}^m.
\end{split}
\end{equation*}
For $\w = w_0 \dots w_{m-1} \in \mathcal{W}(m)$, define the operator
\begin{equation}
\label{e:operator_bw}
\begin{split}
B_{\w} \coloneqq B_{w_{m-1}}(m-1) \cdots B_{w_1}(1) B_{w_0}(0)
\end{split}
\end{equation}
and the corresponding symbol
\begin{equation}\label{e:symbole_bw}
\begin{split}
b_\w \coloneqq \prod_{j = 0}^{m-1} b_{w_j} \circ A^j.
\end{split}
\end{equation}
To a function $c : \mathcal{W}(m) \to \C$, we associate the operator
\begin{equation}
  \label{e:B-c-def}
B_c \coloneqq \sum_{\w \in \mathcal{W}(m)} c(\w)B_{\w}
\end{equation}
and accordingly the symbol
\begin{equation*}
\begin{split}
b_c \coloneqq \sum_{\w \in \mathcal{W}(m)} c(\w)b_{\w}.
\end{split}
\end{equation*}
If $c=\indic_E$ is the characteristic function of a subset $E\subset\mathcal{W}(m)$, then we simply write $B_E$ instead of $B_{\indic_E}$. Notice then that we have
\begin{equation}\label{e:total_B}
B_{\mathcal{W}(m)} = 
B(m-1)\cdots B(1)B(0) =
M_{\mathbf{N},\theta}^{-(m-1)} \p{BM_{\mathbf{N},\theta}}^{m-1} B.
\end{equation}
In the proof of Theorem \ref{t:general}, we will have $B = I$, and thus $B_{\mathcal{W}(m)}=I$ as well. On the other hand, Theorem \ref{t:damped} will be deduced from estimates on large powers of $BM_{\mathbf{N},\theta}$ that will follow from estimates on $B_{\mathcal{W}(m)}$ for $m$ large.

\subsubsection{Propagation times}
\label{s:propagation-times}

We need now to fix a few parameters that will be used to choose a relevant value of~$m$. 
Recall from~\S\ref{s:more-results}
that by the spectral gap assumption~\eqref{e:spectral-gap}, $A$ has two
particular simple eigenvalues
\begin{equation}
  \label{e:lambda-pm-def}
\lambda_\pm\in\mathbb R,\quad
\lambda_-=\lambda_+^{-1},\quad
|\lambda_-|<1<|\lambda_+|
\end{equation}
and all other eigenvalues~$\lambda$ of $A$ are contained in the open annulus
$|\lambda_-|<|\lambda|<|\lambda_+|$.
Fix a constant $\gamma$ such that
$$
1<\gamma<|\lambda_+|,\quad
\Spec(A)\setminus \{\lambda_+,\lambda_-\}\subset \{\lambda\colon \gamma^{-1}<|\lambda|<\gamma\}.
$$
Define the hyperplanes
\begin{equation}
  \label{e:L-pm-def}
L_\mp:=\Range(A-\lambda_\pm I)\ \subset\ \mathbb R^{2n}.
\end{equation}
Note that $L_\mp$ are coisotropic (since they have codimension~1), invariant under~$A$ and, denoting by $E_\pm$ the eigenspaces of $A$ corresponding to $\lambda_\pm$,
\begin{equation}
  \label{e:L-pm-sum}
\R^{2n} = E_+\oplus L_-  = E_- \oplus L_+.
\end{equation}
Moreover, since $A\in\Sp(2n,\mathbb R)$ we have (where $\sigma$ denotes the symplectic form from~\eqref{e:symplectic-def})
\begin{equation}
  \label{e:L-pm-perp}
\sigma(z,w)=0\quad\text{for all}\quad
z\in E_\pm,\quad
w\in L_\pm.
\end{equation}
Next, $L_\mp\otimes\mathbb C$ is the sum of all generalized eigenspaces of~$A$ with eigenvalues not equal to~$\lambda_\pm$. Therefore
we have the spectral radius bounds
\begin{equation}
  \label{e:L-pm-norm}
\begin{aligned}
\big\|A^{\pm m}\big\|&=\mathcal O(|\lambda_+|^m)\quad\text{as}\quad m\to\infty,\\
\big\|A^{\pm m}|_{L_\mp}\big\|&=\mathcal O(\gamma^m)\quad\text{as}\quad m\to \infty.
\end{aligned}
\end{equation}
Let us now fix two numbers $\rho,\rho' \in (0,1)$ such that
\begin{equation}\label{e:conditions_rho}
\rho + \rho' < 1, \quad  \rho \frac{\log \gamma}{\log \va{\lambda_+}}<\rho'<{1\over 2}<\rho.
\end{equation}
We also fix an integer $J$ that satisfies
\begin{equation}\label{e:condition_J}
J > 1 + \frac{2 \log 2}{\log \va{\lambda_+}}.
\end{equation}
We now set
\begin{equation}\label{e:def_times}
T_0 := \left\lfloor \frac{\rho \log \mathbf{N}}{J \log \va{\lambda_+}} \right\rfloor \quad \textup{ and } \quad T_1 := J T_0\approx
{\rho\log\mathbf N\over\log|\lambda_+|}.
\end{equation}
We call $T_0$ the \emph{short logarithmic propagation time} and $T_1$
the \emph{long logarithmic propagation time}.

\Remark Let us give some explanations regarding the choice of $\rho,\rho'$
and the propagation times $T_0,T_1$:
\begin{itemize}
\item We will use the semiclassical symbol classes $S_{L_\pm,\rho,\rho'}(\mathbb T^{2n})$ introduced in~\S\ref{s:quantization-on-torus} and for that we need $0\leq \rho'\leq \rho$
and $\rho+\rho'<1$. In particular, this is used in Lemma~\ref{l:long_logarithmic_words}
and, most crucially, in Lemma~\ref{l:non_stationary}.
\item From~\eqref{e:conditions_rho} and~\eqref{e:def_times} we see that
when $\mathbf N$ is large,
\begin{equation}
  \label{e:special-times-condition}
|\lambda_+|^{T_1}\sim \mathbf N^{\rho},\quad
\gamma^{T_1}\ll\mathbf N^{\rho'}.
\end{equation}
These inequalities are used to show that the symbols $b_{\w}$, $\w\in\mathcal W(T_1)$,
lie in the class $S_{L_-,\rho+\varepsilon,\rho'+\varepsilon}(\mathbb T^{2n})$,
see Lemma~\ref{l:long_logarithmic_words}.
They are also used in the proof of the porosity property,
Lemma~\ref{l:porosity}.
\item The requirement $\rho>{1\over 2}$ ensures
that the support of $b_{\w}$,
$\w\in\mathcal W(T_1)$,
is porous in the direction of the eigenvector of $A$ with eigenvalue $\lambda_+$
on scales (almost) up to $h^{\rho}\ll h^{1/2}$,
so that the fractal uncertainty principle can be applied~--
see Proposition~\ref{p:fractal_uncertainty_principle},
Lemma~\ref{l:porosity}, and the last step of the proof of Lemma~\ref{l:coltar_fup}
in~\S\ref{s:fup-endgame}.
\item The inequality~\eqref{e:condition_J} on~$J$ ensures that
the errors coming from the exotic calculus $S_{L_-,{\rho\over J},{\rho'\over J}}$ decay faster
than the growth of the number of elements in~$\mathcal W(T_0)$. More precisely,
it makes the remainders in~\eqref{e:controlled_support} and~\eqref{e:estimate_BZ}
decay as a negative power of~$\mathbf N$.
\item It is also useful to consider what happens in degenerate
cases. Assume first that all the eigenvalues in $\Spec(A)\setminus \{\lambda_+,\lambda_-\}$
lie on the unit circle (this is true in particular if $n=1$).
Then we could take $\rho$ to be any fixed number
in $({1\over 2},1)$ and $\rho'$ to be any sufficiently small positive number.
This is the choice made in~\cite{meassupp}
(which additionally took $\rho$ close to~1).
\item On the other hand, if $\gamma$ is close to $|\lambda_+|$ (i.e. the
spectral gap of $A$ is small) then our conditions
force $\rho'<{1\over 2}<\rho$ to both be close to~$1\over 2$.
\end{itemize}

\subsubsection{Partition into controlled/uncontrolled words and main estimates}
\label{s:partition-control}

We now decompose the operator $B_{\mathcal W(m)}$, $m:=2T_1$, into
the sum of two operators corresponding to the
controlled and uncontrolled region (see~\eqref{e:decomposition} below), and state the main estimates used
in the proofs of Theorems~\ref{t:general} and~\ref{t:damped}.
 
Let $F: \mathcal{W}(T_0) \to \left[0,1\right]$ be the function that gives the proportion of the digit $1$ in a word, that is for $\w = w_0 \dots w_{T_0 - 1}$ we have
\begin{equation}\label{e:definition_F}
F(\w) = \frac{\# \set{j \in \set{0, \dots, T_0 - 1}: w_j= 1}}{T_0}.
\end{equation}
Let then $\alpha > 0$ be very small (small enough so that Lemma \ref{l:uncontrolled} below holds) and define
\begin{equation}\label{e:def_Z}
\mathcal{Z}:= \set{\w \in \mathcal{W}(T_0): F(\w) \geq \alpha}.
\end{equation}
We call elements of $\mathcal Z$ \emph{controlled short logarithmic words}.

We next use the set $\mathcal{Z}$ to split $\mathcal{W}\p{2 T_1}$ into two subsets:
$\mathcal W(2T_1)=\mathcal X\sqcup \mathcal Y$,
writing each word in $\mathcal W(2T_1)$ as a concatenation
of $2J$ words in $\mathcal W(T_0)$:
\begin{align}
\label{e:def_X}
\mathcal{X} &\,:= \set{ \w = \w^{(1)} \dots \w^{(2J)} : \w^{(\ell)} \in \mathcal{W}(T_0) \setminus \mathcal{Z} \textup{ for all } \ell = 1,\dots, 2J};\\
\label{e:def_Y}
\mathcal{Y} &\,:= \set{\w = \w^{(1)} \dots \w^{(2J)}: \w^{(\ell)} \in \mathcal{Z} \textup{ for some } \ell \in \set{1,\dots,2J}}.
\end{align}
We call elements of $\mathcal X$ \emph{uncontrolled long logarithmic words}
and elements of $\mathcal Y$ \emph{controlled long logarithmic words}.

It follows from \eqref{e:total_B} that (with the operators $B_{\mathcal X},B_{\mathcal Y}$
defined using~\eqref{e:B-c-def})
\begin{equation}\label{e:decomposition}
M_{\mathbf{N},\theta}^{-(2 T_1-1)} (BM_{\mathbf{N},\theta})^{2 T_1 -1} B = 
B_{\mathcal W(2T_1)}=B_{\mathcal{X}} + B_{\mathcal{Y}}.
\end{equation}
In order to get an estimate on $M_{\mathbf{N},\theta}^{-(2 T_1-1)} (BM_{\mathbf{N},\theta})^{2 T_1 -1} B$, we will control $B_{\mathcal{X}}$ and $B_{\mathcal{Y}}$ separately. Let us start with $B_{\mathcal{Y}}$. It will be dealt with differently in the proofs of Theorem \ref{t:general} and \ref{t:damped}, but the main idea is the same: we make an assumption on the symbol~$b_1$ that translates into control on the operator~$B_1$ that is inherited by $B_{\mathcal{Y}}$. In more practical terms, we will use the following lemma in the proof of Theorem \ref{t:general}.
\begin{lemm}\label{l:controlled_support}
Let $a \in C^{\infty}\p{\T^{2n}}$. Assume that $b_1,b_2\in C^\infty(\mathbb T^{2n})$ satisfy
\begin{equation}
b_1,b_2\geq 0,\quad
b_1+b_2=1,\quad
\supp b_1\subset \bigcup_{m\in\mathbb Z}A^m\big(\{a\neq 0\}\big).
\end{equation}
Let $\varepsilon > 0$. Then there is a constant $C > 0$ such that for all
$\mathbf N$
and $u \in \mathcal{H}_{\mathbf{N}}(\theta)$
\begin{equation}\label{e:controlled_support}
\n{B_{\mathcal{Y}} u}_{\mathcal{H}} \leq C \n{\Op_{\mathbf N, \theta}(a) u}_{\mathcal H} + C \mathbf{r}_M(u) \log \mathbf{N} + C \mathbf{N}^{ - \frac{1}{2} + \frac{1}{2J}\big(1 + \frac{2 \log 2}{\log \va{\lambda_+}}\big) + \varepsilon} \n{u}_{\mathcal{H}}. 
\end{equation} 
\end{lemm}
Here we recall that the quantity $\mathbf{r}_M(u)$, defined in~\eqref{e:quasimode},
measures how close $u$ is to an eigenfunction of~$M_{\mathbf N,\theta}$.
Note that the condition \eqref{e:condition_J} on $J$ ensures that the power of $\mathbf{N}$ in
the last term on the right-hand side of~\eqref{e:controlled_support} is negative for
$\varepsilon$ small enough.

The proof of Theorem \ref{t:damped} will rely on the following estimate instead of Lemma \ref{l:controlled_support}.
\begin{lemm}\label{l:controlled_damped}
Assume that $b_1,b_2$ satisfy~\eqref{e:b-assumptions} and
\begin{equation}
  \label{e:b-assumptions-dwe}
|b_1|+|b_2| < 1\quad\text{on}\quad
\supp b_1.
\end{equation}
Then there exist $C,\delta > 0$ such that for every $\mathbf{N}$ we have
\begin{equation}\label{e:controlled_damped}
\n{B_{\mathcal{Y}}}_{\mathcal{H} \to \mathcal{H}} \leq C \mathbf{N}^{- \delta}.
\end{equation}
\end{lemm}
The proofs of Lemmas~\ref{l:controlled_support}--\ref{l:controlled_damped} can be found in~\S \ref{s:controlled}.

The control on $B_{\mathcal{X}}$ will be more subtle to obtain. We will use the following estimate, whose proof ultimately relies on the Fractal Uncertainty Principle. Recall the subtori $\mathbb T_\pm\subset \mathbb T^{2n}$ defined in~\eqref{e:T-pm-def}, and make the following
\begin{defi}
\label{d:safe}
Let $U\subset \mathbb T^{2n}$ be a set. We say that $U$ is \emph{safe}
if and only if for every $x\in\mathbb T^{2n}$, each
of the shifted tori $x+\mathbb T_+$, $x+\mathbb T_-$ intersects~$U$.
\end{defi}
Notice that being safe is slightly more restrictive than satisfying the geometric control condition with respect to $\mathbb{T}_+$ and $\mathbb{T}_-$ (Definition \ref{d:geometric_control_condition}) as we do not have here the flexibility to replace a point $x$ by its image by an iterate of $A$.

The control on $B_{\mathcal X}$ is achieved in
\begin{lemm}\label{l:uncontrolled}
Assume that $b_1,b_2$ satisfy~\eqref{e:b-assumptions} and the complements
$\mathbb T^{2n}\setminus \supp b_1$,
$\mathbb T^{2n}\setminus \supp b_2$
are both safe. Assume also that the constant $\alpha$ in~\eqref{e:def_Z} is chosen small enough.
Then there are constants $C,\delta > 0$ such that for every $\mathbf{N}$ we have
\begin{equation*}
\n{B_{\mathcal{X}}}_{\mathcal{H} \to \mathcal{H}} \leq C \mathbf{N}^{- \delta}.
\end{equation*}
\end{lemm}
The proof of Lemma \ref{l:uncontrolled} is reduced to a decay
result for long logarithmic words, Proposition~\ref{l:FUP_type}, in~\S\ref{s:reduction}.
The proof of the latter result is the point of~\S\ref{s:decay_long_words}.

\subsection{Proofs of Theorems \ref{t:general} and \ref{t:damped}}\label{s:proofs}

Let us explain now how Lemmas \ref{l:controlled_support}--\ref{l:uncontrolled} allow us to prove Theorems \ref{t:general} and \ref{t:damped}. First of all, we need to construct the functions $b_1$ and~$b_2$ that appear in the statement of these lemmas. To do so, we will use the following
two lemmas:
\begin{lemm}\label{l:safe-cpct}
Assume that $V\subset\T^{2n}$ is a safe open set.
Then $V$ contains a safe compact subset~$K$.
\end{lemm}
\begin{proof}
We argue by contradiction: we write $V=\bigcup_{m\in\mathbb N}K_m$ where $K_m$ are
compact sets with $K_m\subset K_{m+1}^\circ$ and assume that none of the $K_m$'s are safe.
Then for every $m \in \N$ there exist
$$
x_m \in \T^{2n},\quad
\sigma_m \in \set{+,-}\quad\text{such that}\quad
(x_m + \T_{\sigma_m})\cap K_m=\emptyset.
$$
Up to extracting a subsequence, we may assume that $(\sigma_m)_{m \in \N}$ is constant equal to some $\sigma \in \set{+,-}$ and that $(x_m)_{m \in \N}$ converges to some point $x \in \T^{2n}$.
Since $V$ is safe, the set $x+\mathbb T_\sigma$ intersects $V$.
Take $y\in (x+\mathbb T_\sigma)\cap V$ and put
$y_m:=x_m-x+y\in x_m+\mathbb T_\sigma$. Then $y_m\to y\in V$,
so $y_m\in K_m$ for $m$ large enough.
This gives a contradiction with our assumption that $x_m+\mathbb T_\sigma$ does not intersect~$K_m$.
\end{proof}

\begin{lemm}\label{l:partition_unity}
Let $U$ be an open subset of $\T^{2n}$
which is safe in the sense of Definition~\ref{d:safe}.
Then there exist $a_1,a_2\in C^\infty(\mathbb T^{2n})$ such that
$$
a_1,a_2\geq 0,\quad
a_1+a_2=1,\quad
\supp a_1\subset U
$$
and the complements $\mathbb T^{2n}\setminus \supp a_1$,
$\mathbb T^{2n}\setminus\supp a_2$ are both safe.
\end{lemm}
\begin{proof}
\noindent 1. We show that there exist two compact sets
\begin{equation}
  \label{e:K-1-2}
K_1,K_2\subset\mathbb T^{2n}\quad\text{such that}\quad
K_1\cap K_2=\emptyset,\quad
K_1\subset U,
\end{equation}
and $K_1,K_2$ are both safe. To do this, let $H\subset\mathbb R^{2n}$ be a hyperplane
transverse to each of the tangent spaces $V_\pm\otimes\mathbb R$ of $\mathbb T_\pm$
where $V_\pm\subset\mathbb Q^{2n}$ are the subspaces defined in the paragraph
preceding~\eqref{e:T-pm-def}. Denote by $\pi:\mathbb R^{2n}\to\mathbb T^{2n}$ the projection map.
Take large $R>0$,
denote by $B_H(R)$ the closed ball of radius~$R$ in $H$, and define
$$
D_R:=\pi(B_H(R))\ \subset\ \mathbb T^{2n}.
$$
Then, we can fix $R$ large enough so that the set $D_R$ is safe.
Indeed, every element of $\mathbb T^{2n}$ can be written
as $\pi(x)$ for some $x\in [0,1]^{2n}\subset\mathbb R^{2n}$.
Then we can decompose $x=x'_\pm+x''_\pm$
where $x'_\pm\in H$, $x''_\pm\in V_\pm\otimes\mathbb R$.
Moreover, if $R$ is large enough then we can choose this decomposition
so that $x'_\pm\in B_H(R)$. Then $\pi(x'_\pm)\in (\pi(x)+\mathbb T_\pm)\cap D_R$.

Now, the open set $U\setminus D_R$ is safe since $U$ is open and safe and each
intersection $D_R\cap (x+\mathbb T_\pm)$ has empty interior in $x+\mathbb T_\pm$.
Then by Lemma~\ref{l:safe-cpct} there exists a safe compact set $K_1\subset U\setminus D_R$.
The complement $\mathbb T^{2n}\setminus K_1$ contains $D_R$ and thus
is an open safe set. Using Lemma~\ref{l:safe-cpct} again, let
$K_2$ be a compact safe subset of this complement,
then $K_1,K_2$ satisfy~\eqref{e:K-1-2}.

\noindent 2. Using a partition of unity subordinate
to the cover of $\mathbb T^{2n}$ by the sets
$U\setminus K_2$, $\mathbb T^{2n}\setminus K_1$
we choose $a_1,a_2\in C^\infty(\mathbb T^{2n})$ such that
$$
a_1,a_2\geq 0,\quad
a_1+a_2=1,\quad
\supp a_1\subset U\setminus K_2,\quad
\supp a_2\subset\mathbb T^{2n}\setminus K_1.
$$
The complements of $\supp a_1$, $\supp a_2$ contain the sets
$K_2$, $K_1$ and are thus both safe, finishing the proof.
\end{proof}
We are now ready to prove Theorems \ref{t:general} and \ref{t:damped}.
\begin{proof}[Proof of Theorem \ref{t:general}]
Let $a \in C^\infty\p{\T^{2n}}$ be as in the statement of Theorem \ref{t:general}.
Since $\set{a \neq 0}$ satisfies the geometric control condition transversally to $\T_+$ and $\T_-$, the open set
$$
U:=\bigcup_{m \in \Z} A^{m} \big(\{a \neq 0\}\big)
$$
is safe.
We apply Lemma \ref{l:partition_unity} to~$U$
to construct two functions $a_1,a_2$ and we set $b_1 := a_1$ and $b_2 := a_2$
in~\S\ref{s:words}.
Notice that we have consequently $b = b_1+b_2=1$ and \eqref{e:decomposition} becomes
\begin{equation*}
I = B_{\mathcal{X}} + B_{\mathcal{Y}}.
\end{equation*}
Fix $\alpha > 0$ be small enough so that Lemma~\ref{l:uncontrolled} applies,
that is there are $C, \delta > 0$ such that for every $u \in \mathcal{H}_{\mathbf N}(\theta)$ we have
\begin{equation}\label{e:unc_support}
\begin{split}
\n{B_{\mathcal{X}} u }_{\mathcal{H}} \leq C \mathbf N^{- \delta} \n{u}_{\mathcal{H}}.
\end{split}
\end{equation}
Next, applying Lemma~\ref{l:controlled_support} with sufficiently small $\varepsilon>0$,
recalling~\eqref{e:condition_J},
and making $C$ larger and $\delta>0$ smaller if needed, we have for every $u \in \mathcal{H}_{\mathbf N}(\theta)$
\begin{equation}\label{e:con_support}
\n{B_{\mathcal{Y}} u}_{\mathcal{H}} \leq C \n{\Op_{\mathbf{N},\theta}(a)u}_{\mathcal{H}} + C \mathbf{r}_M(u) \log \mathbf{N} +  C \mathbf{N}^{ - \delta} \n{u}_{\mathcal{H}}.
\end{equation}
Putting \eqref{e:unc_support} and \eqref{e:con_support} together, we find that, for every $u \in \mathcal{H}_{\mathbf N}(\theta)$, we have
\begin{equation}\label{e:almost_support}
\n{u}_{\mathcal{H}} \leq C \n{\Op_{\mathbf{N},\theta}(a)u}_{\mathcal{H}} + C \mathbf{r}_M(u) \log \mathbf{N} +  C \mathbf{N}^{- \delta}\n{u}_{\mathcal{H}}.
\end{equation}
Now, for $\mathbf N$ large enough we can remove the last term on the right-hand side
of~\eqref{e:almost_support}, obtaining~\eqref{e:general-estimate} and finishing the proof.\end{proof}
\smallskip
\begin{proof}[Proof of Theorem \ref{t:damped}]
We will find $\mathbf N$-independent constants $C, \delta, \kappa > 0$ and an integer $m_0\geq 0$ such that
\begin{equation}\label{e:estimate_spectral_radius}
\n{\p{\Op_{\mathbf{N},\theta}(b) M_{\mathbf{N},\theta}}^{(2m_0+1)\lfloor \kappa \log \mathbf N \rfloor}}_{\mathcal{H} \to \mathcal{H}} \leq C \mathbf{N}^{- \delta}.
\end{equation}
It will then follow from \eqref{e:estimate_spectral_radius} that the spectral radius of $\Op_{\mathbf{N},\theta}(b) M_{\mathbf N,\theta}$ is bounded above by
\begin{equation*}
C^{\frac{1}{(2m_0+1)\lfloor \kappa \log \mathbf N \rfloor}} \exp\p{- \frac{\delta \log \mathbf N}{(2m_0+1)\lfloor \kappa \log \mathbf N \rfloor}}\quad \xrightarrow{\mathbf N \to + \infty}\quad e^{- \frac{\delta}{(2m_0+1)\kappa}} < 1
\end{equation*}
which will give the conclusion of the theorem.

\noindent 1. We first reduce to the situation when
the set $\{|b|<1\}$ is safe (which is a stronger assumption
than made in Theorem~\ref{t:damped}).
Consider the open set
$$
U:=\bigcup_{m\in\mathbb Z}A^m\big(\{|b|<1\}\big)=\bigcup_{m_0\geq 1} U_{m_0}\quad\text{where}\quad
U_{m_0}:=\bigcup_{|m|\leq m_0}A^m\big(\{|b|<1\}\big).
$$
By the assumption of the theorem, the set $U$ is safe.
By Lemma~\ref{l:safe-cpct}, $U$
contains a safe compact subset $K$. Since each $U_{m_0}$ is open,
we may fix~$m_0$ such that $U_{m_0}\supset K$,
which implies that $U_{m_0}$ is safe.

Using~\eqref{e:our-egorov} and~\eqref{e:our-product-2} (with $\rho = \rho' = 0$), we next see that
\begin{equation*}
\p{\Op_{\mathbf{N},\theta}(b) M_{\mathbf N,\theta}}^{2 m_0 +1} = M^{m_0}_{\mathbf{N},\theta} \Op_{\mathbf{N},\theta}(\tilde{b}) M^{m_0 +1}_{\mathbf{N},\theta} + \mathcal{O}\p{\mathbf N^{-1}}_{\mathcal{H}_{\mathbf N}(\theta) \to \mathcal{H}_{\mathbf N}(\theta)},
\end{equation*}
where 
\begin{equation*}
\tilde{b} := \prod_{m = -m_0}^{m_0} b \circ A^m.
\end{equation*}
Since $|b|\leq 1$ everywhere, the set $\{|\tilde b|<1\}=U_{m_0}$ is safe.

Using \eqref{e:basic-norm} to bound the operator norm of $\Op_{\mathbf{N},\theta}(\tilde{b})$ by $1 + \mathcal{O}(\mathbf{N}^{-{1\over 2}})$, we have for any fixed $\kappa$
\begin{equation*}
\begin{split}
& \p{\Op_{\mathbf{N},\theta}(b) M_{\mathbf N,\theta}}^{(2 m_0 +1)\lfloor \kappa \log \mathbf N \rfloor } \\ & \qquad \qquad = M_{\mathbf{N},\theta}^{m_0} \p{\Op_{\mathbf{N},\theta}(\tilde{b})M_{\mathbf N,\theta}^{2 m_0 +1}}^{\lfloor \kappa \log \mathbf N \rfloor} M_{\mathbf{N},\theta}^{-m_0} + \mathcal{O}(\mathbf{N}^{-1+})_{\mathcal{H}_{\mathbf N}(\theta) \to \mathcal{H}_{\mathbf N}(\theta))}.
\end{split}
\end{equation*}
Hence to show~\eqref{e:estimate_spectral_radius} it suffices to prove
\begin{equation}
  \label{e:estimate_spectral_radius_2}
\Big\|\big(\Op_{\mathbf N,\theta}(\tilde b)M_{\mathbf N,\theta}^{2m_0+1}\big)^{\lfloor\kappa\log\mathbf N\rfloor}\Big\|_{\mathcal H\to\mathcal H}\leq C\mathbf N^{-\delta}.
\end{equation}
The operator $M^{2m_0+1}$ lies in $\mathcal M_{A^{2m_0+1}}$,
with the matrix $A^{2m_0+1}$ still satisfying the spectral gap
assumption~\eqref{e:spectral-gap} and producing the
same tori $\mathbb T_\pm$. Therefore, to show~\eqref{e:estimate_spectral_radius_2}
it suffices to prove the bound
\begin{equation}
  \label{e:estimate_spectral_radius_3}
\Big\|\big(\Op_{\mathbf N,\theta}(b)M_{\mathbf N,\theta}\big)^{\lfloor\kappa\log\mathbf N\rfloor}\Big\|_{\mathcal H\to\mathcal H}\leq C\mathbf N^{-\delta}
\end{equation}
for any $A$ satisfying~\eqref{e:spectral-gap} and any $M\in \mathcal M_A$,
where we assume that $b\in C^\infty(\mathbb T^{2n})$,
$|b|\leq 1$ everywhere, and the set $\{|b|<1\}$ is safe.

\noindent 2. We now show~\eqref{e:estimate_spectral_radius_3}.
Using Lemma~\ref{l:partition_unity} for the set $\{|b|<1\}$,
we construct two cutoff functions $a_1,a_2\in C^\infty(\mathbb T^{2n})$. Put
$$
b_1 = a_1 b,\quad
b_2 = a_2 b.
$$
Note that $b = b_1 + b_2$ in agreement with the convention of~\S\ref{s:words}
and $B = \Op_{\mathbf{N},\theta}(b)$.
Moreover, $|b_1|+|b_2|=|b|\leq 1$.
 By construction, $b_1$ and $b_2$ satisfy the hypotheses of Lemma~\ref{l:uncontrolled}. Consequently, we can choose $\alpha$ small enough so that this lemma applies: there are $C,
\delta > 0$ such that
\begin{equation*}
\n{B_{\mathcal{X}}}_{\mathcal{H} \to \mathcal{H}} \leq C \mathbf{N}^{- \delta}.
\end{equation*}
As $a_1$ is supported in $\set{\va{b} < 1}$, we see that $\va{b_1} + \va{b_2} < 1$ on
$\supp b_1$, and we can consequently apply Lemma~\ref{l:controlled_damped} to see that, up to making $C$ larger and $\delta$ smaller, we also have
\begin{equation*}
\n{B_{\mathcal{Y}}}_{\mathcal{H} \to \mathcal{H}} \leq C \mathbf{N}^{-\delta}.
\end{equation*}
Recalling \eqref{e:decomposition}, we see that
\begin{equation*}
\begin{split}
\n{\p{\Op_{\mathbf{N},\theta}(b) M_{\mathbf{N},\theta}}^{2 T_1}}_{\mathcal{H} \to \mathcal{H}} & = \n{M_{\mathbf{N},\theta}^{2 T_1 - 1} \p{B_{\mathcal{X}} + B_{\mathcal{Y}}} M_{\mathbf{N},\theta}}_{\mathcal{H} \to \mathcal{H}} \\
    & \leq \n{B_{\mathcal{X}}}_{\mathcal{H} \to \mathcal{H}} + \n{B_{\mathcal{Y}}}_{\mathcal{H} \to \mathcal{H}} \leq 2 C \mathbf{N}^{- \delta}.
\end{split}
\end{equation*}
Since $T_1$ is defined by \eqref{e:def_times}, we just established \eqref{e:estimate_spectral_radius_3}, which ends the proof of the theorem.
\end{proof}

\subsection{Proof of Theorem~\ref{t:measures}}
\label{s:proof-measures}

We argue by contradiction. Let $\mu$ be a semiclassical measure associated to~$A$
and assume that $\supp\mu$ does not contain any sets of the form
$x+\mathbb T_+$ or $x+\mathbb T_-$. Then the complement
$\mathbb T^{2n}\setminus\supp\mu$ is an open safe set in the sense of Definition~\ref{d:safe}.
{By Lemma~\ref{l:safe-cpct}},
there exists a compact safe set $K$ such that
$K\cap\supp\mu=\emptyset$. Take
$a\in C^\infty(\mathbb T^{2n})$ such that
$\supp a\cap\supp\mu=\emptyset$ and
$a=1$ on $K$. Then the set $\{a\neq 0\}$ is safe, so
$a$ satisfies the hypothesis of Theorem~\ref{t:general}.

Let $u_j\in\mathcal H_{\mathbf N_j}(\theta_j)$ be a sequence converging to~$\mu$
in the sense of~\eqref{e:measure-def}. Since $u_j$ is an eigenfunction of
$M_{\mathbf N_j,\theta_j}$ and $\mathbf N_j\to\infty$, Theorem~\ref{t:general}
shows that there exists a constant $C_a$ such that for all large enough $j$
\begin{equation}
  \label{e:lobound}
1=\|u_j\|_{\mathcal H}\leq C_a\|\Op_{\mathbf N_j,\theta_j}(a)u_j\|_{\mathcal H}.
\end{equation}
Now, by~\eqref{e:our-adjoint} and~\eqref{e:our-product-2} (with $\rho=\rho'=0$)
we have 
\begin{equation}
  \label{e:lobound-2}
\begin{aligned}
\|\Op_{\mathbf N_j,\theta_j}(a)u_j\|_{\mathcal H}^2
&=\langle\Op_{\mathbf N_j,\theta_j}(a)^*\Op_{\mathbf N_j,\theta_j}(a)u_j,u_j\rangle_{\mathcal H}\\
&=\langle\Op_{\mathbf N_j,\theta_j}(|a|^2)u_j,u_j\rangle_H+\mathcal O(\mathbf N_j^{-1}).
\end{aligned}
\end{equation}
By~\eqref{e:lobound}, the left-hand side of~\eqref{e:lobound-2}
is bounded away from~0 as $j\to\infty$. By~\eqref{e:measure-def}, the right-hand side
of~\eqref{e:lobound-2} converges to $\int_{\mathbb T^{2n}}|a|^2\,d\mu$; thus
this integral is positive. This contradicts the fact that $\supp a\cap\supp\mu=\emptyset$
and finishes the proof.

\subsection{Estimates in the controlled region}\label{s:controlled}

Here we prove Lemmas \ref{l:controlled_support} and \ref{l:controlled_damped}, using the notation that we introduced in \S \ref{s:words}. We start by relating the operator $B_\w$ from~\eqref{e:operator_bw} and the symbol $b_\w$ from~\eqref{e:symbole_bw}.
Recall the symbol classes $S_{L,\rho,\rho'}(\mathbb T^{2n})$ from~\S\ref{s:quantization-on-torus},
the constants $\rho,\rho'$ defined in~\eqref{e:conditions_rho},
the integers $J,T_0$ defined in~\eqref{e:condition_J} and~\eqref{e:def_times}, and 
the hyperplane $L_-\subset\mathbb R^{2n}$ defined in~\eqref{e:L-pm-def}.
\begin{lemm}\label{l:short_log_words}
Let $\varepsilon > 0$. For every $\w \in \mathcal{W}(T_0)$, the symbol $b_\w$ belongs to the symbol class $S_{L_-, \frac{\rho}{J} + \varepsilon, \frac{\rho'}{J} + \varepsilon}(\mathbb T^{2n})$ and
\begin{equation*}
B_\w = \Op_{\mathbf{N},\theta}(b_\w) + \mathcal{O}\p{\mathbf{N}^{\frac{\rho + \rho'}{J} + \varepsilon - 1}}_{\mathcal{H}_{\mathbf N}(\theta) \to \mathcal{H}_{\mathbf N}(\theta)}.
\end{equation*}
Here the semi-norms of $b_\w$ and the constant in the $\mathcal{O}(\bullet)$ are bounded uniformly in $\w$.
\end{lemm}
\begin{proof}
1. We first show that each $S_{L_-,{\rho\over J},{\rho'\over J}}$-seminorm
of the symbols
\begin{equation*}
b_1 \circ A^j \quad \textup{ and } \quad b_2 \circ A^j,\quad
0\leq j\leq T_0
\end{equation*}
is bounded uniformly in $j$ and $\mathbf N$. To simplify notation we give the proof for $b_1$,
which applies to $b_2$ as well (in fact, it applies with $b_1$ replaced by any fixed function
in $C^\infty(\mathbb T^{2n})$).

Let $X_1,\dots,X_k, Y_1,\dots,Y_m$ be constant vector fields on $\R^{2n}$ such that $Y_1,\dots,Y_m$ are tangent to $L_-$. Since $A$ is a linear map, we compute for $x \in \T^{2n}$
\begin{equation*}
X_1 \dots X_k Y_1 \dots Y_m (b_1 \circ A^j) (x) = D^{k+m} b_1 (A^j x) \cdot \p{A^j X_1,\dots,A^j X_k, A^j Y_1,\dots, A^j Y_m}.
\end{equation*}
Here, $D^{k +m} b_1$ denotes the $k+m$-th derivative of $b_1$, which is a $k+m$-linear form, uniformly bounded since it does not depend on $\mathbf{N}$. Consequently, for some $\mathbf N$-independent constant $C > 0$, we have
\begin{equation*}
\sup\va{X_1 \dots X_k Y_1 \dots Y_m (b_1 \circ A^j)} \leq C \va{A^j X_1} \cdots \va{A^j X_k} \va{A^j Y_1} \cdots \va{A^j Y_m}.
\end{equation*}
Using the norm bounds~\eqref{e:L-pm-norm} on $\|A^j\|$ and $\|A^j|_{L_-}\|$
we get (for a different choice of~$C$)
\begin{equation*}
\begin{split}
\sup\va{X_1 \dots X_k Y_1 \dots Y_m (b_1 \circ A^j)} & \leq C \va{\lambda_+}^{jk} \gamma^{jm} \leq C \va{\lambda_+}^{k T_0} \gamma^{m T_0} \\
    & \leq C \mathbf{N}^{\frac{\rho}{J} k + \frac{\rho'}{J} m},
\end{split}
\end{equation*}
which is precisely the required estimate. Here we used the definition~\eqref{e:def_times} of $T_0$ and the condition \eqref{e:conditions_rho} that we imposed on $\rho'$.

\noindent 2. Let $\w\in \mathcal W(T_0)$. By the exact Egorov theorem~\eqref{e:our-egorov}, and recalling~\eqref{e:conjugation-def}, we have
$$
B_{w_j}(j)=\Op_{\mathbf N,\theta}(b_{w_j}\circ A^j)\quad\text{for all}\quad j=0,\dots,T_0-1.
$$
Now it remains to use Lemma~\ref{l:long_product}, the definitions~\eqref{e:operator_bw}
and~\eqref{e:symbole_bw}
of $B_\w$ and $b_\w$, and the uniform bound etablished in Step~1 of this proof.
\end{proof}
With Lemma~\ref{l:short_log_words} at our disposal, we can produce the proof of Lemma \ref{l:controlled_damped}.

\begin{proof}[Proof of Lemma \ref{l:controlled_damped}] 
1. Let us prove first the required estimate for the operator $B_{\mathcal{Z}}$ (instead of $B_{\mathcal{Y}}$) associated to the set $\mathcal{Z}$ defined by \eqref{e:def_Z}.
Put
$$
\eta:=\max_{\supp b_1} \big(|b_1|+|b_2|\big)<1
$$
then $\va{b_1} \leq \eta \p{ 1 - \va{b_2}}$ everywhere. For $\w\in\mathcal W(T_0)$,
denote
$$
\tilde b_\w:=\prod_{j=0}^{T_0-1}\tilde b_{w_j}\circ A^j\quad\text{where}\quad
\tilde b_1:=1-|b_2|,\quad
\tilde b_2:=|b_2|.
$$
Recalling the function $F$ from~\eqref{e:definition_F} and the definition~\eqref{e:symbole_bw} of $b_\w$, we see that
\begin{equation*}
\va{b_\w} \leq \eta^{\alpha T_0} \tilde b_\w\quad\text{for all}\quad
\w\in\mathcal Z.
\end{equation*}
Summing over $\w \in \mathcal{Z}$, we find that
\begin{equation*}
\begin{split}
\va{b_{\mathcal{Z}}} & \leq \eta^{\alpha T_0} \sum_{\w \in \mathcal{Z}}\tilde b_\w \leq \eta^{\alpha T_0} \sum_{\w \in \mathcal{W}(T_0)} \tilde b_\w = \eta^{\alpha T_0}.
\end{split}
\end{equation*}
Let $\varepsilon > 0$ be very small. By~\eqref{e:def_times}, there are at most $2^{T_0} \leq \mathbf{N}^{\frac{\rho \log 2}{J \log \va{\lambda_+}}}$ elements in~$\mathcal{Z}$. It follows from Lemma \ref{l:short_log_words} that $\mathbf{N}^{-\frac{\rho \log 2}{J \log \va{\lambda_+}}} b_{\mathcal{Z}}$ is bounded in $S_{L_-, \frac{\rho}{J} + \varepsilon, \frac{\rho'}{J} + \varepsilon}$ uniformly in~$\mathbf N$ and that 
\begin{equation*}
\mathbf{N}^{-\frac{\rho \log 2}{J \log \va{\lambda_+}}} B_{\mathcal{Z}} = \Op_{\mathbf{N},\theta}\p{\mathbf{N}^{-\frac{\rho \log 2}{J \log \va{\lambda_+}}} b_{\mathcal{Z}}} + \mathcal{O}\p{\mathbf{N}^{\frac{\rho + \rho'}{J} + \varepsilon - 1}}_{\mathcal{H}_{\mathbf N}(\theta) \to \mathcal{H}_{\mathbf N}(\theta)}.
\end{equation*}
We can consequently apply \eqref{e:basic-norm} to the symbol $\mathbf{N}^{-\frac{\rho \log 2}{J \log \va{\lambda_+}}} b_{\mathcal{Z}}$ in order to find that, for some $C > 0$ that may vary from one line to another,
\begin{equation}\label{e:estimate_BZ}
\begin{split}
\n{B_{\mathcal{Z}}}_{\mathcal{H} \to \mathcal{H}} & \leq \mathbf{N}^{\frac{\rho \log 2}{J \log \va{\lambda_+}}}\n{\Op_{\mathbf{N},\theta}\p{\mathbf{N}^{-\frac{\rho \log 2}{J \log \va{\lambda_+}}} b_{\mathcal{Z}}}}_{\mathcal{H} \to \mathcal{H}} + C \mathbf{N}^{\frac{\rho \log 2}{J \log \va{\lambda_+}}+\frac{\rho + \rho'}{J} + \varepsilon - 1} \\
   & \leq \eta^{\alpha T_0} + C \mathbf{N}^{\frac{\rho \log 2}{J \log \va{\lambda_+}}+\frac{\rho + \rho'}{2J} + \varepsilon - {1\over 2}} \\
   & \leq C \mathbf{N}^{- \delta},
\end{split}
\end{equation}
with
\begin{equation*}
\delta := \min \p{- \frac{\alpha \rho \log \eta}{J \log \va{\lambda_+}}, {1\over 2} - \frac{\rho \log 2}{J \log \va{\lambda_+}}-\frac{\rho + \rho'}{2J} - \varepsilon}.
\end{equation*}
Notice that the condition \eqref{e:condition_J} that we imposed on $J$ ensures that $\delta$ is positive (provided $\varepsilon$ is small enough).

\noindent 2. The same proof with $\eta$ replaced by $1$ gives 
\begin{equation}\label{e:estimate_others}
\n{B_{\mathcal{W}(T_0)}}_{\mathcal{H} \to \mathcal{H}} \leq 1 + C \mathbf{N}^{- \delta},\quad \n{B_{\mathcal{W}(T_0) \setminus \mathcal{Z}}}_{\mathcal{H} \to \mathcal{H}} \leq 1 + C \mathbf{N}^{- \delta} .
\end{equation}

\noindent 3. Let us now use the estimates \eqref{e:estimate_BZ} and \eqref{e:estimate_others} to prove the lemma. If $\w=\w^{(1)}\dots \w^{(2J)}\in\mathcal W(2T_1)$
is the concatenation of the words $\w^{(1)},\dots,\w^{(2J)}\in\mathcal W(T_0)$,
then by~\eqref{e:operator_bw} we have
$$
B_\w=B_{\w^{(2J)}}((2J-1)T_0)\cdots B_{\w^{(2)}}(T_0)B_{\w^{(1)}}(0).
$$
Using the definition~\eqref{e:def_Y} of $\mathcal{Y}$, 
and splitting this set into the disjoint union of $2J$ subsets corresponding
to the largest $\ell$ such that $\w^{(\ell)}\in\mathcal Z$,
we write
\begin{equation}
\label{e:Y-decomposed}
\begin{split}
B_{\mathcal{Y}} = \sum_{\ell = 1}^{2J} M_{\mathbf{N},\theta}^{-(2J - 1)T_0} \p{B_{\mathcal{W}(T_0) \setminus \mathcal{Z}} M_{\mathbf{N},\theta}^{T_0}}^{2 J - \ell} B_{\mathcal{Z}}  \p{M_{\mathbf{N},\theta}^{T_0} B_{\mathcal{W}(T_0)} }^{\ell - 1} .
\end{split}
\end{equation}
According to \eqref{e:estimate_BZ} and \eqref{e:estimate_others}, the operator norm of each term in this sum is less than
\begin{equation*}
C \p{1 + C\mathbf{N}^{- \delta}}^{2J - 1} \mathbf{N}^{- \delta}. 
\end{equation*}
Since the number of terms in this sum is $2J$, that does not depend on $\mathbf{N}$, the estimate \eqref{e:controlled_damped} follows.
\end{proof}
In order to prove Lemma \ref{l:controlled_support}, we need a few more preliminary results.
We start with a norm estimate on the operators $B_c$ defined in~\eqref{e:B-c-def}.
\begin{lemm}\label{l:garding_words}
Assume that $0 \leq b_1,b_2 \leq 1$. Let $\varepsilon > 0$. Let
$$
c,d:\mathcal W(T_0)\to\mathbb C,\quad
|c(\w)|\leq d(\w)\leq 1\quad\text{for all}\quad \w\in\mathcal W(T_0).
$$
Then there is a constant $C > 0$, that does not depend on $c$ nor $d$ such that, for every $u \in \mathcal{H}_{\mathbf{N}}(\theta)$, we have
\begin{equation*}
\n{B_c u}_{\mathcal{H}} \leq \n{B_d u}_{\mathcal{H}} + C \mathbf{N}^{- \frac{1}{2} +\frac{1}{2J}\p{1 + \frac{2 \log 2}{\log \va{\lambda_+}}}+\varepsilon } \n{u}_{\mathcal{H}}
\end{equation*}
\end{lemm}
\begin{proof}
Since the number of elements in~$\mathcal{W}(T_0)$ is $2^{T_0}\leq\mathbf{N}^{\frac{\rho \log 2}{J \log \va{\lambda_+}}}$, it follows from Lemma~\ref{l:short_log_words} that $\mathbf{N}^{-\frac{\rho \log 2}{J \log \va{\lambda_+}}} b_c$ and $\mathbf{N}^{-\frac{\rho \log 2}{J \log \va{\lambda_+}}} b_d$ are bounded in the symbol class $S_{L_-, \frac{\rho}{J}+ \varepsilon,\frac{\rho'}{J}+ \varepsilon}$ uniformly in $\mathbf N,c,d$ and that
\begin{equation*}
\mathbf{N}^{-\frac{\rho \log 2}{J \log \va{\lambda_+}}} B_c = \Op_{\mathbf{N},\theta}\p{\mathbf{N}^{-\frac{\rho \log 2}{J \log \va{\lambda_+}}} b_c} + \mathcal{O}\p{\mathbf{N}^{\frac{\rho + \rho'}{J} + \varepsilon - 1}}_{\mathcal{H}_{\mathbf{N}}(\theta) \to \mathcal{H}_{\mathbf{N}}(\theta)}.
\end{equation*}
The same estimate is satisfied by $B_d$ and $b_d$.
By our assumption, $|b_c|\leq b_d$ everywhere. Thus by Lemma~\ref{l:advanced-norm}
\begin{equation*}
\n{B_c u}_{\mathcal{H}} \leq \n{B_d u}_{\mathcal{H}} + C \mathbf{N}^{- {1\over 2} + \frac{\rho + \rho'}{2J} + \frac{ \rho \log 2}{J \log \va{\lambda_+}} +  \varepsilon} \n{u}_{\mathcal{H}}.
\end{equation*}
The result then follows by using that $\rho \leq \rho + \rho' \leq 1$.
\end{proof}

We will also use a standard elliptic estimate.
\begin{lemm}\label{l:elliptic_estimate}
Let $a \in C^{\infty}(\T^{2n})$ and assume that
$$
\supp b_1\ \subset\ \bigcup_{\ell \in \Z} A^{ \ell}(\set{a \neq 0}).
$$
Then there is a constant $C > 0$ such that, for every $m \in \Z$ and every $u \in \mathcal{H}_{\mathbf{N}}(\theta)$, we have (with $\mathbf r_M(u)$ defined in~\eqref{e:quasimode})
\begin{equation}
\label{e:elliptic_estimate}
\n{B_1(m)u}_{\mathcal{H}} \leq C \n{\Op_{\mathbf{N},\theta}(a) u}_{\mathcal{H}} + C (1 + |m|) \mathbf{r}_M(u) + C \mathbf{N}^{-1}\n{u}_{\mathcal{H}}. 
\end{equation} 
\end{lemm}
\begin{proof} 
1. First of all, we may assume that $a\geq 0$ everywhere, since
we may replace $a$ with $|a|^2$ and use the following corollary of~\eqref{e:our-product-2}
(with $\rho=\rho'=0$):
$$
\begin{aligned}
\|\Op_{\mathbf N,\theta}(|a|^2)u\|_{\mathcal H}&\leq
\|\Op_{\mathbf N,\theta}(\bar a)\Op_{\mathbf N,\theta}(a)u\|_{\mathcal H}
+C\mathbf N^{-1}\|u\|_{\mathcal H}\\
&\leq C\|\Op_{\mathbf N,\theta}(a)u\|_{\mathcal H}
+C\mathbf N^{-1}\|u\|_{\mathcal H}.
\end{aligned}
$$
Here $C$ denotes an $\mathbf N$-independent constant whose precise value may change from line to line.

\noindent 2. By a compactness argument, we see that there exists $\ell_0 \in \N$ such that 
$$
\supp b_1\ \subset\
\bigcup_{\ell = - \ell_0}^{\ell_0} A^{ \ell} (\set{a>0})
=\{\tilde a>0\}\quad\text{where}\quad
\tilde a:=\sum_{\ell=-\ell_0}^{\ell_0}
a\circ A^\ell.
$$
We write $b_1=q \tilde a$ for some $q\in C^\infty(\mathbb T^{2n})$. Then
by~\eqref{e:our-product-2}
and~\eqref{e:our-egorov}
$$
\begin{aligned}
\|\Op_{\mathbf N,\theta}(b_1)u\|_{\mathcal H}
&\leq 
\|\Op_{\mathbf N,\theta}(q)\Op_{\mathbf N,\theta}(\tilde a)u\|_{\mathcal H}
+C\mathbf N^{-1}\|u\|_{\mathcal H}\\
&\leq C\|\Op_{\mathbf N,\theta}(\tilde a)u\|_{\mathcal H}
+C\mathbf N^{-1}\|u\|_{\mathcal H}\\
&=C\sum_{\ell=-\ell_0}^{\ell_0}\|M_{\mathbf N,\theta}^{-\ell}\Op_{\mathbf N,\theta}(a)
M_{\mathbf N,\theta}^\ell u\|_{\mathcal H}
+C\mathbf N^{-1}\|u\|_{\mathcal H}.
\end{aligned}
$$
Applying this with $u$ replaced by $M_{\mathbf N,\theta}^m u$, we see that
for all~$m$, with the constant $C$ independent of~$m$,
\begin{equation}
  \label{e:ell-int-1}
\|B_1(m)u\|_{\mathcal H}
\leq C\sum_{\ell=-\ell_0}^{\ell_0}
\|\Op_{\mathbf N,\theta}(a)M_{\mathbf N,\theta}^{m+\ell}u\|_{\mathcal H}
+C\mathbf N^{-1}\|u\|_{\mathcal H}.
\end{equation}

\noindent 3. We have for every operator $A:\mathcal H_{\mathbf N}(\theta)\to\mathcal H_{\mathbf N}(\theta)$, $m\in\mathbb Z$, and $u\in\mathcal H_{\mathbf N}(\theta)$
\begin{equation}
  \label{e:ell-int-2}
\|AM_{\mathbf N,\theta}^m u\|_{\mathcal H}
\leq \|Au\|_{\mathcal H}+\|A\|_{\mathcal H\to\mathcal H} |m|\mathbf r_M(u).
\end{equation}
Indeed, take $z\in\mathbb S^1$ such that
$\mathbf r_M(u)=\|M_{\mathbf N,\theta}u-zu\|_{\mathcal H}$. Then we have when $m \geq 0$,
\begin{equation*}
\begin{split}
\n{M_{\mathbf{N},\theta}^m u - z^m u}_{\mathcal{H}} & \leq \sum_{\ell = 1}^{m} \n{z^{m-\ell} M_{\mathbf{N},\theta}^{\ell - 1}\p{M_{\mathbf{N},\theta} - z}u}_{\mathcal{H}} \\ & \leq |m| \mathbf{r}_M(u).
\end{split}
\end{equation*}
By a similar computation, we see that this estimate still holds when $m < 0$. These estimates
imply~\eqref{e:ell-int-2}.

\noindent 4. Finally, putting together~\eqref{e:ell-int-1} and~\eqref{e:ell-int-2}
(with $A:=\Op_{\mathbf N,\theta}(a)$) we get~\eqref{e:elliptic_estimate}.
\end{proof}

We are now ready to prove Lemma \ref{l:controlled_support}.

\begin{proof}[Proof of Lemma \ref{l:controlled_support}]
We will prove the estimate \eqref{e:controlled_support} first for $B_F$, then for $B_{\mathcal{Z}}$, and finally for $B_{\mathcal{Y}}$.

\noindent 1. We start by considering the operator $B_F$ associated to the function $F$ defined by~\eqref{e:definition_F}. Notice that the assumption that $b_1 + b_2 = 1$ implies that $B_1(j) + B_2(j) = 1$ for $j = 0,\dots, T_0 - 1$. It follows from the definition of $F$ that
\begin{equation*}
B_F = \frac{1}{T_0} \sum_{\w \in \mathcal{W}(T_0)} \p{\sum_{j= 0}^{T_0 - 1} \indic_{\set{w_j = 1}}} B_{\w} = \frac{1}{T_0} \sum_{j = 0}^{T_0 - 1} \sum_{\substack{\w \in \mathcal{W}(T_0) \\ w_j = 1}} B_{\w} = \frac{1}{T_0} \sum_{j = 0}^{T_0 - 1} B_1(j).
\end{equation*}
Consequently, using Lemma \ref{l:elliptic_estimate}, we find that for $u \in \mathcal{H}_{\mathbf{N}}(\theta)$ we have
\begin{equation*}
\n{B_F u}_{\mathcal{H}} \leq C \n{\Op_{\mathbf{N},\theta}(a) u}_{\mathcal{H}} + C T_0 \mathbf{r}_M(u) + C \mathbf{N}^{-1} \n{u}_{\mathcal{H}}.
\end{equation*}
Recalling the definition \eqref{e:def_times} of $T_0$, we find that, up to making $C$ larger, we have
\begin{equation}\label{e:estimate_BF}
\n{B_F u}_{\mathcal{H}} \leq C \n{\Op_{\mathbf{N},\theta}(a) u}_{\mathcal{H}} + C \mathbf{r}_M(u) \log \mathbf{N} + C \mathbf{N}^{-1} \n{u}_{\mathcal{H}}.
\end{equation}

\noindent 2. By the  definition~\eqref{e:def_Z} of $\mathcal Z$,
we have $F \geq \alpha \indic_{\mathcal{Z}}$. By Lemma \ref{l:garding_words}, we deduce from~\eqref{e:estimate_BF} that, for some new $C > 0$ depending on~$\alpha$ and every $u \in \mathcal{H}_{\mathbf{N}}(\theta)$, we have
\begin{equation}\label{e:estimate_BZ2}
\n{B_{\mathcal{Z}}u}_{\mathcal{H}} \leq C \n{\Op_{\mathbf{N},\theta}(a) u}_{\mathcal{H}} + C \mathbf{r}_M(u) \log \mathbf{N} + C \mathbf{N}^{- \frac{1}{2} + \frac{1}{2J}\p{1+ \frac{2 \log 2}{\log \va{\lambda_+}}} + \varepsilon} \n{u}_{\mathcal{H}}.
\end{equation}

\noindent 3. From our assumption that $b_1 + b_2 = 1$, we deduce that $B_{\mathcal{W}(T_0)} = I$. Hence we have similarly to~\eqref{e:Y-decomposed} 
$$
\label{e:decomposition_BY}
B_{\mathcal{Y}} = \sum_{ \ell = 1}^{2J} M_{\mathbf N,\theta}^{-(2J-1)T_0}
(B_{\mathcal W(T_0)\setminus\mathcal Z}M_{\mathbf N,\theta}^{T_0})^{2J-\ell}
B_{\mathcal Z} M_{\mathbf N,\theta}^{(\ell-1)T_0}.
$$
Similarly to~\eqref{e:estimate_others} we have $\|B_{\mathcal Z}\|_{\mathcal H\to\mathcal H},\|B_{\mathcal W(T_0)\setminus\mathcal Z}\|_{\mathcal H\to\mathcal H}\leq 1+C\mathbf N^{-\delta}$ for some $\delta>0$, so by~\eqref{e:ell-int-2}
\begin{equation}
\label{e:decomposition_BY_2}
\begin{aligned}
\|B_{\mathcal Y}u\|_{\mathcal H}&\leq
2\sum_{\ell=1}^{2J}
\|B_{\mathcal Z}M_{\mathbf N,\theta}^{(\ell-1)T_0}u\|_{\mathcal H}
\\&\leq
4J\|B_{\mathcal Z}u\|_{\mathcal H}
+C\mathbf r_M(u)\log\mathbf N.
\end{aligned}
\end{equation}
Now \eqref{e:controlled_support} follows from~\eqref{e:estimate_BZ2} and~\eqref{e:decomposition_BY_2}.
\end{proof}

\subsection{Reduction to a fractal uncertainty principle}\label{s:reduction}

We now explain how Lemma~\ref{l:uncontrolled} may be deduced from a Fractal Uncertainty Principle type statement, Proposition~\ref{l:FUP_type} below. The proof of Proposition~\ref{l:FUP_type} is given in~\S\ref{s:proof_lemma}.

Let $\w \in \mathcal{W}(2 T_1)$. Decompose the word $\w$ into two words of length $T_1$:
\begin{equation*}
\w = \w_+ \w_-, \quad \w_{\pm} \in \mathcal{W}(T_1).
\end{equation*}
Then, we relabel $\w_+$ and $\w_-$ as
\begin{equation*}
\w_+ = w_{T_1}^{+} \dots w_1^{+}, \quad \w_- = w_0^{-} \dots w_{T_1 - 1}^{-},
\end{equation*}
and define the symbols
\begin{equation}\label{e:def_bplus_bminus}
b_+ = \prod_{j = 1}^{T_1} b_{w_j^+} \circ A^{-j}, \quad b_- = \prod_{j = 0}^{T_1 - 1} b_{w_j^-} \circ A^j.
\end{equation}
In~\S\ref{s:decay_long_words}, we will prove the following estimate,
where $ h = (2 \pi \mathbf{N})^{-1}$, $\Op_h$ is the Weyl quantization on~$\mathbb R^n$
defined in~\eqref{e:weyl-quantization}, and we treat $b_\pm\in C^\infty(\mathbb T^{2n})$
as $\mathbb Z^{2n}$-periodic functions in $C^\infty(\mathbb R^{2n})$.
\begin{prop}\label{l:FUP_type}
Assume that the complements $\T^{2n}\setminus\supp b_1$, $\T^{2n}\setminus \supp b_2$
are both safe in the sense of Definition~\ref{d:safe}.
Then there are constants $C,\beta > 0$ that do not depend on~$\w$ nor $\mathbf{N}$ such that
\begin{equation*}
\n{\Op_h(b_-) \Op_h(b_+)}_{L^2(\mathbb R^n) \to L^2(\mathbb R^n)} \leq C h^{\beta}.
\end{equation*}
\end{prop}
In order to take advantage of Proposition~\ref{l:FUP_type}, let us first notice that the proof of Lemma \ref{l:short_log_words} also gives without major modification:
\begin{lemm}\label{l:long_logarithmic_words}
Let $\varepsilon > 0$. The symbols $b_-$ and $b_+$ belong respectively to the symbol classes $S_{L_-,\rho+\varepsilon,\rho'+\varepsilon}(\T^{2n})$ and $S_{L_+,\rho+\varepsilon,\rho'+\varepsilon}(\T^{2n})$, with bounds on the semi-norms that do not depend on $\mathbf N,\w$. (Here $L_\pm$
are defined in~\eqref{e:L-pm-def}.) Moreover, we have
$$
\begin{aligned}
B_{\w_+}(- T_1) &= \Op_{\mathbf{N},\theta}(b_+) + \mathcal{O}\p{\mathbf{N}^{\rho + \rho' + \varepsilon - 1}}_{\mathcal{H}_{\mathbf{N}}(\theta) \to \mathcal{H}_{\mathbf{N}}(\theta)},\\
B_{\w_-} &= \Op_{\mathbf{N},\theta}(b_-) + \mathcal{O}\p{\mathbf{N}^{\rho + \rho' + \varepsilon - 1}}_{\mathcal{H}_{\mathbf{N}}(\theta) \to \mathcal{H}_{\mathbf{N}}(\theta)},
\end{aligned}
$$
where the constants in~$\mathcal{O}(\bullet)$ are uniform in $\mathbf N,\w$.
\end{lemm}
Using Proposition~\ref{l:FUP_type} and Lemma~\ref{l:long_logarithmic_words}, we get a uniform bound on the operator norm of $B_\w$.
\begin{lemm}\label{l:bound_Bw}
Assume that the complements $\T^{2n}\setminus\supp b_1$, $\T^{2n}\setminus \supp b_2$
are both safe. Then there are constants $C,\beta > 0$ that do not depend on $\w$ nor $\mathbf{N}$ such that 
\begin{equation*}
\n{B_\w}_{\mathcal{H} \to \mathcal{H}} \leq C \mathbf{N}^{- \beta}\quad\text{for all}\quad
\w\in\mathcal W(2T_1).
\end{equation*}
\end{lemm}
\begin{proof}
We start by noticing that
\begin{equation}\label{e:factorization}
B_{\w} = M_{\mathbf{N},\theta}^{- T_1} B_{\w_-} B_{\w_+}(-T_1) M_{\mathbf{N},\theta}^{T_1}.
\end{equation}
By Lemma~\ref{l:long_logarithmic_words}, the operators
$\Op_{\mathbf N,\theta}(b_\pm)$ are bounded in norm uniformly in~$\w,\mathbf N$.
Hence, we deduce from Lemma \ref{l:long_logarithmic_words} and \eqref{e:factorization} that, for some $C > 0$, we have
\begin{equation*}
\n{B_{\w}}_{\mathcal{H} \to \mathcal{H}} \leq C \n{\Op_{\mathbf{N},\theta}(b_-) \Op_{\mathbf{N},\theta}(b_+)}_{\mathcal{H} \to \mathcal{H}} + C \mathbf{N}^{\rho + \rho' + \varepsilon - 1}.
\end{equation*}
Now, we deduce from \eqref{e:relator} that for every $g \in C^\infty(\mathbb{T}^{2n}; \mathcal{H}_N)$ we have
\begin{equation*}
\Pi_N \Op_h(b_-) \Op_h(b_+) \Pi_N^{*} g(\theta) = \Op_{\mathbf{N},\theta}(b_-) \Op_{\mathbf{N},\theta}(b_+) g(\theta).
\end{equation*}
Hence, using Lemma \ref{l:L2-decomposed}, we find as we did for \eqref{e:op-norm-decomposed} that
\begin{equation*}
\n{\Op_{\mathbf{N},\theta}(b_-) \Op_{\mathbf{N},\theta}(b_+)}_{\mathcal{H} \to \mathcal{H}} \leq \n{\Op_h(b_-) \Op_h(b_+)}_{L^2 \to L^2}, 
\end{equation*}
and the result follows then immediately from Proposition~\ref{l:FUP_type}.
\end{proof}

In order to get an estimate on the operator norm of $B_\mathcal{X}$ from Lemma \ref{l:bound_Bw}, we will need the following bound on the cardinal of $\mathcal{X}$.

\begin{lemm}\label{l:cardinal_X}
There are a constant $c>0$ (that does not depend on $\alpha$) and a constant $C > 0$ (that may depend on $\alpha$) such that for $\mathbf{N}$ large enough we have
\begin{equation*}
\# \mathcal{X} \leq C (\log \mathbf{N})^{2J} \mathbf{N}^{\frac{2 c \rho \sqrt{\alpha} }{\log \va{\lambda_+}}}.
\end{equation*}
\end{lemm}
\begin{proof}
By~\eqref{e:def_Z}, a word $\mathtt{v} = v_0 \dots v_{T_0 - 1}$ belongs to $\mathcal{W}(T_0) \setminus \mathcal{Z}$ if and only if the set $\set{j \in \set{0,\dots,T_0 - 1}: v_j = 1}$ has fewer than $\alpha T_0$ elements. Hence, assuming $\alpha < 1/2$ and recalling
the definition~\eqref{e:def_times} of~$T_0$, we have
\begin{equation*}
\begin{split}
\# \big(\mathcal{W}(T_0) \setminus \mathcal{Z}\big) & \leq \sum_{0 \leq \ell \leq \alpha T_0} \begin{pmatrix} T_0 \\ \ell \end{pmatrix} \leq \p{\alpha T_0 + 1} \begin{pmatrix} T_0  \\ \lceil \alpha T_0 \rceil \end{pmatrix} \\
   & \leq C \log \mathbf{N} \exp\big( - (\alpha \log \alpha + (1 - \alpha) \log(1 - \alpha)) T_0\big) \\
   & \leq C \mathbf{N}^{\frac{c \rho \sqrt{\alpha}}{J \log \va{\lambda_+}}} \log \mathbf{N}.
\end{split}
\end{equation*}
Here, we applied Stirling's formula and the constant $c > 0 $ is such that
\begin{equation*}
- (\alpha \log \alpha + (1 - \alpha) \log(1 - \alpha)) \leq c \sqrt{\alpha} \quad \textup{ for all } \quad \alpha \in (0,1).
\end{equation*}
The result then follows from the fact that $\# \mathcal{X} = \#( \mathcal{W}(T_0) \setminus \mathcal{Z})^{2J}$.
\end{proof}
We have now all the tools required to prove Lemma \ref{l:uncontrolled}.
\begin{proof}[Proof of Lemma \ref{l:uncontrolled}]
We choose $\alpha > 0$ small enough so that
\begin{equation}\label{e:condition_alpha}
\frac{2 c \rho \sqrt{\alpha} }{\log \va{\lambda_+}} < \beta,
\end{equation}
where $\beta$ is from Lemma \ref{l:bound_Bw} and $c$ is from Lemma \ref{l:cardinal_X}.
Then we see that $B_{\mathcal{X}}$ is the sum of at most $\# \mathcal{X}$ terms each of whose operator norms is $\mathcal{O}(\mathbf{N}^{- \beta})$. Hence, we have for some $C > 0$,
\begin{equation*}
\n{B_{\mathcal{X}}}_{\mathcal{H} \to \mathcal{H}} \leq C (\log \mathbf{N})^{2J} \mathbf{N}^{- \beta + \frac{2 c \rho \sqrt{\alpha} }{\log \va{\lambda_+}}},
\end{equation*}
and the lemma follows due to \eqref{e:condition_alpha}.
\end{proof}

\section{Decay for long words}\label{s:decay_long_words}

In this section, we use the fractal uncertainty principle, Proposition~\ref{p:fractal_uncertainty_principle} below, to prove Proposition~\ref{l:FUP_type} and end the proof of Theorems \ref{t:general} and \ref{t:damped}. In~\S\ref{s:fup}, we recall the definition of porosity and the statement of the fractal uncertainty principle. In~\S \ref{s:porosity}, we establish porosity estimates for the supports of $b_-$ and $b_+$ from Proposition~\ref{l:FUP_type}, which allows us to use the fractal uncertainty principle in~\S\ref{s:proof_lemma} to prove Proposition~\ref{l:FUP_type}.

\subsection{Fractal uncertainty principle}\label{s:fup}

The central tool of our proof is the fractal uncertainty principle, due
originally to Bourgain--Dyatlov~\cite{fullgap}. Roughly speaking,
it states that a function in $L^2(\mathbb R)$ cannot be localized in both
position and (semiclassically rescaled) frequency near a fractal set.
To make the statement precise, we use the following
\begin{defi}
  \label{d:porous}
Let $\nu \in (0,1)$ and $\tau_0 \leq \tau_1$ be nonnegative real numbers. Let $X$ be a subset of $\R$. We say that $X$ is \emph{$\nu$-porous on scales $\tau_0$ to $\tau_1$} if for every interval $I \subset \R$ of length $\va{I} \in \left[\tau_0,\tau_1\right]$, there is a subinterval $J\subset I$ of length $\va{J} = \nu \va{I}$ such that $J \cap X = \emptyset$.
\end{defi}
We also recall the definition of the $1$-dimensional semiclassical Fourier transform:
\begin{equation}
  \label{e:F-h-def}
\mathcal{F}_h f(x) = (2 \pi h)^{- \frac{1}{2}} \int_{\R} e^{-\frac{i}{h} x \eta} f(\eta)\,d\eta,\quad
f\in L^2(\mathbb R).
\end{equation}
Denote by $\indic_X:L^2(\mathbb R)\to L^2(\mathbb R)$ the multiplication
operator by the indicator function of~$X$. We use the following
extension of the fractal uncertainty principle of~\cite{fullgap} proved by
Dyatlov--Jin--Nonnenmacher~\cite[Proposition~2.10]{varfup}
(where we put $\gamma_0^\pm:=\varrho$, $\gamma_1^\pm:=0$
in the notation of~\cite{varfup}):
\begin{prop}\label{p:fractal_uncertainty_principle}
Let $\nu \in (0,1)$ and $ \varrho \in (\frac{1}{2},1]$. Then there exist $C,\beta > 0$ such that for every $h\in (0,1)$ and every $X,Y \subset \R$ which are $\nu$-porous on scales $h^\varrho$ to~$1$, we have
\begin{equation*}
\n{\indic_X \mathcal{F}_h \indic_Y}_{L^2\p{\R} \to L^2\p{\R}} \leq C h^\beta.
\end{equation*}
\end{prop}
\Remark The condition that $\varrho>{1\over 2}$ is essential.
Indeed, the sets
$X=Y=[-{1\over 10}\sqrt h,{1\over 10}\sqrt h]$ are ${1\over 3}$-porous
on scales~$\sqrt h$ to~1, but
$\n{\indic_X\mathcal F_h\indic_Y}_{L^2\to L^2}$ does not
go to~0 as $h\to 0$, as can be checked by applying the operator in question
to the function $h^{-{1\over 4}}\chi(h^{-{1\over 2}}x)$
where $\chi\in \CIc((-{1\over 10},{1\over 10}))$ has $L^2$ norm~1.

\subsection{Porosity property}\label{s:porosity}

In order to use Proposition~\ref{p:fractal_uncertainty_principle}, we need to establish certain porosity properties for sets related to the support of $b_-$ and $b_+$ from Proposition~\ref{l:FUP_type}. Recall that the symbols~$b_\pm$ are defined in~\eqref{e:def_bplus_bminus} using arbitrary words $\w_\pm\in\mathcal W(T_1)$, where the long logarithmic propagation time $T_1$ is defined in~\eqref{e:def_times}.
The functions~$b_1,b_2$ used in the definitions of~$b_\pm$ satisfy~\eqref{e:b-assumptions}
and the complements $\mathbb T^{2n}\setminus \supp b_1$, $\mathbb T^{2n}\setminus\supp b_2$
are both safe in the sense of Definition~\ref{d:safe}.

Fix eigenvectors $e_\pm\in\mathbb R^{2n}\setminus\{0\}$ of~$A$ associated to the eigenvalues $\lambda_\pm$ (see~\eqref{e:lambda-pm-def}).
Note that by~\eqref{e:L-pm-sum} and~\eqref{e:L-pm-perp}
we have $\sigma(e_+,e_-)\neq 0$. We choose $e_\pm$
so that we get the following identity used in~\S\ref{s:fup-endgame} below:
\begin{equation}
  \label{e:e-pm-norm}
\sigma(e_+,e_-)=1.
\end{equation}
We let $\varphi_\pm^t$ be the translation flows on the torus corresponding to $e_\pm$, that is
\begin{equation}
\label{e:phi-pm-def}
\varphi_{\pm}^{t} (z) = z + t e_{\pm} \bmod \Z^{2n}\quad\text{for}\quad
z\in\mathbb T^{2n},\ t\in\mathbb R.
\end{equation}
Recall that the subtori $\mathbb T_\pm$
are defined as the projections to $\mathbb T^{2n}$ of the
spaces $V_\pm\otimes\mathbb R$ where $V_\pm\subset\mathbb Q^{2n}$
are minimal subspaces such that $e_\pm\in V_\pm\otimes\mathbb R$.
\begin{lemm}
\label{l:minimality}
Let $z \in \T^{2n}$. Then the closure in~$\mathbb T^{2n}$ of the orbit of $z$ under $\varphi_\pm^t$
is
\begin{equation}
  \label{e:minimality}
\overline{\{\varphi_\pm^t(z)\mid t\in\mathbb R\}}=z+\T_\pm.
\end{equation}
\end{lemm}
\begin{proof}
Let us consider for instance the case of $\varphi_+^t$. By~\eqref{e:phi-pm-def},
it suffices to show that $G=\T_+$ where
$$
G:=\overline{\mathbb Re_+\bmod \mathbb Z^{2n}}\ \subset\ \mathbb T^{2n}.
$$
Since $\mathbb Re_+\bmod\mathbb Z^{2n}\subset\T_+$, we have $G \subset \T_+$, and let us prove the reciprocal inclusion.

The set $G$ is a closed subgroup of $\T^{2n}$, thus it is a Lie subgroup.
Let $\mathfrak g\subset\mathbb R^{2n}$ be the Lie algebra of $G$; since
$G$ is connected, the exponential map $\mathfrak g\to G$ is onto
and thus $G\simeq\mathfrak g/Z$ where $Z:=\mathfrak g\cap \mathbb Z^{2n}$.
Since $G$ is compact, the rank of the lattice $Z$ is equal to the dimension
of $\mathfrak g$, thus $\mathfrak g=V\otimes\mathbb R$
where $V\subset\mathbb Q^{2n}$ is the subspace generated by~$Z$.
Since $e_+\in\mathfrak g$, by the definition of $V_+$ we have
$V_+\subset V$. This implies that $V_+\otimes\mathbb R\subset\mathfrak g$
and thus $\mathbb T_+\subset G$ as needed.
\end{proof}
We now fix a constant $C_0>0$ (to be chosen in Step~2 of the proof of Lemma~\ref{l:coltar_fup} in~\S\ref{s:fup-endgame} below) and introduce for $z \in \T^{2n}$ the $h$-dependent sets 
\begin{equation}
  \label{e:Omega-pm-def}
\Omega_{\pm}(z) := \set{t \in \R\colon \exists v \in L_{\mp}\text{ such that }\va{v} \leq C_0h^{\rho'} \textup{ and } \varphi_{\pm}^t (z)+ v \in \textup{supp } b_{\mp}}.
\end{equation}
Here $\mathbf N=(2\pi h)^{-1}$ as before;
$L_{\pm}$ and $\rho'$ have been introduced in~\S\ref{s:propagation-times}, see in particular \eqref{e:L-pm-def} and \eqref{e:conditions_rho}.
We can visualize the sets $\Omega_\pm(z)$ as follows: let us lift $\supp b_\mp$
to a subset of $\mathbb R^{2n}$ and $z$ to a point in $\mathbb R^{2n}$. The set
$$
\{z + te_\pm + v\colon t\in\mathbb R,\ v\in L_\mp,\ |v|\leq C_0h^{\rho'}\}
$$
is a cylinder in $\mathbb R^{2n}=\mathbb Re_\pm\oplus L_\mp$,
and $\Omega_\pm(z)$ is the projection onto the $\mathbb R$ direction
of the intersection of this cylinder with $\supp b_\mp$.

The porosity statement needed in order to apply Proposition~\ref{p:fractal_uncertainty_principle} is the following
\begin{lemm}\label{l:porosity}
Let $\varrho \in (0, \rho)$. Then there exist $\nu,h_0\in (0,1)$, independent
of $\mathbf N,\w$
 such that, if $0 < h \leq h_0$, then for every $z \in \T^{2n}$, the sets $\Omega_{+}(z)$ and $\Omega_-(z)$ are $\nu$-porous on scales $h^\varrho$ to $1$.
\end{lemm}
\Remarks
1. Lemma~\ref{l:porosity} is where we use the assumption
that the complements $\mathbb T^{2n}\setminus \supp b_1$, $\mathbb T^{2n}\setminus\supp b_2$
are both safe.

\noindent 2. The choice of scales in Lemma~\ref{l:porosity} can be explained as follows
using~\eqref{e:special-times-condition}
(taking $\Omega_+$ to simplify notation).
On one hand, since $Ae_+=\lambda_+ e_+$,
the map $A^{T_1}$ expands the vector $e_+$
by $|\lambda_+|^{T_1}\sim h^{-\rho}\gg h^{-\varrho}$. Thus we expect
porosity of $\supp b_-$ in the direction of $e_+$ on scales from $h^\varrho$ to~1.
On the other hand, by~\eqref{e:L-pm-norm},
the same map $A^{T_1}$ sends the ball $\{v\in L_-\colon |v|\leq C_0h^{\rho'}\}$
to a set of diameter $\leq C_0\gamma^{T_1}h^{\rho'}\ll 1$,
so changing $\varphi_+^t(z)$ by an element of this ball does not change
much the forward trajectory under $A$ up to time~$T_1$, which is used
to define the symbol $b_-$ in~\eqref{e:def_bplus_bminus}.
\begin{proof}
We show the porosity of~$\Omega_+(z)$, with the case of $\Omega_-(z)$
handled similarly (reversing the direction of time).

\noindent 1. Since the complements $\mathbb T^{2n}\setminus\supp b_1$,
$\mathbb T^{2n}\setminus\supp b_2$ are safe, by Lemma~\ref{l:safe-cpct} there exist compact
subsets $K_1,K_2\subset \mathbb T^{2n}$ such that the interiors $K_1^\circ,K_2^\circ$ are safe and
$$
K_1\cap\supp b_1=K_2\cap\supp b_2=\emptyset.
$$
We claim that there exist constants $R>1$, $r>0$ such that for any
$\ell=1,2$, the intersection of every length~$R$
flow line of $\varphi_+^t$ with $K_\ell$ contains a segment of length $r$.
(Here the length of flow lines is defined using the parametrization by~$t$.)

We argue by contradiction: assume that
such $R,r$ do not exist, then there is a sequence
$z_m\in\mathbb T^{2n}$ such that for each~$m$
the intersection $K_\ell\cap \{\varphi_+^t(z_m)\colon |t|\leq m\}$ does
not contain any segment of length $1/m$. Passing to a subsequence, we may
assume that the sequence $z_m$ converges to some $z_\infty\in\mathbb T^{2n}$ as $m\to\infty$.
Since the interior~$K_\ell^\circ$ is safe, it intersects $z_\infty+\mathbb T_+$.
Then by Lemma~\ref{l:minimality} there exists $t\in\mathbb R$ such that
$\varphi_+^t(z_\infty)\in K_\ell^\circ$. 
For $m$ large
enough the segment $\big\{\varphi_+^s(z_m)\colon s\in [t,t+1/m]\big\}$
lies inside $K_\ell^\circ$. This contradicts our assumption and proves the claim.

\noindent 2. Let $z\in\mathbb T^{2n}$. We will show that $\Omega_+(z)$
is $\nu$-porous on scales $h^\varrho$ to~1 for
$$
\nu:={r\over R|\lambda_+|}.
$$
Let $I \subset \R$ be an interval of length between $h^\varrho$ and $1$. Let $j$ denote the smallest integer such that $\va{\lambda_+}^j \va{I} \geq R$. By
the definition~\eqref{e:def_times} of $T_1$,
and recalling that $\mathbf N=(2\pi h)^{-1}$,
\begin{equation*}
\va{\lambda_+}^{T_1 - 1} \va{I} \geq (2 \pi)^{-\rho} \va{\lambda_+}^{- (1 +J)} h^{\varrho - \rho} \underset{h \to 0}{\to} + \infty,
\end{equation*}
thus $0<j <T_1$, provided $h$ is small enough.

Since $Ae_+=\lambda_+e_+$,
the set $A^{j}\p{\set{\varphi_+^t(z): t \in I}}$ is a flow line of $\varphi_+$ of length $\va{\lambda_+}^j \va{I} \geq R$. Consequently, the intersection of this set with $K_{w_j^-}$
contains a segment of length $r$. It follows that
there exists a segment $J\subset I$ of length $|\lambda_+|^{-j}r$
such that
\begin{equation}
  \label{e:porint}
A^j\varphi_+^t(z)\in K_{w_j^-}\quad\text{for all}\quad t\in J.
\end{equation}
By our choice of $j$, we have $|J|\geq \nu |I|$.
It remains to prove that $J\cap\Omega_+(z)=\emptyset$.
Recalling the definition~\eqref{e:Omega-pm-def} of~$\Omega_+(z)$,
we see that it suffices to show that for each $t\in J$ we have
$$
\varphi_+^t(z)+v\not\in \supp b_-
\quad\text{for all}\quad
v\in L_-\quad\text{such that}\quad
|v|\leq C_0h^{\rho'}.
$$
We have $A^j(\varphi_+^t(z) + v) = A^j \varphi_+^t (z) + A^j v$.
Recalling the bound~\eqref{e:L-pm-norm}
on the norm of~$A^j$ restricted to~$L_-$, as well as the condition~\eqref{e:conditions_rho} we imposed on $\rho'$, we see that (with the constant $C$ depending on~$C_0$)
\begin{equation*}
\va{A^j v} \leq C \gamma^j h^{\rho'} \leq C \gamma^{T_1} h^{\rho'} \leq C(2 \pi)^{-\rho \frac{\log \gamma}{\log \va{\lambda_+}}} h^{\rho' - \rho \frac{\log \gamma}{\log \va{\lambda_+}}} \underset{h \to 0}{\to} 0. 
\end{equation*}
By~\eqref{e:porint} and since $K_{w_j^-}\cap \supp b_{w_j^-}=\emptyset$,
we see that $A^j(\varphi_+^t(z) +v)\not\in \supp b_{w_j^-}$, provided $h$ is small enough
so that $|A^jv|$ is less than the distance between $K_{w_j^-}$ and $\supp b_{w_j^-}$.
Recalling the definition \eqref{e:def_bplus_bminus} of $b_-$, it follows that
$\varphi_+^t(z)+v\not\in \supp b_-$, which finishes the proof.
\end{proof}

\subsection{Proof of Proposition~\ref{l:FUP_type}}\label{s:proof_lemma}

We now give the proof of Proposition~\ref{l:FUP_type}, relying on the
following three ingredients:
\begin{itemize}
\item the fact that for any $\varepsilon>0$, we have $b_\pm\in S_{L_\pm,\rho+\varepsilon,\rho'+\varepsilon}(\mathbb T^{2n})$
uniformly in $h$ and in the words $\w_\pm$
(see Lemma~\ref{l:long_logarithmic_words});
\item the porosity property of the supports $\supp b_\pm$ given by Lemma~\ref{l:porosity};
\item and the fractal uncertainty princple in the form of Proposition~\ref{p:fractal_uncertainty_principle}.
\end{itemize}
Henceforth we treat $b_\pm$ as $\mathbb Z^{2n}$-invariant
symbols in $S_{L_\pm,\rho+\varepsilon,\rho'+\varepsilon}(\mathbb R^{2n})$.
All the constants in the estimates below are independent of $h,\w_\pm$.

\subsubsection{Decomposing the operator and the scheme of the proof}

We start by decomposing the operator $\Op_h(b_-)\Op_h(b_+)$ into
a series, see~\eqref{e:totop-decomposition} below.
For that, fix a function
$$
\widetilde\psi\in \CIc((-1,1)^{2n};\mathbb R),\quad
\sum_{k\in\mathbb Z^{2n}}\widetilde\psi(z-k)^2=1\quad\text{for all}\quad
z\in\mathbb R^{2n}.
$$
For instance, we can start with $\chi\in \CIc((-1,1)^{2n};[0,1])$ such that
the $\mathbb Z^{2n}$-periodic function $F(x):=\sum_{k\in\mathbb Z^{2n}}\chi(x-k)^2$
is everywhere positive, and put $\widetilde\psi(x):=F(x)^{-{1\over 2}}\chi(x)$.

Now, consider the partition of unity $1=\sum_{k\in\mathbb Z^{2n}}\psi_k^2$
where the $h$-dependent symbol $\psi_k\in\CIc(\mathbb R^{2n})$ is given by
\begin{equation}
  \label{e:psi-k-def}
\psi_k(z):=\widetilde\psi\Big({z\over h^{\rho'}}-k\Big),\quad
k\in\mathbb Z^{2n}.
\end{equation}
Recalling Definition~\ref{d:S-L}, we see that $\psi_k$ lies in $S_{L,\rho,\rho'}(\mathbb R^{2n})$ uniformly in~$h,k$ for any coisotropic subspace $L\subset\mathbb R^{2n}$.

We have the following decomposition, with the series
below converging in the strong operator topology
as an operator $L^2(\mathbb R^n)\to L^2(\mathbb R^n)$,
which can be checked by applying it to a function in $\mathscr S(\mathbb R^{n})$:
\begin{equation}
  \label{e:totop-decomposition}
\Op_h(b_-)\Op_h(b_+)=\sum_{k\in\mathbb Z^{2n}}P_k\quad\text{where}\quad P_k:=\Op_h(b_-)\Op_h(\psi_k^2)\Op_h(b_+).
\end{equation}
We now state two estimates on the operators $P_k$ which together will give
Proposition~\ref{l:FUP_type}.
The first one is an almost orthogonality type statement when $|k-\ell|$ is sufficiently large:
\begin{lemm}\label{l:non_stationary}
For every $m > 0$ there exists a constant $C_m >0$ such that for every~$k,\ell \in \Z^{2n}$
such that $|k-\ell|\geq 10\sqrt n$ we have
\begin{equation*}
\n{P_k^* P_\ell}_{L^2 \to L^2} \leq C_m h^m \va{k - \ell}^{-m} \quad \textup{ and } \quad \n{P_k P_\ell^*}_{L^2 \to L^2} \leq C_m h^m \va{k - \ell}^{-m}.
\end{equation*}
\end{lemm}
The second one is the norm bound on each individual $P_k$,
which uses the fractal uncertainty principle:
\begin{lemm}\label{l:coltar_fup}
There exist constants $C,\beta>0$ such that for every $k\in \Z^{2n}$ we have
\begin{equation*}
\|P_k\|_{L^2\to L^2} \leq C h^\beta.
\end{equation*}
Here the constant $\beta$ only depends on the porosity constant $\nu$
in Lemma~\ref{l:porosity} and on~$\rho$.
\end{lemm}
Before proving Lemmas~\ref{l:non_stationary} and~\ref{l:coltar_fup}, let us explain how they imply Proposition~\ref{l:FUP_type}:
\begin{proof}[Proof of Proposition~\ref{l:FUP_type}]
It follows from Lemma~\ref{l:non_stationary} (with $m:=4n+1$) and Lemma~\ref{l:coltar_fup} that there are
constants $C,\beta > 0$ such that
\begin{equation*}
\sup_{k \in \Z^{2n}} \sum_{\ell \in \Z^{2n}} \n{P_k^* P_\ell}_{L^2 \to L^2}^{\frac{1}{2}} \leq C h^\beta \quad \textup{ and } \quad \sup_{k \in \Z^{2n}} \sum_{\ell \in \Z^{2n}} \n{P_k P_\ell^*}_{L^2 \to L^2}^{\frac{1}{2}} \leq C h^\beta.
\end{equation*}
Hence, it follows from the Cotlar--Stein Theorem~\cite[Theorem C.5]{Zworski-Book} 
and the decomposition~\eqref{e:totop-decomposition}
that $\|\Op_h(b_-)\Op_h(b_+)\|_{L^2\to L^2}\leq Ch^\beta$ as needed.
\end{proof}

\subsubsection{Almost orthogonality}

We are left with the proofs of Lemmas \ref{l:non_stationary} and \ref{l:coltar_fup}. We start with Lemma \ref{l:non_stationary}:
\begin{proof}[Proof of Lemma \ref{l:non_stationary}]
1. Let $k,\ell \in \Z^{2n}$ be such that $|k - \ell| \geq 10\sqrt n$.
Define the linear functional $q:\mathbb R^{2n}\to\mathbb R$ by
$$
q(z)={\langle z,\ell-k\rangle\over |k-\ell|},\quad
z\in\mathbb R^{2n}.
$$
Note that $q$ has norm~1. Putting $r_0:=h^{\rho'}q({k+\ell\over 2})$, we have
\begin{equation}
  \label{e:nonst-int-1}
\begin{aligned}
\supp \psi_k&\subset \big\{z\in \mathbb R^{2n}\,\big|\, q(z)\leq r_0-h^{\rho'}\textstyle{|k-\ell|\over 4}\},\\
\supp \psi_\ell&\subset \big\{z\in \mathbb R^{2n}\,\big|\, q(z)\geq r_0+h^{\rho'}\textstyle{|k-\ell|\over 4}\}.
\end{aligned}
\end{equation}
Indeed, assume that $z\in\supp\psi_k$. Then
$|q(h^{-\rho'}z-k)|\leq |h^{-\rho'}z-k|\leq \sqrt{2n}$.
Since $h^{\rho'}q(k)=r_0-h^{\rho'}{|k-\ell|\over 2}$
and $|k-\ell|\geq 10\sqrt n$, we have
$q(z)\leq r_0-h^{\rho'}{|k-\ell|\over 4}$.
This gives the first statement in~\eqref{e:nonst-int-1}, with the second one
proved similarly.
Note that~\eqref{e:nonst-int-1} implies in particular
that $\supp\psi_k\cap\supp\psi_\ell=\emptyset$.

\noindent 2. We estimate the norm $\|P_k^*P_\ell\|_{L^2\to L^2}$,
with $\|P_kP_\ell^*\|_{L^2\to L^2}$ estimated in a similar way.
We write using~\eqref{e:basic-product} and~\eqref{e:basic-adjoint}
\begin{equation*}
P_k^* P_\ell =
\Op_h(\bar{b}_+) \Op_h(\psi_k^2) \Op_h(\bar{b}_- \# b_-) \Op_h(\psi_\ell^2) \Op_h(b_+).
\end{equation*}
It follows from Lemmas~\ref{l:quantization-properties} and~\ref{l:long_logarithmic_words} that $\Op_h(\bar{b}_+)$ is uniformly bounded on $L^2$. 
Thus it suffices to show that
\begin{equation}
  \label{e:nonst-int-2}
\|\Op_h(\psi_k^2) \Op_h(\bar{b}_- \# b_-) \Op_h(\psi_\ell^2)\|_{L^2\to L^2}
\leq C_m h^m|k-\ell|^{-m}.
\end{equation}
To show~\eqref{e:nonst-int-2}, it suffices to apply Lemma~\ref{l:nonint-sup-2} with
$$
r := h^{\rho'} \textstyle\frac{|k - \ell|}{4},\quad
a := \psi_k^2,\quad
b:=\psi_\ell^2,\quad
c := \bar{b}_- \# b_- .
$$
Here for each $\varepsilon>0$ the symbols $a,b,c$ are bounded in the class $S_{L_-,\rho+\varepsilon,\rho'+\varepsilon}(\mathbb R^{2n})$ uniformly in~$h$ by Lemmas~\ref{l:quantization-properties} and~\ref{l:long_logarithmic_words}, as well as~\eqref{e:psi-k-def}.
The support condition of Lemma~\ref{l:nonint-sup-2}
is satisfied by~\eqref{e:nonst-int-1}.
\end{proof}

\subsubsection{Decay for an individual summand}
\label{s:fup-endgame}

Finally, we apply the fractal uncertainty principle to prove Lemma~\ref{l:coltar_fup}.
\begin{proof}[Proof of Lemma \ref{l:coltar_fup}]
\noindent 1. Recall that the $\psi_k$'s belong uniformly to both the symbol classes $S_{L_+,\rho,\rho'}$ and $S_{L_-,\rho,\rho'}$. Recalling Lemma \ref{l:long_logarithmic_words}, we can apply the product formula from Lemma \ref{l:quantization-properties} to find that
for every $\varepsilon>0$
\begin{equation*}
P_k = \Op_h(b_- \psi_k) \Op_h(b_+ \psi_k) + \mathcal{O}(h^{1 - \rho - \rho'-\varepsilon})_{L^2\to L^2}.
\end{equation*}
Therefore it suffices to show that
\begin{equation}
  \label{e:endgame-fup}
\|\Op_h(b_-\psi_k)\Op_h(b_+\psi_k)\|_{L^2\to L^2}\leq Ch^\beta.
\end{equation}

\noindent 2. We next study the supports of the symbols $b_\pm\psi_k$.
We have from~\eqref{e:psi-k-def}
$$
\supp\psi_k\ \subset\ h^{\rho'}k+(-h^{\rho'},h^{\rho'})^{2n}.
$$
Thus by~\eqref{e:L-pm-sum} any $z\in\supp\psi_k$ can be written as
$z=h^{\rho'}k+t_\pm e_\pm+v_\mp$ where $t_\pm\in\mathbb R$,
$v_\mp\in L_\mp$, and $|v_\mp|\leq C_0h^{\rho'}$ for some
constant $C_0$ depending only on the matrix~$A$.
Choose $s_\pm^{(k)}\in\mathbb R$ such that $h^{\rho'}k\in s_\pm^{(k)}e_\pm+L_\mp$.
Put $z^{(k)}:=h^{\rho'}k\bmod\mathbb Z^{2n}\in\mathbb T^{2n}$.
Recalling the definition~\eqref{e:Omega-pm-def} of the sets~$\Omega_\pm(z)$,
we get
\begin{equation}
  \label{e:supporter-1}
\supp(b_\mp\psi_k)\ \subset\ \bigcup_{t\in\widetilde\Omega_\pm}\big(t e_\pm+L_\mp\big)\quad\text{where}\quad
\widetilde\Omega_\pm:=s_\pm^{(k)}+\Omega_\pm(z^{(k)})\ \subset\ \mathbb R.
\end{equation}

\noindent 3. We now conjugate by a metaplectic transformation which
`straightens out' the vectors $e_\pm$ and the subspaces $L_\pm$. Using \eqref{e:L-pm-perp},
\eqref{e:e-pm-norm}, and the linear version of Darboux's Theorem,
we construct a symplectic matrix $Q \in \Sp(2n,\R)$ such that
\begin{itemize}
\item $Q \partial_{x_1} = e_-$ and $Q \partial_{\xi_1} = e_+$;
\item $Q \Span (\partial_{x_1},\dots, \partial_{x_n},\partial_{\xi_2},\dots,\partial_{\xi_n}) = L_-$;
\item $Q \Span (\partial_{x_2},\dots,\partial_{x_n},\partial_{\xi_1},\dots,\partial_{\xi_n}) = L_+$.
\end{itemize}
Let $\widetilde M\in\mathcal M_Q$ be a metaplectic operator
associated to~$Q$ (see~\S\ref{s:metaplectic}). 
Then by~\eqref{e:egorov-metaplectic}
\begin{equation*}
\begin{split}
\n{\Op_h(b_- \psi_k) \Op_h(b_+ \psi_k)}_{L^2 \to L^2} & = \n{\widetilde M^{-1} \Op_h(b_- \psi_k) \Op_h(b_+ \psi_k) \widetilde M}_{L^2 \to L ^2} \\
    & = \n{\Op_h\big((b_-\psi_k)\circ Q\big) \Op_h\big((b_+\psi_k)\circ Q\big)}_{L^2 \to L ^2}.
\end{split}
\end{equation*}
Thus~\eqref{e:endgame-fup} reduces to
\begin{equation}
  \label{e:endgame-fup-2}
\n{\Op_h\big((b_-\psi_k)\circ Q\big) \Op_h\big((b_+\psi_k)\circ Q\big)}_{L^2 \to L ^2}\leq Ch^\beta
\end{equation}
and the support condition~\eqref{e:supporter-1} becomes
\begin{equation}
  \label{e:supporter-2}
\begin{aligned}
\supp\big((b_-\psi_k)\circ Q\big)\ &\subset\ \{ (x,\xi)\mid \xi_1\in \widetilde\Omega_+\},\\
\supp\big((b_+\psi_k)\circ Q\big)\ &\subset\ \{  (x,\xi)\mid x_1\in \widetilde\Omega_-\}.
\end{aligned}
\end{equation}

\noindent 4. For $\delta>0$, denote the $\delta$-neighborhood of $\widetilde\Omega_\pm$ by
$$
\widetilde\Omega_\pm(\delta):=\widetilde\Omega_\pm+[-\delta,\delta].
$$
Let $\chi_\mp$ be the convolutions of
the indicator functions of~$\widetilde\Omega_\pm({1\over 2}h^{\rho})$ with the function
$h^{-\rho}\chi(h^{-\rho}t)$ where $\chi\in \CIc((-{1\over 2},{1\over 2}))$ is a nonnegative
function integrating to~1. 
Then
$$
\chi_\mp\in C^\infty(\mathbb R;[0,1]),\quad
\supp\chi_\mp\subset \widetilde\Omega_\pm(h^\rho),\quad
\chi_\mp=1\quad\text{on}\quad \widetilde\Omega_\pm
$$
and for each $\ell$ there exists a constant $C_\ell$ (depending only on $\ell$
and the choice of $\chi$) such that
$$
\sup_{t\in\mathbb R} |\partial_t^\ell \chi_\pm(t)|\leq C_\ell h^{-\rho\ell}.
$$
Define the symbols $\widetilde \chi_\pm\in C^\infty(\mathbb R^{2n})$ by
$$
\widetilde\chi_-(x,\xi)=\chi_-(\xi_1),\quad
\widetilde\chi_+(x,\xi)=\chi_+(x_1).
$$
Then $\widetilde\chi_\pm$ lie in the symbol class $S_{Q^{-1}L_\pm,\rho,0}(\mathbb R^{2n})$
uniformly in~$h$. On the other hand, by Lemma~\ref{l:long_logarithmic_words}
the symbols $(b_\pm\psi_k)\circ Q$ lie in the larger class $S_{Q^{-1}L_\pm,\rho+\varepsilon,\rho'+\varepsilon}(\mathbb R^{2n})$ uniformly in~$h$ for every fixed $\varepsilon>0$.
By Lemma~\ref{l:quantization-properties} and since
$(b_\pm\psi_k)\circ Q=\widetilde\chi_\pm((b_\pm\psi_k)\circ Q)$ by~\eqref{e:supporter-2},
we have
$$
\begin{aligned}
\Op_h\big((b_-\psi_k)\circ Q\big)&=\Op_h\big((b_-\psi_k)\circ Q\big)\Op_h(\widetilde\chi_-)+\mathcal O(h^{1-\rho-\rho'-\varepsilon})_{L^2\to L^2},\\
\Op_h\big((b_+\psi_k)\circ Q\big)&=\Op_h(\widetilde\chi_+)\Op_h\big((b_+\psi_k)\circ Q\big)+\mathcal O(h^{1-\rho-\rho'-\varepsilon})_{L^2\to L^2}.
\end{aligned}
$$
Note that $\Op_h(\widetilde\chi_+)=\chi_+(x_1)$ is a multiplication operator
and $\Op_h(\widetilde\chi_-)=\chi_-(-ih\partial_{x_1})$ is a Fourier multiplier
(see~\cite[Theorem~4.9]{Zworski-Book}).
Multiplying the above estimates and using that $\Op_h((b_\pm\psi_k)\circ Q)$ are bounded uniformly
as operators on~$L^2(\mathbb R^n)$, we reduce~\eqref{e:endgame-fup-2} to the following estimate:
\begin{equation}
  \label{e:endgame-fup-3}
\|\chi_-(-ih\partial_{x_1})\chi_+(x_1)\|_{L^2(\mathbb R^n)\to L^2(\mathbb R^n)}\leq Ch^\beta.
\end{equation}

\noindent 5. If we consider $L^2(\mathbb R^n)$ as the Hilbert tensor product
$L^2(\mathbb R)\otimes L^2(\mathbb R^{n-1})$ corresponding to the decomposition
$x=(x_1,x')$, $x':=(x_2,\dots,x_n)$, then
the operators $\chi_+(x_1)$ and $\chi_-(-ih\partial_{x_1})$
are the tensor products of the same operators in one variable
with the identity operator on $L^2(\mathbb R^{n-1})$. Thus~\eqref{e:endgame-fup-3} is equivalent to
\begin{equation}
  \label{e:endgame-fup-4}
\|\chi_-(-ih\partial_{x_1})\chi_+(x_1)\|_{L^2(\mathbb R)\to L^2(\mathbb R)}\leq Ch^\beta
\end{equation}
where we now treat the factors in the product as operators on $L^2(\mathbb R)$.
Denote by~$\mathcal F_h:L^2(\mathbb R)\to L^2(\mathbb R)$ the unitary semiclassical
Fourier transform, see~\eqref{e:F-h-def}. Then
$\chi_-(-ih\partial_{x_1})=\mathcal F_h^{-1}\chi_-(x_1)\mathcal F_h$.
Thus the left-hand side of~\eqref{e:endgame-fup-4} is equal
to $\|\chi_-(x_1)\mathcal F_h\chi_+(x_1)\|_{L^2(\mathbb R)\to L^2(\mathbb R)}$.
Since $\chi_\pm=\chi_\pm\indic_{\widetilde\Omega_\mp(h^{\rho})}$
and $|\chi_\pm|\leq 1$, the bound~\eqref{e:endgame-fup-4} reduces to
\begin{equation}
  \label{e:endgame-fup-5}
\|\indic_{\widetilde\Omega_+(h^\rho)}\mathcal F_h\indic_{\widetilde\Omega_-(h^\rho)}\|_{L^2(\mathbb R)\to L^2(\mathbb R)}\leq Ch^\beta.
\end{equation}

\noindent 6. We finally apply the fractal uncertainty principle. Fix $\varrho\in({1\over 2},\rho)$, which is possible since~$\rho>{1\over 2}$ by~\eqref{e:conditions_rho}.
By Lemma~\ref{l:porosity}, there exists $\nu>0$ such that the sets $\widetilde\Omega_\pm$ are $\nu$-porous on scales $h^\varrho$ to~1. Since $h^\rho\ll h^\varrho$ for $h\ll 1$,
the neighborhoods $\widetilde\Omega_\pm(h^\rho)$ are $\nu\over 3$-porous
on scales $h^\varrho$ to~1~-- see for example~\cite[Lemma~2.11]{varfup}.
Now~\eqref{e:endgame-fup-5} follows from the fractal uncertainty principle
of Proposition~\ref{p:fractal_uncertainty_principle}, and the proof is finished.
\end{proof}

\appendix

\section{Properties of integer symplectic matrices}\label{appendix:symplectic_matrices}

In this Appendix, we discuss the algebraic hypotheses made on the matrix $A$ in Theorems \ref{t:general}, \ref{t:damped}, and \ref{t:measures}. More precisely, we investigate the spaces $V_+$ and $V_-$ (and hence the tori $\T_+$ and $\T_-$) defined in~\eqref{e:T-pm-def}. In particular, we prove Lemma~\ref{l:irreducible} that allows us to deduce Theorem~\ref{t:basic} from Theorem~\ref{t:general}.

\subsection{Algebraic considerations} 

We start by giving a new characterization of $V_+$ and $V_-$. Let $A\in\Sp(2n,\mathbb Z)$
satisfy the spectral gap condition~\eqref{e:spectral-gap}
and recall from the introduction that $V_\pm$ were defined as the smallest subspaces
of $\mathbb Q^{2n}$ such that $E_\pm\subset V_\pm\otimes\mathbb R$
where $E_\pm\subset\mathbb R^{2n}$ are the eigenspaces of $A$ corresponding to the eigenvalues
$\lambda_+$ and $\lambda_-:=\lambda_+^{-1}$.

We will be using basic field theory, see e.g.~\cite[Chapter~13]{Dummit-Foote}. Recall
that for an algebraic number $\lambda\in\mathbb C$, its \emph{minimal polynomial}
(over $\mathbb Q$)
is the unique irreducible monic polynomial $P\in \mathbb Q[x]$ such that
$P(\lambda)=0$. Two algebraic numbers are called \emph{Galois conjugates} if
they have the same minimal polynomial.
\begin{lemm}\label{l:characterization_Vpm}
Let $P_{\pm}$ denote the minimal polynomials of $\lambda_{\pm}$. Then $V_{\pm} = \ker P_{\pm}(A)$. The dimensions of $V_\pm$ are equal to each other and to the degrees of $P_\pm$.
Moreover, we have the following two cases:
\begin{enumerate}
\item if $\lambda_+$ is a Galois conjugate of $\lambda_-$, then $V_+=V_-$;
\item otherwise $V_+\cap V_-=\{0\}$.
\end{enumerate}
\end{lemm}
\begin{proof}
1. We first show that $V_+=\ker P_+(A)$ and $\dim V_+=\deg P_+$. (The case of $V_-$ is treated similarly.)
Note that $\ker P_+(A)$ is an $A$-invariant subspace of $\mathbb{Q}^{2n}$.
Any (complex) eigenvalue of the endomorphism $A|_{\ker P_+(A)}$ has to be a root of $P_+$,
thus the characteristic polynomial $\widetilde P_+\in \mathbb Q[x]$ of $A|_{\ker P_+(A)}$ is a power of $P_+$. On the other hand, $\widetilde P_+$ divides the characteristic polynomial
of~$A$. Since $\lambda_+$ is a simple eigenvalue of~$A$, we see
that $\widetilde P_+=P_+$.

Since $P_+(\lambda_+) = 0$, we see that $E_+ \subset \ker P_+(A) \otimes \mathbb{R}$, and consequently we have $V_+ \subset \ker P_+(A)$. As $V_+$ is $A$-invariant,
the characteristic polynomial of the endomorphism $A|_{V_+}$ divides the characteristic polynomial $P_+$ of the endomorphism
$A|_{\ker P_+(A)}$. Since $P_+$ is irreducible over $\mathbb Q$
and $\dim V_+>0$,
we see that these two characteristic polynomials are equal. It follows
that $V_+=\ker P_+(A)$ and $\dim V_+=\deg P_+$.

\noindent 2.
Recall that the degree of $P_{\pm}$ is the dimension of the field $\mathbb{Q}(\lambda_{\pm})$ as a vector field over~$\mathbb{Q}$. Since $\lambda_+=\lambda_-^{-1}$, we have $\mathbb{Q}(\lambda_+) = \mathbb{Q}(\lambda_-)$, so that $\deg P_+ = \deg P_-$. It follows that $\dim V_+ = \dim V_-$.

\noindent 3.
Since $P_+$ and $P_-$ are irreducible over the rationals, either they are coprime, in which case $V_+ \cap V_- = \{0\}$, or they are equal, in which case $V_+ = V_-$, due to the characterization we just proved. If $P_+ = P_-$, then $P_+(\lambda_-) = 0$ and $\lambda_-$ is a Galois conjugate of $\lambda_+$. Reciprocally, if $P_+(\lambda_-) =0$, then $P_+$ and $P_-$ are not coprime, so that $P_+ = P_-$.
\end{proof}

In order to discuss the sharpness of Theorem~\ref{t:measures}, we introduce a decomposition of~$\mathbb{Q}^{2n}$. For a subspace $V\subset \mathbb Q^{2n}$, denote
by $V^{\perp\sigma}\subset\mathbb Q^{2n}$ its symplectic complement,
see~\eqref{e:perp-sigma-def}. Recall
that $V$ is called \emph{symplectic} if $V\cap V^{\perp\sigma}=\{0\}$.
\begin{lemm}\label{l:decomposition_symplectic}
We have the following two cases:
\begin{enumerate}
\item if $\lambda_+$ is a Galois conjugate of $\lambda_-$, then
$V_+=V_-$ is symplectic;
\item otherwise $V_\pm$ are both isotropic and the symplectic form $\sigma$
is nondegenerate on $V_+ + V_-$.
\end{enumerate}
Consequently, we have a decomposition of $\mathbb{Q}^{2n}$ into
\begin{equation*}
\mathbb{Q}^{2n} = V_0 \oplus V_1,
\end{equation*}
where $V_0 = V_+ + V_-$ and $V_1 = (V_+ + V_-)^{\perp \sigma}$ are symplectic. 
\end{lemm}
\begin{proof}
1. For each complex eigenvalue $\lambda\in\Spec(A)$, define the space
of generalized eigenvectors
$$
V(\lambda):=\{v\in \mathbb C^{2n}\mid \exists\ell\geq 0:\ (A-\lambda I)^\ell v=0\}.
$$
Then we have the decomposition
\begin{equation}
  \label{e:dead-composer}
\mathbb C^{2n}=\bigoplus_{\lambda\in \Spec(A)}V(\lambda).
\end{equation}
We claim that for all $\lambda,\lambda'\in\Spec(A)$
such that $\lambda\lambda'\neq 1$,
\begin{equation}
  \label{e:perpor}
\sigma(v,v')=0\quad\text{for all}\quad v\in V(\lambda),\ v'\in V(\lambda').
\end{equation}
To prove~\eqref{e:perpor}, we argue by induction on $\ell+\ell'$ where $\ell,\ell'\geq 0$
are the smallest numbers such that $(A-\lambda I)^\ell v=(A-\lambda'I)^{\ell'}v'=0$.
If $\ell=0$ or $\ell'=0$, then $\sigma(v,v')=0$ since 
$v=0$ or $v'=0$. Otherwise we use that $A$ is symplectic to write
$$
\sigma(v,v')=\sigma(Av,Av')=\lambda\lambda' \sigma(v,v')
+\sigma((A-\lambda I)v,Av')+\sigma(\lambda v,(A-\lambda'I)v').
$$
Using the inductive hypothesis we see that the last two terms on the right-hand side are~0,
which gives $\sigma(v,v')=0$ as needed.

By~\eqref{e:dead-composer} and~\eqref{e:perpor}, we see that for all $\lambda\in\Spec(A)$
\begin{equation}
  \label{e:perp-composer}
V(\lambda)^{\perp\sigma}=\bigoplus_{\lambda'\in\Spec(A),\,\lambda'\neq\lambda^{-1}}V(\lambda').
\end{equation}

\noindent 2. Since $\lambda_\pm$ is a simple eigenvalue of $A$ and
$P_\pm$ is its minimal polynomial, each root of~$P_\pm$ is a simple eigenvalue of~$A$.
By Lemma~\ref{l:characterization_Vpm} we have
$$
V_\pm\otimes\mathbb C=\bigoplus_{\lambda,\, P_\pm(\lambda)=0}V(\lambda).
$$
Since $P_\pm$ are the minimal polynomials of $\lambda_\pm$ and
$\lambda_+=\lambda_-^{-1}$, we have $P_-(\lambda)=c\lambda^{\deg P_+}P_+(\lambda^{-1})$
for some $c\in\mathbb Q\setminus \{0\}$. It follows from~\eqref{e:perp-composer} that
$$
(V_\pm\otimes\mathbb C)^{\perp\sigma}=\bigoplus_{\lambda\in\Spec(A),\,P_\mp(\lambda)\neq 0}V(\lambda).
$$
If $\lambda_+$ is a Galois conjugate of $\lambda_-$,
then $P_+=P_-$, so $V_+=V_-$ is symplectic.
Otherwise $P_+$ and $P_-$ are coprime, so $V_\pm$ are both isotropic
and $\sigma$ is nondegenerate on $V_+ + V_-$.
\end{proof}

\Remark Using Lemma \ref{l:minimality}, the algebraic consideration from this section have dynamical implication. Lemma \ref{l:characterization_Vpm} that if $z \in \mathbb{T}^{2n}$ then the closure of the orbits of $z$ for the flows $(\varphi^t_+)_{t \in \mathbb{R}}$ and $(\varphi^t_{-})_{t \in \mathbb{R}}$ are either identical (if $\lambda_+$ is a Galois conjugate of $\lambda_-$) or have a finite number of points of intersection.

From Lemma \ref{l:decomposition_symplectic}, we know that if $z \in \mathbb{T}^{2n}$ then the closure of the orbit of $z$ under the action by translation of the $2$-dimensional vector space generated by $e_+$ and $e_-$ is always a symplectic subtorus of $\mathbb{T}^{2n}$.

\subsection{Most favorable cases}\label{s:most_favorable}

Theorem~\ref{t:measures} gives a condition on the support of semiclassical measures for $A$ in terms of the spaces $V_+$ and $V_-$. The larger these spaces are, the stronger the conclusion of Theorem \ref{t:measures} is. Considering the decomposition from Lemma~\ref{l:decomposition_symplectic}, the most favorable case is when $V_1$ is trivial. In that situation, there are still two possibilities according to Lemma~\ref{l:characterization_Vpm}:
\begin{enumerate}
\item $\mathbb{Q}^{2n} = V_+ = V_-$, or
\item $\mathbb{Q}^{2n} = V_+ \oplus V_-$ and $V_\pm$ are Lagrangian.
\end{enumerate}
In case~(1), Theorem~\ref{t:measures} says that all semiclassical measures for $A$ are fully supported. Actually, this is exactly the setting of Theorem~\ref{t:basic}, as we prove now.
\begin{lemm}\label{l:irreducible}
The characteristic polynomial of $A$ is irreducible over $\mathbb{Q}$ if and only if $\mathbb{Q}^{2n} = V_+ = V_-$ (that is $\mathbb{T}^{2n} = \mathbb{T}_+ = \mathbb{T}_-$, or equivalently the flows $(\varphi_+^t)_{t \in \mathbb{R}}$ and $(\varphi_-^t)_{t \in \mathbb{R}}$ are minimal).
\end{lemm}
\begin{proof}
Notice that $P_+$ divides the characteristic polynomial of $A$ and recall that the dimension of $V_+$ is the degree of $P_+$. Hence, if $V_+$ is equal to $\mathbb{Q}^{2n}$, the degree of $P_+$ is $2n$ and $P_+$ must be the characteristic polynomial of $A$, which is consequently irreducible. Reciprocally, if the characteristic polynomial of $A$ is irreducible, it must be equal to~$P_+$, so that $V_+ = \mathbb{Q}^{2n}$.
\end{proof}
Of course, when $V_+ = V_- = \mathbb{Q}^{2n}$, the control of the support of semiclassical measures for $A$ given by Theorem~\ref{t:measures} is sharp. When $n=1$, $A$
satisfies the spectral gap condition~\eqref{e:spectral-gap} if and only if it is hyperbolic
(i.e. it has no eigenvalues on the unit circle), and in this case we always
have $\mathbb Q^2=V_+=V_-$. When $n>1$, one can easily construct
examples of matrices satisfying \eqref{e:spectral-gap} with irreducible characteristic polynomials. For example, when $n=2$ one can take
\begin{equation}
\label{e:A-example-1}
A=\begin{pmatrix}
0 & 0 & 1 &0 \\
0 & 0 & 0 & 1 \\
- 1 & 0 & 0 & 1 \\
0 & -1 & 1 & 2
\end{pmatrix}
\end{equation}
with the characteristic polynomial
$$
P(\lambda)=\lambda^4-2\lambda^3+\lambda^2-2\lambda+1=(\lambda^2-(1+\sqrt 2)\lambda+1)(\lambda^2-(1-\sqrt 2)\lambda+1)
$$
which has a root in $(0,1)$, a root in $(1,\infty)$, and two complex roots
on the unit circle.

Let us now consider the case (2), when $\mathbb{Q}^{2n} = V_+ \oplus V_-$. Our result is still sharp in this situation since, under some mild additional assumptions, Kelmer~\cite[Theorem~1]{Kelmer-cat} constructed semiclassical measures supported in some translate of $\mathbb{T}_+$ and
semiclassical measures supported in some translate of $\mathbb{T}_-$. A basic example
(previously presented by Gurevich~\cite{Gurevich-cat} and Kelmer~\cite{Kelmer-cat}) is 
\begin{equation}\label{e:isotropic_case}
A = \begin{pmatrix} B & 0 \\ 0 & B^{-T} \end{pmatrix},
\end{equation}
where $B \in \GL(n,\mathbb{Z})$, $|\det B|=1$, has irreducible characteristic polynomial and a leading simple eigenvalue, that also dominates the inverses of the eigenvalues of $B$ (so that
$A$ satisfies the spectral gap condition~\eqref{e:spectral-gap}). One can take for instance
\begin{equation*}
B = \begin{pmatrix} 0 & 1 & 0 \\ 0 & 0 & 1 \\ 1 & 1 & 0 \end{pmatrix}.
\end{equation*}
Using the coordinates $(x,\xi)$ on $\mathbb R^{2n}$, we see that
the spaces $V_\pm$ are given by $V_+=\{\xi=0\}$, $V_-=\{x=0\}$.
Note that if we allow $B$ to be in $\GL(n,\mathbb{Q})$, then, when $\mathbb{Q}^{2n} = V_+ \oplus V_-$, the matrix $A$ is always of the form \eqref{e:isotropic_case} after a symplectic (rational) change of coordinates.

Using Proposition \ref{p:explicit_quantization}, we see that for a matrix $A$ of the form \eqref{e:isotropic_case} and $\theta=0$, the following elements of $\mathcal{H}_{\mathbf{N}}(0)$ are eigenvectors for the quantizations $M_{\mathbf N,0}$ of~$A$:
\begin{equation*}\label{e:eigenvectors_example}
\mathbf{e}_0^0 \quad \textup{ and } \quad \mathbf{N}^{-\frac{n}{2}} \sum_{j \in \Z_{\mathbf{N}}^n} \mathbf{e}_j^0.
\end{equation*}
It follows from Proposition \ref{p:explicit_pseudors} that these eigenvectors converge respectively to the semiclassical measures
\begin{equation*}
a \mapsto \int_{\mathbb{T}^{n}} a(0,\xi) \,d\xi \quad \textup{ and } \quad a \mapsto \int_{\T^{n}} a(x,0) \,dx.
\end{equation*}
These measures are supported respectively in $\T_-$ and $\T_+$.

\subsection{General case}

For now, we only considered the case in which the space $V_1$ from Lemma~\ref{l:decomposition_symplectic} is trivial. Let us now discuss what happens when $V_1$ is non-trivial. Let $\T_0$ and $\T_1$ be the subtori of $\T^{2n}$ tangent respectively to $V_0$ and $V_1$. As before, we consider two cases:
\begin{enumerate}
\item If $\lambda_+$ is a Galois conjugate of $\lambda_-$, then
Theorem~\ref{t:measures} shows that the support of every semiclassical
measure contains a translate of $\T_0=\T_\pm$.
\item Otherwise Theorem~\ref{t:measures} shows that the support of
every semiclassical measure contains a translate of $\T_+$ or $\T_-$, which are different tori (their tangent spaces intersect trivially).
On the other hand, from \cite[Theorem~1]{Kelmer-cat}, we know that
(under mild additional assumptions) there are semiclassical measures supported in some translate of~$\mathbb{T}_+ +\T_1$ and semiclassical measures supported in some translate of~$\T_-+ \T_1$.
\end{enumerate}
Note that in both cases the conclusion of Theorem~\ref{t:measures} is not sharp. However, we cannot say more on the support of the semiclassical measures for $A$ without further information on the action of $A$ on~$\T_1$. 

To illustrate this fact, take a matrix $B \in \Sp(2n,\mathbb{Z})$ that satisfies~\eqref{e:spectral-gap} and a matrix $C \in \Sp(2n',\mathbb{Z})$ whose eigenvalues are dominated by the leading eigenvalue of $B$. Assume in addition that in the decomposition from Lemma~\ref{l:decomposition_symplectic} for the matrix $B$, the factor $V_1$ is trivial. Then, we form the matrix
$$
A := B\oplus C\ \in\ \Sp(2(n+n'),\mathbb Z).
$$
Notice that the matrix $A$ satisfies the condition \eqref{e:spectral-gap} and that the action of $A$ on the spaces $V_0$ and $V_1$ is given respectively by the matrices $B$ and $C$. The quantizations $M_{\mathbf N,\theta}$ of $A$ are tensor products of quantizations
of $B$ with quantizations of~$C$, with a basis of eigenfunctions consisting
of tensor products of eigenfunctions corresponding to $B$ with those corresponding to $C$.
The torus decomposed as $\T^{2(n+n')} = \T_0 \times \T_1$ and the semiclassical measures for $A$ associated to eigenfunctions of product type are of the form $\mu = \mu_0 \times \mu_1$ where $\mu_0$ and $\mu_1$ are semiclassical measures respectively for $B$ and~$C$. We know that the support of $\mu_0$ must contain a translate of $\T_+$ or $\T_-$ (and this estimate cannot be improved as discussed in \S \ref{s:most_favorable}).

If the semiclassical measures for $C$ have large supports, then Theorem~\ref{t:measures} is not sharp. For instance, if $C$ satisfies~\eqref{e:spectral-gap} and has an irreducible characteristic polynomial over the rationals, then $\mu_1$ must be fully supported, so that the supports of the semiclassical measures for $A$ (associated to eigenfunctions
of product type) contain a translate of $\T_+ \times \T_1$ or $\T_- \times \T_1$.

However, it is not true in general that the semiclassical measures for $C$ have a large support. The most extreme case is when $C$ is given by the symplectic rotation matrix~$F$ from~\eqref{e:special-sp}. In that case, the Dirac mass at $0$ is a semiclassical measure for~$C$ (as proved in Lemma~\ref{l:dirac_zero} below). Hence, all the measures of the form $\mu_0 \times \delta_0$ are semiclassical measures for $A$, and we see that Theorem~\ref{t:measures} is sharp in that case. A concrete example of a matrix $A$ for which this happens is
\begin{equation}
  \label{e:A-example-2}
A=\begin{pmatrix}2 & 3 \\ 1 & 2\end{pmatrix} \oplus
\begin{pmatrix} 0 & 1 \\ -1 & 0\end{pmatrix}=
\begin{pmatrix} 2 & 0 & 3 & 0 \\
0 & 0 & 0 & 1 \\
1 & 0 & 2 & 0 \\
0 & -1 & 0 & 0
\end{pmatrix}.
\end{equation}
We end this section with an example of matrix with the Dirac mass at $0$ as a semiclassical measure, that was needed for our discussion above.
\begin{lemm}\label{l:dirac_zero}
Let $F$ be the symplectic matrix from \eqref{e:special-sp}. Then the Dirac mass at $0$ is a semiclassical measure associated to $F$.
\end{lemm}
\begin{proof}
Note that $\varphi_F=0$, so the quantization condition~\eqref{e:theta-quantize-condition}
holds for $\theta=0$ and all~$N$.
We will construct an eigenvector for the quantizations of $F$ using a Gaussian function localized near $0$ in phase space. We start with the function
$$
f(x) = e^{- \frac{|x|^2}{2h}}\in L^2(\mathbb R^n),\quad \mathcal F_h f=f
$$
where $\mathcal F_h\in \mathcal M_F$ is the semiclassical Fourier transform on $L^2(\mathbb R^n)$ defined in~\eqref{e:F-h-def-2}
and a quantization of~$F$ on $\mathcal H_{\mathbf N}(0)$ is given by
$\mathcal F_h|_{\mathcal H_{\mathbf N}(0)}$.
Using the projector $\Pi_{\mathbf N}$ from~\S\ref{s:decomposing-L2}, define the state
\begin{equation*}
f_{\mathbf{N}} \coloneqq \Pi_{\mathbf{N}}(0) f = \sum_{j \in \mathbb{Z}_{\mathbf{N}}^n} f_{\mathbf N,j} \mathbf{e}_j^0\ \in\ \mathcal{H}_{\mathbf{N}}(0),
\end{equation*}
where $f_{\mathbf{N},j} = \langle f, \mathbf{e}_j^0 \rangle_{L^2}$ for $j \in \mathbb{Z}_{\mathbf{N}}^n$. More explicitly, recalling~\eqref{e:e-kappa-def} we have
\begin{equation}
\label{e:expression_coordinates}
f_{\mathbf{N},j} = \mathbf{N}^{- \frac{n}{2}} \sum_{k \in \mathbb{Z}^n} e^{- \pi \mathbf{N}\left|k + \frac{j}{\mathbf{N}}\right|^2}.
\end{equation}
Since $\mathcal F_hf=f$, it follows from Lemma \ref{l:L2-decomposed} and the intertwining relation \eqref{e:intertwining} that $f_{\mathbf{N}}$ is an eigenvector for the quantizations of $F$ on $\mathcal{H}_{\mathbf{N}}(0)$.
By a diagonal argument (similarly to~\cite[Theorem~5.2]{Zworski-Book})
there is a sequence of even numbers $\mathbf{N}_p \underset{p \to \infty}{\to} \infty$ and a semiclassical measure $\mu$ for $F$ such that, for every $a \in C^{\infty}(\mathbb{T}^{2n})$,
\begin{equation}\label{e:convergence_to_mu}
\frac{\langle \Op_{\mathbf{N}_p,0}(a) f_{\mathbf{N}_p}, f_{\mathbf{N}_p} \rangle_{\mathcal{H}}}{\n{f_{\mathbf N_p}}^2_{\mathcal{H}}} \underset{ p \to  \infty}{\to} \int_{\mathbb{T}^{2n}} a\,d\mu.
\end{equation}
We will show that $\mu$ is the delta measure at~$(0,0)$, which
(by the diagonal argument again and since the limit of every convergent subsequence is the same)
implies that the convergence statement~\eqref{e:convergence_to_mu}
holds for the entire sequence $f_{\mathbf N}$.
Let us prove first that $\mu$ is supported in $\{x=0\}$.
Let $a(x)\in C^\infty(\mathbb T^n)$ be such that the ball centered at~0 of some small radius $\varepsilon>0$ does not intersect $\supp a$. By~\eqref{e:Op-a-special},
we have
\begin{equation}\label{e:explicit_pseudors_specific}
\langle \Op_{\mathbf{N},0}(a) f_{\mathbf{N}}, f_{\mathbf{N}} \rangle_{\mathcal{H}} =
\sum_{j\in\mathbb Z_{\mathbf N}^n}a\Big({j\over \mathbf N}\Big) |f_{\mathbf N,j}|^2.
\end{equation}
From~\eqref{e:expression_coordinates} we get
$|f_{\mathbf{N},j}|^2 \leq C \mathbf N^{-n}e^{- 2\pi \varepsilon^2 \mathbf N}$
for all $j\in \mathbb Z_{\mathbf N}^n$ such that $j/\mathbf N\in\supp a$.
On the other hand $\|f_{\mathbf N}\|_{\mathcal H}\geq |f_{\mathbf N,0}|\geq \mathbf N^{-{n\over 2}}$. Thus
$$
\int_{\mathbb T^{2n}}a(x)\,d\mu=\lim_{\mathbf N\to\infty}{\langle\Op_{\mathbf N,0}(a)f_{\mathbf N},f_{\mathbf N}\rangle_{\mathcal H}\over \|f_{\mathbf N}\|_{\mathcal H}^2}=0
$$
for all $a\in \CIc(\mathbb T^n\setminus \{0\})$, which gives that
$\supp\mu\subset \{x=0\}$.
Since $\mu$ is a semiclassical measure associated to $F$, it is invariant under~$F$.
Thus $\supp\mu\subset F(\{x=0\})=\{\xi=0\}$. It follows that
$\supp\mu=\{(0,0)\}$. Since $\mu$ is a probability measure, it has to be the delta measure
at~$(0,0)$.
\end{proof}

\medskip\noindent\textbf{Acknowledgements.}
We would like to thank the anonymous referee for a careful reading of the paper and many useful comments.
SD was supported by NSF CAREER grant DMS-1749858
and a Sloan Research Fellowship.
Most of this work was done while MJ was supported by the European Research Council (ERC) under the European Union's Horizon 2020 research and innovation programme (grant agreement No 787304), and working at LPSM\footnote{Laboratoire de Probabilit\'es, Statistique et Mod\'elisation (LPSM), CNRS, Sorbonne Universit\'e, Universit\'e de Paris, 4, Place Jussieu, 75005 Paris, France}.

\bibliographystyle{alpha}
\bibliography{General,Dyatlov,QC,Scattering}

\end{document}